\documentclass[reqno,12pt]{amsart}
\usepackage[margin={2cm,2.5cm}]{geometry}
\usepackage{amssymb}
\usepackage{amsmath}
\usepackage{amsthm}
\usepackage{wasysym}
\usepackage{mathtools}
\usepackage{color}
\usepackage[svgnames]{xcolor}
\usepackage[pdftex]{hyperref}
\usepackage{enumitem}
\usepackage{booktabs}
\usepackage[nospace]{cite}
\usepackage{soul} 
\usepackage{tikz}
\usepackage{tkz-tab}
\usepackage{tkz-fct}  
\usetikzlibrary{calc,intersections}
\definecolor{color1}{RGB}{27,158,119}
\definecolor{color2}{RGB}{217,95,2}
\definecolor{color3}{RGB}{117,112,179}
\definecolor{color4}{RGB}{231,41,138}
\usepackage{pgfplots}
\pgfplotsset{compat=1.9}

\hypersetup{
  pdftitle={},
  pdfauthor={Fran\c cois Genoud, Stefan Le Coz  and Julien Royer},
  pdfsubject={},
  pdfkeywords={}, 
  colorlinks=true,
  linkcolor=DarkBlue  ,          %
  citecolor=DarkRed,        %
  filecolor=DarkMagenta,      %
  urlcolor=DarkGreen,           %
}

\DeclarePairedDelimiter{\norm}{\lVert}{\rVert}
\DeclarePairedDelimiter{\abs}{\lvert}{\rvert}

\newcommand{\scalar}[2]{\left( #1,#2 \right)}

\newcommand{\dual}[2]{\left\langle #1,#2 \right\rangle}

\newcommand{\eps}{\varepsilon}

\newcommand{\N}{\mathbb{N}}
\newcommand{\R}{\mathbb{R}}
\newcommand{\C}{\mathbb{C}}

\DeclareMathOperator{\sech}{sech}
\DeclareMathOperator{\Span}{span}

\renewcommand{\leq}{\leqslant}
\renewcommand{\geq}{\geqslant}

\usepackage{color}

\newcommand{\ep}{\epsilon}

\newcommand{\la}{\langle}
\newcommand{\ra}{\rangle}
\newcommand{\wt}{\widetilde}
\newcommand{\diff}{\,\mathrm{d}}
\newcommand{\dif}{\mathrm{d}}

\newcommand{\und}{\underline}
\newcommand{\intg}{\int_{\Gc}}

\let\le=\leqslant
\let\ge= \geqslant

\newcommand{\disp}{\displaystyle}
\newcommand\txt{\textstyle}

\DeclareMathOperator \rge{ran}

\DeclareMathOperator \im{Im}
\DeclareMathOperator \re{Re}

\newcommand{\Gc}{\mathcal G}
\newcommand{\OC}{\mathcal O_{\kappa}(b,\lambda,\alpha)}
\newcommand{\Hrad}{H^1_{\mathrm{rad}}(\mathcal G)}
\newcommand{\Uc}{\mathcal U} \newcommand{\Vc}{\mathcal V}
\newcommand{\f}{\varphi}

\theoremstyle{plain}
\newtheorem{theorem}{Theorem}[section]
\newtheorem{proposition}[theorem]{Proposition}
\newtheorem{corollary}[theorem]{Corollary}
\newtheorem{lemma}[theorem]{Lemma}

\theoremstyle{definition}

\theoremstyle{remark}
\newtheorem{remark}[theorem]{Remark}

\numberwithin{equation}{section}

\newcommand{\Modop}{\mathcal M_0}
\newcommand{\tModop}{\mathcal M_1}
\newcommand{\Mod}{\mathrm{Mod}}

\begin{document}

\title[Blow-up solutions on nonlinear quantum star graphs]
{A minimal mass blow-up solution on \\ a nonlinear quantum star graph}

\author[F.~Genoud]{Fran\c cois Genoud}

\author[S.~Le Coz]{Stefan Le Coz}
\thanks{This work was supported by the ANR LabEx CIMI (grant ANR-11-LABX-0040) within the French State Programme ``Investissements d'Avenir'' and by the ANR project NQG (grant ANR-23-CE40-0005-01).
The authors are grateful to the Bernoulli Center at EPFL, where decisive discussions were held in the final stage of this work.}

\author[J.~Royer]{Julien Royer}

\address[Fran\c cois Genoud]{Ecole Polytechnique F\'ed\'erale de Lausanne,
\newline\indent
EPFL Station 4
\newline\indent
1015 Lausanne
\newline\indent
Switzerland}
\email{francois.genoud@epfl.ch}

\address[Stefan Le Coz and Julien Royer]{Institut de Math\'ematiques de Toulouse,
  \newline\indent
  Universit\'e Paul Sabatier
  \newline\indent
  118 route de Narbonne, 31062 Toulouse Cedex 9
  \newline\indent
  France}
\email[Stefan Le Coz]{stefan.le-coz@math.univ-toulouse.fr}
\email[Julien Royer]{julien.royer@math.univ-toulouse.fr}

\subjclass[2010]{}

\date{\today}

\begin{abstract}
We construct a finite-time blow-up solution to the mass-critical focusing nonlinear Schr\"odinger equation on a metric star graph with an arbitrary number of edges. We show that all solutions are global if their mass is smaller than
an explicit constant, called ``minimal mass''.
We then construct a solution with minimal mass and arbitrary energy, which blows up in finite time at the vertex of the star graph.
The blow-up profile and blow-up speed are explicitly characterized.  
The main novelty of the paper is the construction of the blow-up profile in time-dependent domains of singularly perturbed Laplacians.
\end{abstract}

\maketitle

\tableofcontents

\section{Introduction}

\subsection{Setting and main result}

We consider a metric star graph $\Gc$ with $N$ edges of infinite length, as illustrated in
Figure~\ref{fig:star_graph}. 
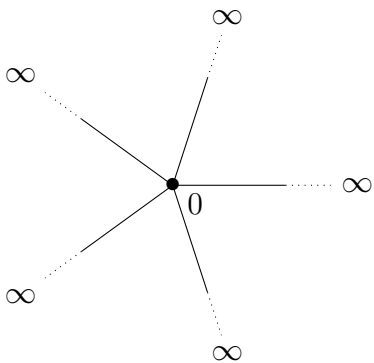
\begin{figure}[htbp!]
  \centering
\begin{tikzpicture}[scale=0.030]
  \node at (0,0) {$\bullet$};
  \node at (10,-8) {$0$};
    \foreach \i in {1,2,3,4,5}{
  \draw[rotate=72*\i] (0,0) -- (50,0) node[anchor={180+72*\i}] {};
    \draw[rotate=72*\i,dotted] (50,0) -- (70,0) node[anchor={180+72*\i}] {$\infty$};
  }
  \end{tikzpicture}    
  \caption{A star graph with $5$ edges.}
  \label{fig:star_graph}
\end{figure}

Let $\gamma \in \R$. We consider on $\Gc$ the focusing nonlinear Schr\"odinger equation
\begin{equation}
  \label{eq:nls}
  i u_t + u_{xx} - \gamma \delta u +| u|^4 u=0,
\end{equation}
where $u_{xx}$ denotes the Laplacian of $u$ on each edge,
$\delta$ is the Dirac mass at the vertex and $\gamma \in \R$. 
The main conserved quantities associated with this nonlinear evolution equation are the energy 
\begin{equation}
\label{eq:def_E}
E(u):=\frac12 \norm{u_{ x}}_{L^2(\mathcal G)}^2+\frac\gamma 2| u(0)|^2-\frac16\| u\|_{L^6(\mathcal G)}^6,
\end{equation}
where $u(0)$ is the common value of each component of $u$ at 0, and the mass
\begin{equation}
\label{eq:def_M}
M(u):=\frac12\| u\|_{L^2(\mathcal G)}^2.
\end{equation}

The exact description of $\mathcal G$, the precise interpretation of~\eqref{eq:nls},
and the definition of the usual function spaces on $\mathcal G$ will be given in Section~\ref{sec:star_graphs}.

Our main objective in this paper is to construct a finite-time blow-up solution of~\eqref{eq:nls}. 
Let ${\mathsf Q}:\R\to\R$ be defined by
\begin{equation} \label{def:Q}
\mathsf Q(x)=3^{\frac14}\sech^{\frac12}(2x).
\end{equation}
The function ${\mathsf Q}$ is the positive even ground state of 
the focusing mass-critical nonlinear Schr\"odinger equation on the line. It is well known that the mass of $\mathsf Q$, given by
$$
M_{\mathsf Q}:=\frac12\int_\R \mathsf Q^2 \diff x = \frac {\pi\sqrt{3}}4
$$ 
gives the 
threshold between global existence and finite time blow-up for the associated Cauchy problem. 
It turns out that $M_{\mathsf Q}$ also determines the mass threshold for global existence of solutions of~\eqref{eq:nls} on the graph. 
Indeed, we shall prove in Section~\ref{sec:global} that, if the initial condition
$u_0$ satisfies
$$
M(u_0)<\min\left\{\txt\frac N2,1\right\} M_{\mathsf Q},
$$ 
then the corresponding solution of~\eqref{eq:nls} is global.

A function on $\Gc$ is called {\em radial} 
if all its components on the edges of the graph are equal (see Section~\ref{sec:star_graphs}).
Let $Q$ be the radial 
function on $\mathcal G$ consisting 
of a half-copy of $\mathsf Q$ on each edge.
If we restrict ourselves to radial solutions, then the threshold for global existence of solutions of~\eqref{eq:nls} becomes
\begin{equation}\label{eq:min_mass}
M( Q)=\txt\frac N2 M_{\mathsf Q}.
\end{equation}
 
In the attractive case $\gamma<0$,
we construct a radial minimal mass blow-up solution, that is,
a solution of~\eqref{eq:nls} with mass~\eqref{eq:min_mass} which blows up in finite time at the vertex. 
More precisely, our main result is the following theorem. 

\begin{theorem} \label{th:blowup}
Suppose $\gamma<0$. Let $E^\star\in\mathbb R$. There exist $t_0<0$ and a radial
solution $u\in C([t_0,0),H^1(\mathcal G))$ of~\eqref{eq:nls}  such that 
$$
M(u)=M(Q), \quad E(u)=E^\star,
$$
and which blows up at $t=0$ as
\begin{equation}\label{blowup_speed}
\|u_{ x}(t)\|_{L^2(\mathcal G)}   \mathop{\mbox{$\sim$}}\limits_{t \to 0^-} \frac {C_\gamma}{|t|^{2/3}},
\end{equation}
for an explicit constant $C_\gamma>0$.%
\end{theorem}

To the best of our knowledge, the present work is the first construction of a finite time blow-up solution for a nonlinear Schr\"odinger equation on a quantum graph.

\subsection{Mass-critical NLS on the line}

We recall some well-known facts for the classical mass-critical nonlinear 
Schr\"odinger equation on the line
\begin{equation}
  \label{eq:1}
  iu_t+u_{xx}+|u|^4u=0,
\end{equation}
where $u:\mathbb R_t\times \mathbb R_x\to\mathbb C$.
The  Cauchy problem for~\eqref{eq:1} is well-posed in the energy space $H^1(\mathbb R)$, we have conservation of energy, mass (and momentum) and the blow-up alternative holds. Of particular interest is the standing wave solution
\(
e^{it} \mathsf Q(x),
\)
where the profile $\mathsf Q:\mathbb R\to \mathbb R$ (already defined in~\eqref{def:Q}) is the unique even positive solution in $H^1(\mathbb R)$ of the differential equation
\begin{equation} \label{eq:Q}
-\mathsf Q''+ \mathsf Q- \mathsf Q^5=0.
\end{equation}
We already mentioned that the mass of $\mathsf Q$ gives the threshold between global existence and blow-up. Precisely, any solution of~\eqref{eq:1} with mass smaller that $M_{\mathsf Q}$ is global, whereas there exists a minimal mass blow-up solution, i.e.~a solution of~\eqref{eq:1} with mass $M_{\mathsf Q}$ which blows up in finite time. It turns out that such a solution can be found as an explicit pseudo-conformal transform of the standing wave. Indeed, let
\begin{equation}\label{eq:PC}
S(t,x)=\frac{1}{\sqrt{|t|}} \mathsf Q\Big(\frac{x}{|t|}\Big)e^{-i\frac{|x|^2}{4|t|}}e^{\frac it}.
\end{equation}
Then $S$ is a solution of~\eqref{eq:1} and we have
\begin{equation}\label{eq:PC_speed}
\norm{S(t)}_{L^2(\R)}=\norm{\mathsf Q}_{L^2(\R)},
\quad \norm{\partial_x S(t)}_{L^2(\R)}\mathop{\mbox{$\sim$}}\limits_{t \to 0^-}\frac{C}{|t|}.
\end{equation}
In particular, $S$ blows up at $t=0$ with the so-called \emph{pseudo-conformal speed} 
$|t|^{-1}$. Furthermore, up to the symmetries of the equation, 
$S$ is the unique minimal mass blow-up solution (see~\cite{Me93}).

On the graph $\Gc$ with $\gamma=0$, the function constructed by considering a half-copy of $S(t)$ on each edge is a solution of~\eqref{eq:nls}, which blows up at time $T=0$ at the central vertex with pseudo-conformal speed~\eqref{eq:PC_speed}.
When $\gamma<0$, one cannot construct a simple solution based on $S(t)$, but the proof of
Theorem~\ref{th:blowup} (see, in particular, Proposition~\ref{prop:profile}) suggests that radial minimal mass blow-up solutions should also be governed by the profile $\mathsf Q$ in this case.
In the repulsive case $\gamma>0$, there are no radial minimal mass solutions 
blowing up in finite time at the vertex; see Section~\ref{sec:global}.

\subsection{Minimal mass blow-up solutions}

There exists an important literature about the construction of minimal mass 
blow-up solutions in various settings.
For the classical pure power mass-critical nonlinear Schr\"odinger equation on $\R^d$, 
a minimal mass blow-up solution is explicitly obtained as a pseudo-conformal transform 
of a standing wave, in any dimension, similarly to~\eqref{eq:PC} for $d=1$. 
In the seminal paper~\cite{Me93}, Merle showed that it is the unique 
minimal mass blow-up solution up to the symmetries of the equation.  
Existence and uniqueness of a minimal mass blow-up solution for NLS equations which do not 
possess the pseudo-conformal symmetry is more involved.
The study was initiated by Merle himself in~\cite{Me96}, where he established 
a sufficient condition for the existence of a minimal mass blow-up solution in the case 
of a Schr\"odinger equation with inhomogeneous mass-critical nonlinearity $k(x)|u|^{\frac4d}u$. 
Further contributions (see e.g.~Banica, Carles, Duyckaerts~\cite{BaCaDu11}, 
Bourgain and Wang~\cite{BoWa97}, Krieger and Schlag~\cite{KrSc09}) 
treated the problem perturbatively from the homogeneous case, and required a 
flatness assumption on~$k$. A nonperturbative approach was called for in order 
to remove the flatness assumption.
The breakthrough came from the work of Rapha\"el and Szeftel~\cite{RaSz11}, 
in which existence and uniqueness of a minimal mass blow-up solution for the 
inhomogeneous mass-critical nonlinearity was established. 
The approach of~\cite{RaSz11} is very robust and was applied for instance 
by Krieger, Lenzmann and Rapha\"el~\cite{KrLeRa13} to the critical half-wave equation, 
or by Martel and Pilod~\cite{MaPi17} to the Benjamin-Ono equation.
The construction of the profile of the minimal mass blow-up solution 
was later refined by Le Coz, Martel, Rapha\"el~\cite{LeMaRa16} in the context of 
the nonlinear Schr\"odinger equation with a double power nonlinearity, 
where a minimal mass solution exhibiting a new blow-up speed was constructed. 
The approach of~\cite{LeMaRa16,RaSz11} was successfully implemented by Matsui~\cite{Ma20a,Ma20b,Ma21a,Ma21b,Ma21c,Ma21d,Ma21e,Ma23} for various Schr\"odinger equations 
(e.g.~with singular potentials or with a Hartree nonlinearity). 
Several improvements to the work~\cite{LeMaRa16} have been made by Matsui, 
in particular the observation that the blow-up profile is more naturally constructed 
in the virial space instead of $H^1$. Recently, the paper~\cite{LeMaRa16} 
was transposed by Tang and Xu~\cite{TaXu21} to 
the nonlinear Schr\"odinger equation on the line with a Dirac mass at the origin.

\subsection{Star graphs}

On the other hand, there is also a wide literature on nonlinear quantum graphs which
cannot be shortly summarized. 
For an introduction to nonlinear Schr\"odinger equations on quantum graphs 
and their physical motivations, one may refer to the survey of Noja~\cite{No14}. 
For star graphs in particular, one may refer to the recent monograph of Angulo Pava 
and Cavalcante de Melo~\cite{AnCa19}. In this introduction, 
we will only present the results close to our work,
along with a very partial sample of the rest of the literature. 
Many of the works devoted to nonlinear quantum graphs focus on
existence and variational characterizations of standing waves. 
Among the earliest studies, one finds the works by Fukuizumi in collaboration 
with (separately) Jeanjean, Le Coz, Ohta and Ozawa~\cite{FuJe08,FuOhOz08,LeFuFiKsSi08}, 
which are devoted to the case of a line with a Dirac mass at the origin 
(equivalent to a $2$-star graph). The first author of the present paper, 
together with Malomed and Weissh\"aupl~\cite{GeMaWe16},
studied orbital stability of standing waves for the $2$-star graph with a cubic-quintic
nonlinearity.
The variational characterization of standing waves on star graphs was considered by 
Adami, Cacciapuoti, Finco and Noja 
\cite{AdCaFiNo12,adami2012stationary,AdCaFiNo14a, AdCaFiNo14, AdCaFiNo16}. 
Further developments for the study of standing waves on generic quantum graphs 
started with Adami, Serra and Tilli~\cite{AdSeTi16, AdSeTi17a, AdSeTi17b}, 
where a topological obstruction for the existence of ground states 
on quantum graphs was discovered.
Elements such as well-posedness of the Cauchy problem, Strichartz estimates  
and conservation laws on star graphs can be found in the work of 
Adami, Cacciapuoti, Finco and Noja~\cite{AdCaFiNo11} (along with the analysis of the collision of a fast solitary wave with the vertex, which is the main object of the paper). 
The $2$-star graph with non-zero boundary conditions has been investigated 
by Ianni, Le Coz and Royer~\cite{IaLeRo17}. 
The case of a loop (which is equivalent to a segment with periodic boundary conditions) 
was studied by Gustafson, Le Coz and Tsai~\cite{GuLeTs17}. 
Absence of scattering of global solutions towards standing waves was established by 
Aoki, Inui, Mizutani~\cite{AoInMi21}, while scattering on the $2$-star graph was obtained
by Banica and Visciglia~\cite{BaVi16}.
Exponential stability in the presence of damping on one branch was obtained by 
Ammari, Bchatnia and Mehenaoui~\cite{AmBcMe21}. 
Existence of ground states on star graphs with finite and infinite egdes was studied by 
Li, Li and Shi~\cite{LiLiSh18}.
On balanced star graphs (i.e.~star graphs with adjusted coefficients on the edges, see~\cite{SoMaSaSaNa10}), Kairzhan, Pelinovsky and Goodman~\cite{KaPeGo19} proved the nonlinear instability (by drift) of spectrally stable shifted states. 
Standing waves of the nonlinear Schr\"odinger equation with logarithmic nonlinearity 
was considered by Goloshchapova~\cite{Gol19} 
(see also the earlier work of Ardila~\cite{Ar17} for well-posedness and existence results). 
Instability of non-ground state standing waves on star graphs was obtained by 
Kairzhan~\cite{Ka19} in the repulsive and attractive cases. 
Instability by blow-up of standing waves on star graphs for 
mass-supercritical nonlinearities was proved by Goloshchapova and Ohta~\cite{GoOh20}.   
Stability and instability results were obtained by Angulo Pava and Goloshchapova~\cite{AnGo18a,AnGo18b} using the extension theory of symmetric operators for star graphs 
with $\delta$ or $\delta'$ interaction at the vertex. 
Star graphs with $\delta_s'$ conditions were considered by Goloshchapova in~\cite{Go22}.
Recently, Besse, Duboscq and Le Coz~\cite{BeDuLe22A,BeDuLe22B} developed a Python Library~\cite{Grafidi} for the numerical simulation of Schr\"odinger equations on quantum graphs. A numerical approach for the calculation of ground states is studied in~\cite{BeDuLe22A}, while the implementation of the library and further experiments are presented in~\cite{BeDuLe22B}.

\subsection{Main novelty of our construction}\label{sec:novelty}

Our proof of Theorem~\ref{th:blowup} follows the 
strategy laid down in~\cite{LeMaRa16,RaSz11} with some improvements obtained by 
Matsui~\cite{Ma20a,Ma20b,Ma21a,Ma21b,Ma21c,Ma21d,Ma21e,Ma23}. 
In particular, we shall work directly in the virial space, which provides a natural setting 
and allows us to avoid the 
localization procedure of the virial-energy functional used in~\cite{LeMaRa16}. 
We have also reformulated the blow-up profile expansion borrowed from~\cite{LeMaRa16}, 
thereby making it more tractable for the proof. In the particular case of the $2$-star graph, 
we recover the result stated in~\cite{TaXu21} for the mass-critical
NLS on the line with an attractive Dirac mass. In fact, as explained in more details
below, some of the arguments exposed in~\cite{TaXu21} are not completely rigorous,
thus the present work also provides the first exact treatment of the problem on the line.

One of the main features of this article is a rigorous treatment of the Dirac 
distribution $\delta$ (also called ``Dirac mass'' here). 
The formal differential operator $\partial_{xx}-\gamma\delta$ appearing
in~\eqref{eq:nls} will be given a precise definition in Section~\ref{sec:star_graphs}
as a Laplace operator perturbed by a boundary condition at the vertex of the graph.
As will be seen in Section~\ref{sec:outline}, we will thus need to adapt the method from 
\cite{LeMaRa16,RaSz11} in order to construct an approximate blow-up solution which remains
inside the domain of this operator at all times. In rescaled variables
suitable to our construction, the domain itself depends on time, thus all coefficients
appearing in the construction of the approximate blow-up profile (see~\eqref{def:P}) also
depend on time. Aside from the technical difficulties arising from this, it is a noteworthy 
realization of our work that the method from~\cite{LeMaRa16,RaSz11} is robust enough 
to tackle equations for which the domain of the linear operator plays a crucial role in the analysis.

Regarding equations like~\eqref{eq:nls}, the distribution $\delta$
is often referred to as ``delta potential'' in the literature.
We find this terminology unfortunate.
Indeed, it is too often the case that $\delta$ is treated as if it were a classical point-wise defined potential, with the misleading notation $\delta(x)$ which doesn't really
make any sense, since $\delta$ acts on functions and does not have point values. 
Of course, it is well known that no function can represent the distribution $\delta$. 

In fact, while the statement of the main result of~\cite{TaXu21} certainly holds true (and we recover the same result in the case of the $2$-star graph), 
it becomes apparent upon careful examination that the proof provided in~\cite{TaXu21} lacks the necessary precision to fully substantiate this result.
This is mainly due to formal manipulations handling $\delta$ as a function, not a distribution.

For instance, the common confusion between the distribution $\delta$ and a point-wise defined potential is reinforced by the notation $g(u)=\mu \delta u$, introduced by the authors at the bottom of page 1732. There is indeed a crucial difference between $f(u)=|u|^4u$, which takes a function $u$ and returns a power of this function, therefore also a function (precisely, $f:H^1(\mathbb R)\to H^1(\mathbb R)$) and $g(u)=\mu \delta u$, which takes a function but returns a \emph{distribution} (precisely, $g:H^1(\mathbb R)\to \mathcal D'(\mathbb R)$). 
Therefore, $f(u)(x)$ makes sense while $g(u)(x)$ is meaningless. The two terms are, however, treated most of the time on the same level in~\cite{TaXu21}.

An example of the confusion generated by this inexact notation appears in the definition of $\beta$
in equation (2.16) of \cite{TaXu21}. Indeed, the authors have defined $G(u)=\frac12\mu\delta |u|^2$  
on page 1732 (hence $G(u)\in H^{-1}(\mathbb R)$) 
but write $G(Q)=\frac12 Q(0)^2$ in (2.16) (as though $G(Q)\in\mathbb R$).
Another example is the estimate of the remainder $\Psi_K$ defined in \cite[(2.7)]{TaXu21}
and its supposed point-wise derivative $\partial_y\Psi_K$. Indeed, $\Psi_K$ contains $\delta$ terms (appearing in the $F_{j,k}^\pm$). As such, $\Psi_K$ does not have point-values, let alone a point-wise  derivative.
Yet another example occurs in \cite[(2.26)]{TaXu21}. As already mentioned, $\Psi_K$ contains $\delta$ terms. In addition, $\tilde E$ contains the term $\int_\R G(P_b)\diff y$, which can be interpreted as $\mu|P_b(0)|^2$. Hence $\tilde E'$ contains a $\delta$ corresponding to this term. The term $\langle i\tilde E'(\lambda,P_b),\Psi_Ke^{-i{b|y|^2}/{4}}\rangle$ in~\cite[(2.26)]{TaXu21} will therefore contain a term of the type $\dual{\delta}{\delta}$, which is not a meaningful expression.

\subsection{Organization of the paper and notation}
The rest of the paper is organized as follows. Section~\ref{sec:outline} provides a detailed outline
of the construction of our blow-up solution. The proof of Theorem~\ref{th:blowup} is given there,
assuming a number of propositions. Section~\ref{sec:cauchy} presents 
the Cauchy theory for~\eqref{eq:nls} in the spaces relevant for our analysis. 
We give in Section~\ref{generalized_kernel.sec} some detailed properties of linearized operators.
In Sections~\ref{sec:profile} to~\ref{unifest_proof.sec}, 
the propositions used in the proof of Theorem~\ref{th:blowup} are proved. Appendix~\ref{sec:dynamical_systems}, devoted to the model dynamical system, closes the paper.

We shall write $f\lesssim g$ or $g\gtrsim f$ to mean that there is a universal constant $C>0$
(i.e.~which does not depend on the dynamical variables) such that $f\le Cg$. We will
write $f \sim g$ as $t\to 0^-$ (or $s\to+\infty$) if $f/g\to 1$ as $t\to 0^-$ (or $s\to+\infty$).
When no confusion is possible, we may simply write $L^2, H^1$, etc. instead of
$L^2(\mathcal{G}), H^1(\mathcal{G})$, etc. The inner product  on $L^2(\mathcal G)$ will be denoted by $\scalar{\cdot}{\cdot}_{L^2}$ 
or simply $\scalar{\cdot}{\cdot}$. The duality product between $H^1(\mathcal{G})$ and $H^1(\mathcal{G})^\star$ will be denoted by $\dual{\cdot}{\cdot}$.

\section{Outline of the proof}\label{sec:outline}

In this section we introduce the required functional setting and we prove Theorem~\ref{th:blowup} 
using a number of auxiliary results, which are proved in the following sections.

\subsection{Functional setting on the star graph}\label{sec:star_graphs}

Let $\mathcal G$ be a \emph{metric star graph} with $N$ edges, 
i.e.~a vertex $0$ to which are connected $N$ edges $e_1,\dots,e_N$ of infinite length. 
We thus identify each edge $e$ with the interval $I_e=\R_+:=[0,\infty)$, 
the left endpoint $0$ corresponding to the vertex. 
A schematic representation of a star graph is given in Figure~\ref{fig:star_graph}. 

A function $  u:\mathcal G\to\mathbb C$ is a collection of functions 
$u_j:I_{e_j}\to\mathbb C$, $j=1,\dots,N$. 
It will be called
{\em radial} if all its components $u_j:\R_+\to\mathbb C$ are identical. 
In this case, we will identify $  u$ with any one 
of its components.
Hence, $  x\in \mathcal G$ will be identified with $x\in\R_+$, and we will simply interpret $  u$ as a complex-valued function of $x\in\R_+$. 

Lebesgue and Sobolev spaces on $\mathcal G$ are defined by
\[
L^p(\mathcal G)=\bigoplus_{j=1}^N L^p(I_{e_j}),\quad H^s(\mathcal G)=\bigoplus_{j=1}^N H^s(I_{e_j}),
\]
with norms
$$
\|  u\|_{L^p(\mathcal G)}^p=\sum_{j=1}^N \|u_j\|_{L^p(\R_+)}^p, \quad 
\|  u\|_{H^s(\mathcal G)}^2=\sum_{j=1}^N \|u_j\|_{H^s(\R_+)}^2.
$$

We equip $\mathcal G$ with the Laplace operator with Dirac condition at the vertex, i.e.~the selfadjoint operator $H_\beta$ defined by
\[
  \begin{aligned}
    H_\beta:D(H_\beta)\subset L^2(\mathcal G)&\to L^2(\mathcal G),\\
    (u_1,\dots,u_N)&\mapsto (-\partial_{xx}u_1,\dots,-\partial_{xx}u_N),
  \end{aligned}
\]
where the domain $D(H_\beta)$ is denoted and defined by
\begin{equation} \label{dom:Hgamma}
  D_\beta \equiv D(H_\beta):=\left\{  u\in H^2(\mathcal G): \forall j,k=1,\dots,N,\,  u(  0):=u_j(0)=u_k(0),\;\sum_{j=1}^Nu_j'(0)=\beta   u(  0)\right\}.
\end{equation}
We use here a general parameter $\beta$ in order to cover the case $\beta=\gamma$ in 
the original equation~\eqref{eq:nls}, as well as $\beta=\gamma\lambda$ in the 
rescaled variables (see~\eqref{eq:nls-v}).

We observe that the domain contains a continuity condition at $0$ and a jump condition for the derivatives. For $\beta=0$ we recover the classical Kirchhoff-Neumann conditions. For $\beta\neq0$ and $N=2$, we recover the case of the line with a Dirac mass at $0$. 

The quadratic form associated with $H_\beta$ is
\begin{equation*}\label{quadratic_form}
q_\beta(  u):=\dual{H_\beta   u}{  u}=\sum_{j=1}^N\norm{u_j'}_{L^2(\R_+)}^2
+\beta |  u(  0)|^2,
\end{equation*}
defined on the domain  
\[
H^1_D(\mathcal G)\equiv D(q_\beta):=\left\{  u\in H^1(\mathcal G):\forall j,k=1,\dots,N,\,  u(  0):=u_j(0)=u_k(0)\right\}.
\]
Observe that the domain of the quadratic form retains the continuity at the vertex, but the jump condition on the derivatives is now transposed to the expression of the quadratic form instead of the domain. 

In this paper, we will mostly work in the subspace $\Hrad \subset H^1_D(\mathcal G)$ of radial functions,
defined by
\[
\Hrad=\left\{  u\in H^1(\mathcal G):\forall j,k=1,\dots,N,\,u_j=u_k\right\}.
\]
We also define
\[
L^2_{\mathrm{rad}}(\Gc) =
\left\{  u\in L^2(\mathcal G):\forall j,k=1,\dots,N,\,u_j=u_k\right\}.
\]
With the convention of identifying the function $u$ on $\Gc$ with any of its components, 
we have 
\[
\norm{u}_{L^2(\Gc)}^2 = N \norm{u}_{L^2(0,\infty)}^2, \quad \forall u\in L^2_{\mathrm{rad}}(\Gc).
\]

Note that the operator $H_\beta$ can also be extended to an operator $\mathsf H_\beta$  
from $H^1_D(\mathcal G)$ to its dual $H^1_D(\mathcal G)^\star$. 
More precisely, we may write
\[
\mathsf H_\beta=-\partial_{ {xx}}+\beta \delta
\]
where, for every $  u \in H^1_D(\mathcal G)$,
\[
\la -  u_{ {xx}},  v\ra:=\re\int_\mathcal{G}   u_{  x} \overline{  v_{  x}}\diff   x
\equiv \re\sum_{j=1}^N\int_0^\infty  \partial_x u_j \overline{\partial_x v_j}\diff x,
\quad \forall   v \in H^1_D(\mathcal G),
\]
\[
\la\delta   u,  v\ra
:=\re\big(   u(  0)\overline{  v(  0)}\big), \quad \forall   v \in H^1_D(\mathcal G). 
\]

This allows us to split the operator $\mathsf H_\beta$ into two parts, $-\partial_{ {xx}}$ and $\beta\delta$, whenever needed. We emphasize that whenever 
$-\partial_{ {xx}}$ and $\beta\delta$ are treated separately, they are always taken in the $H^1-(H^1)^\star$ sense (the operator $H_\beta$ as an $L^2-L^2$ operator with domain cannot be split).

Using the operator $H_\gamma$ defined above, equation~\eqref{eq:nls} can now be precisely
interpreted as
\begin{equation}
  \label{eq:nls-Hg}
  i u_t-H_\gamma u+|u|^4 u=0,
\end{equation}
for an unknown function $u:\R\times \Gc \to \C$.

\subsection{Cauchy problem}
We start by stating here the main results concerning the Cauchy problem for~\eqref{eq:nls-Hg}. The well-posedness in the space $H^1_D(\Gc)$ can be obtained by a classical line of arguments. For our purposes, we will need the solution to live in the domain of $H_\gamma$ as well as in weighted spaces. 
For $k \in \N$ we set
\begin{equation}
\label{eq:def_Sigma_k} 
\Sigma^k(\mathcal G)=\{  u\in H^k(\mathcal G):\norm{  u}_{\Sigma^k}<\infty\},\quad \norm{  u}_{\Sigma^k}^2=\sum_{0\leq \alpha,\beta\leq k}\left(\norm{  x^\alpha   u}_{L^2}^2+\norm{\partial_x^\beta   u}_{L^2}^2\right).
\end{equation}

The well-posedness results are summarized in the following proposition. 

\begin{proposition}
\label{prop:cauchy}
    Let $u_0\in H^1_D(\mathcal G)$ be an initial data for the problem~\eqref{eq:nls-Hg}. Let $t_0 \in \R$. Then there exists a unique maximal solution
\[
  u\in C\left((T_{\min},T_{\max}),H^1_D(\mathcal G)\right) \cap  C^1\left((T_{\min},T_{\max}),H^1_D(\mathcal G)^\star\right)
\]
(with $T_{\min} < t_0 < T_{\max}$) such that $u(t_0)=u_0$. The blow-up alternative holds and there is continuous dependence with respect to the initial data. The energy $E$ and the mass $M$, defined in~\eqref{eq:def_E} and~\eqref{eq:def_M}, are conserved along the time evolution. Moreover, the following properties hold.
\begin{enumerate}[label={\upshape(\roman*)}]
\item If $u_0$ is radial, then so  is $u(t)$ for any $t\in(-T_{\min},T_{\max})$.
\item If $u_0\in D_\gamma$, then $u$ verifies
\[
  u\in C\left((T_{\min},T_{\max}),D_\gamma\right) \cap C^1\left((T_{\min},T_{\max}),L^2(\mathcal G)\right).
\]
\item If $u_0\in H^1_D(\mathcal G)\cap\Sigma^1(\mathcal G)$, then $u$ verifies
\[
u\in C\left((T_{\min},T_{\max}),H^1_D(\mathcal G)\cap\Sigma^1(\mathcal G)\right)\cap  C^1\left((T_{\min},T_{\max}),H^1_D(\mathcal G)^\star\right).
\]
\item If $u_0\in D_\gamma\cap\Sigma^2(\mathcal G)$,
then $u$ verifies
\[
u\in C\left((T_{\min},T_{\max}),D_\gamma\cap\Sigma^2(\mathcal G)\right)\cap C^1\left((T_{\min},T_{\max}),L^2(\mathcal G)\right).
\]
\end{enumerate}
\end{proposition}

We will also need a dependency property with respect to the initial data slightly different from the usual one.

\begin{lemma} \label{lem:conv-L2-solution}
Let $(  u_{n,0})_{n \in \N^*}$ and $  u_0$ in $H_D^1(\mathcal G)$ such that $  u_{n,0} \to   u_0$ in $L^2(\mathcal G)$. Let $(  u_n)$ and $  u$ be the solutions of~\eqref{eq:nls-Hg} such that $  u_n(t_0) =   u_{n,0}$ and $  u(t_0) =   u_0$. Let $J$ be a compact interval of $\R$ such that $(  u_n)$ and $  u$ are defined on $J$. We assume that $\sup_{n \in \N^*} \|  u_n\|_{L^\infty(J,H^1(\mathcal G))} < \infty$. Then, as $n \to \infty$, we have
\[
\|  u_n-  u\|_{L^\infty(J,L^2(\mathcal G))} \to 0.
\]
\end{lemma}

\subsection{Change of variables}

The approach we adopt is based on a change of variables transforming a finite time blow-up
solution into a solution that is global in positive time.
We seek a radial solution $  u$ of~\eqref{eq:nls-Hg} in the form 
\begin{equation}\label{chvarw}
  u(t,  x)=\frac{1}{\sqrt{\lambda(s)}}  w(s,  y)e^{i(\theta(s)-b(s)  y^2/4)},\quad
t<0, \   x\in\mathcal G,
\end{equation}
where the new variables $s$ and $  y$ satisfy
\begin{equation}\label{chvar}
\frac{\dif s}{\dif t}=\frac{1}{\lambda(s)^2}, \qquad   y=\frac{  x}{\lambda(s)}.
\end{equation}
We will construct
$  w$ global and bounded in $H^1_D(\mathcal G)$, together with 
{\em modulation parameters} $\lambda(s)$, $b(s)$ and $\theta(s)$ such that $\lambda(s)>0$,
$$ 
\lambda(s)\to 0^+, \quad b(s)\to 0^+, \quad \theta(s)\to\infty, \quad s\to+\infty.
$$
This type of ansatz is common in blow-up analysis (see the references given in introduction for similar constructions).
The exact definition of the rescaled time $s$ will appear in Section~\ref{unifest_proof.sec}. 
By straightforward calculations, $  u$ solves~\eqref{eq:nls-Hg} if and only if $  w$ solves 
\begin{multline}\label{nls_w}
i  w_s-H_{\gamma\lambda}  w-  w+ |  w|^4  w
+(1-\theta_s)  w+\Big(b_s-b^2-2b\frac{\lambda_s}{\lambda}\Big)\frac{y^2}{4}  w
-i\Big(b+\frac{\lambda_s}{\lambda}\Big)\Lambda   w = 0,
\end{multline}
where the scaling operator $\Lambda$ is defined for each component $w_j$ of $  w$ by
\begin{equation}
\label{eq:dilation}
    \Lambda w_j(y_j)=\frac12w_j(y_j)+y_jw_j'(y_j)=\frac \dif {\dif\lambda} \big( \sqrt \lambda  w_j (\lambda y_j) \big) \big|_{\lambda = 1}, \quad j=1,\dots,N.
\end{equation}

\subsection{Blow-up profile}

To prove Theorem~\ref{th:blowup}, we seek $  w$ in the form 
$$
  w(s,  y)=  P(s,  y)+  h(s,  y),
$$ 
for a suitable approximate solution profile $  P$. The result will then follow 
from~\eqref{chvarw} and~\eqref{chvar} by proving that
$\lambda(s)\sim s^{-2}$ and $  h(s)\to0$ in a well-chosen norm, as $s\to+\infty$.

The blow-up profile $  P$ is constructed as an approximate solution of the 
auxiliary equation
\[
i  P_s - H_{\lambda\gamma }  P -  P+f(  P)+ \alpha\frac{y^2}{4}  P=0,
\]
with
$$
f(z)=|z|^4z, \quad z\in\mathbb{C}.
$$

For $\kappa \in \N^*$, we define
\begin{equation}\label{def_of_Sigma}
\Theta_\kappa = \Big\{ (j,k) \in \N \times \N^* \, : \, \frac {j} 2 + k < \kappa \Big\}.
\end{equation}

For $\nu \in \N$ we denote by $C^\nu_{\exp}$ the set of radial functions $u$ on $\Gc$ 
which are of class $C^\nu$ on each edge and such that $\|u\|_{C^\nu_{\exp,\rho}} < \infty$ 
for some $\rho > 0$, where
\[
\|  u\|_{C^\nu_{\exp,\rho}}  
= \sup_{0\leq m \leq \nu} \sup_{y \in \Gc} e^{\rho \abs y} \abs {u^{(m)}(y)}.
\]
We denote by $C^\infty_{\exp}$ the intersection of the $C^\nu_{\exp}$ for $\nu \in \N$. 
Note that $C^\infty_{\exp}$ is stable under multiplication by a polynomial.

Let $\beta \in (-N,N)$. The equation
\begin{equation} \label{eq:Qbeta}
    H_\beta Q_\beta + Q_\beta - |Q_\beta|^4 Q_\beta =0,
\end{equation}
has a unique radial non-trivial solution in $D(H_\beta)$. Note that uniqueness holds only for radial solutions and there exist non-radial solutions to~\eqref{eq:Qbeta}, see e.g.~\cite{adami2012stationary}.
It is given on each edge by 
\begin{equation}\label{eq:def_of_Q_beta}
    Q_\beta(y)=Q(y-\tau_\beta) = \frac {3^{\frac 14}}{\sqrt {\cosh(2(y-\tau_\beta))}}, \quad \tau_\beta = \frac12\tanh^{-1}\left(\frac\beta N\right). 
\end{equation}
We provide a representation of the function $  Q_\beta$ in Figure~\ref{fig:Q_gamma} (picture made with the Grafidi library, see~\cite{Grafidi,BeDuLe22A,BeDuLe22B}).
\begin{figure}
    \centering
    \includegraphics[width=0.5\textwidth]{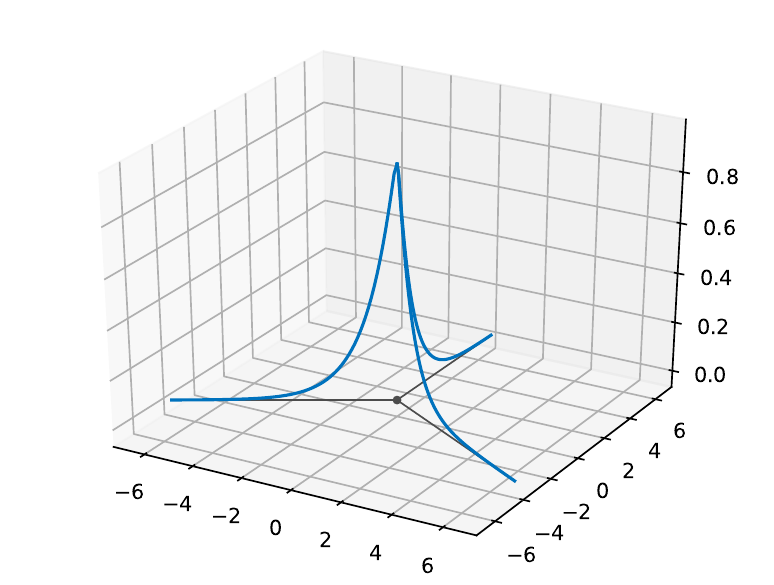}
    \caption{The function $  Q_\beta$ on a $3$-star graph for $\beta=-5/2$}
    \label{fig:Q_gamma}
\end{figure}
Notice for future reference that 
\begin{equation} \label{eq:der-Q-beta}
\partial_\beta Q_\beta(y) = -\frac N {2(N^2-\beta^2)} Q'(y-\tau_\beta)= -\frac N {2(N^2-\beta^2)} Q_\beta'(y). 
\end{equation}

We define the constant
\begin{equation} \label{def:alpha-01}
\alpha_\star  = -2\gamma\frac{ Q(0)^2}{\|yQ\|_{L^2(\mathcal G)}^2}
\end{equation}
which will play a central role in the analysis. Note that $\alpha_\star > 0$ for $\gamma < 0$.

The following proposition will be proved in Section~\ref{sec:profile}. 

\begin{proposition}[Approximate blow-up profile]
\label{prop:profile}
Let $\kappa\in\N^*$.
 There exist families $(  P_{j,k,\beta})$ and $(\alpha_{j,k,\beta})$ which depend on $(j,k,\beta) \in \Theta_\kappa \times (-N,N)$ and satisfy the following properties.
\begin{enumerate}[label={\upshape(\roman*)}]
\item For $(j,k) \in \Theta_\kappa$, $\beta \in (-N,N)$ and $\ell \in \N$, we have $P_{j,k,\beta} \in D_{\beta}$ and $\partial_\beta^\ell P_{j,k,\beta} \in  C^\infty_{\exp}$.
\item For $(j,k) \in \Theta_\kappa$ and $\beta \in (-N,N)$, we have $\alpha_{j,k,\beta} \in \R$. Moreover, $\alpha_{j,k,\beta} = 0$ if $j$ is odd.

\item In particular, $\alpha_{0,1,0} = \alpha_\star$.

\item Given an interval $J$ of $\R$, $b \in C^1(J,\R)$ and $\lambda \in C^1(J,\R_+^*)$ such that $\gamma \lambda(s) \in (-N,N)$ for all $s \in J$, if we set 
\begin{equation} \label{def:P}
  P = P_{b,\lambda}  =   Q_{\gamma \lambda} + \sum_{(j,k) \in \Theta_\kappa} (ib)^j \lambda^k   P_{j,k,\gamma \lambda },
\end{equation}
\begin{equation} \label{def:alpha}
\alpha = \alpha(b,\lambda)  =  
\sum_{(j,k) \in \Theta_\kappa} (ib)^j \lambda^k \alpha_{j,k,\gamma \lambda }
\end{equation}
and 
\begin{equation} \label{def:Psi-K}
  \Psi_\kappa =   \Psi_\kappa(b,\lambda) 
=i  P_s - H_{\gamma \lambda}    P-  P +f(  P)+ \alpha\frac{y^2}{4}  P,
\end{equation}
then for any $\nu \in \N$ there exist $\rho > 0$ and $C > 0$ such that, for any $s\in J$, there holds $  \Psi_\kappa(s)\in C^\nu_{\exp}$ and 
\begin{equation}\label{remainder_est}
\norm{  \Psi_\kappa}_{C^\nu_{\exp,\rho}} \leq C 
\lambda\left(\left|\frac{\lambda_s}{\lambda} + b \right|+\left|b_s+b^2-\alpha(b,\lambda)\right| \right)
+ C(b^2+\lambda)^{\kappa}.
\end{equation}

\item Defining
\begin{equation}\label{eq:def_tildeP}
\tilde {  P}(b,\lambda,\theta)=\lambda^{-1/2}  P_{b,\lambda}e^{i(\theta-b\frac{y^2}{4})}, 
\end{equation}
we have, for any $s\in J$,
\begin{equation}\label{deriv_energy}
\Big|\frac{\dif}{\dif s}E(\tilde {  P})\Big| \lesssim 
\frac{1}{\lambda^2}
\left(\Big|\frac{\lambda_s}{\lambda}+b\Big|+|b_s+b^2-\alpha(b,\lambda)|+(b^2+\lambda)^\kappa\right).
\end{equation}
\item There exist $(\eps_{j,k})_{(j,k) \in \Theta_\kappa}\subset\R$ such that, for any $s\in J$,
\begin{equation}\label{energy_nls_dynsys}
\Big|E(\tilde 
{  P}(b,\lambda,\theta))-C_Q\mathcal{E}(b,\lambda)\Big|
\lesssim \frac{(b^2+\lambda)^\kappa}{\lambda^2},
\end{equation}
where
\begin{equation}\label{asympt_nrj}
C_Q=\frac18\|yQ\|_{L^2(\Gc)}^2, \qquad \mathcal{E}(b,\lambda)=\mathcal{E}_\mathrm{mo}(b,\lambda)
+\sum_{\substack{(j,k) \in \Theta_\kappa\\ j \text{ even}, \ j/2+k\ge2}}
\eps_{j,k} b^j\lambda^{k-2}
\end{equation}
and $\mathcal{E}_\mathrm{mo}$ is the Hamiltonian of the model dynamical system, 
defined in~\eqref{Hamiltonian_mo}.
\end{enumerate}
\end{proposition}

\subsection{Modulation parameters}
A choice of modulation parameters $\theta(s), b(s), \lambda(s)$ can be made so that 
the remainder $  h$ satisfies orthogonality conditions which are useful to construct our solution. 
This is ensured by the following proposition, which will be proved in Section~\ref{sec:modulation}. Notice that the function $\rho \in L^2(\Gc)$ which appears in the last condition will be defined in Lemma~\ref{nondeg.lem} below.

\begin{proposition}[Modulation parameters]
\label{prop:modulation}
Let $I$ be an interval of $\R$ and consider a solution
$u \in C^1(I,L^2(\Gc))$ of~\eqref{eq:nls-Hg}. 
There exists $0<\varepsilon<N/|\gamma|$ with the following property. If, for all $t\in I$, there exist
$\theta\in\R%
$ and $\lambda\in(0,\eps)$ such that
\begin{equation}
\label{eq:dist_u_modulation} 
    \norm*{  u(t,  x)-\frac{1}{\sqrt{\lambda}}e^{i\theta}  Q\left(\frac {  x}\lambda\right)}_{L^2(\Gc)}\leq \varepsilon,
\end{equation}
then there exist $\theta \in C^1(I,\R)$, $b \in C^1(I,\R)$ and 
$\lambda\in C^1(I,(0,N/|\gamma|))$ such that the function $h \in C^1(I,L^2(\Gc))$ defined by
\[
  u(t,  x)=\frac{1}{\sqrt{\lambda(t)}}e^{i\theta(t)
-i\frac{b(t)  x^2}{4\lambda(t)^2}}\left(  P_{b(t),\lambda(t)}\left(\frac {  x}{\lambda(t)}\right)
+  h\left(t,\frac {  x}{\lambda(t)}\right)\right)
\]
satisfies for all $t \in I$ the orthogonality conditions
\begin{equation*}
\big(  h(t),i\Lambda   
P_{b(t),\lambda(t)} \big)_{L^2(\Gc)}
= \big(  h(t),y^2   P_{b(t),\lambda(t)}\big)_{L^2(\Gc)}
=\big(  h(t),i  \rho\big)_{L^2(\Gc)}=0.
\end{equation*}
\end{proposition}

\begin{remark}
To keep a light notation in this section, 
we use the same letters $b,\lambda,\theta, h$ to
denote the modulation parameters and rest as functions of $t$ or $s$. We will later
be more specific, see~\eqref{mod_param_st}.
\end{remark}

As we shall see in Section~\ref{sec:modulation_estimates}, 
the modulation parameters $b(s)$ and $\lambda(s)$ are governed, 
at first order as $s\to\infty$, by the model dynamical system 
\begin{equation}\label{dynsys_unp}
b_s+b^2-\alpha_\star\lambda=0, \quad
\frac{\lambda_s}{\lambda}+b=0,
\end{equation}
where $\alpha_\star $ is defined in~\eqref{def:alpha-01}.
The system~\eqref{dynsys_unp} is Hamiltonian, with conserved energy
\begin{equation}\label{Hamiltonian_mo}
\mathcal E_\mathrm{mo} (b,\lambda)
=\frac{b^2}{\lambda^2}-\frac{2\alpha_\star}{\lambda}.
\end{equation}
An exact solution with energy $\mathcal E_\mathrm{mo} = 0$ is given by
\begin{equation}
\label{def:lambda-mo}
b_\mathrm{mo}(s)=\frac2s, \quad \lambda_\mathrm{mo}(s)=\frac{2}{\alpha_\star s^2}.
\end{equation}

Proposition~\ref{prop:profile} shows that, at leading order, the energy of the
rescaled profile $\tilde {  P}$ is governed by the Hamiltonian energy~\eqref{Hamiltonian_mo}. 
However, the correction appearing
as a power expansion in~\eqref{asympt_nrj} does not vanish as $s\to\infty$. 
One has $\mathcal{E}(b,\lambda)=\mathcal{E}_\mathrm{mo}(b,\lambda)+e_0+o(1)$, where
$e_0$ is a constant (see Remark~\ref{rem:energy_corr}).
Nevertheless,
the relation between $E(\tilde {  P})$ and $\mathcal{E}_\mathrm{mo}$ suggests 
that, up to a shift of $e_0$ and a rescaling by $C_Q$,
the energy of our solution of~\eqref{eq:nls-Hg} should be controlled
by the model Hamiltonian energy $\mathcal{E}_\mathrm{mo}$.
More precisely, a natural approach to proving Theorem~\ref{th:blowup} would be: 
given $E^\star\in\R$, define $\mathcal{E}_\mathrm{mo}^\star$ by 
$E^\star=C_Q(\mathcal{E}_\mathrm{mo}^\star+e_0)$ and choose final data for the modulation parameters
$(b,\lambda)$ at a large time $s=s_1$ as $(\und b(s_1),\und \lambda(s_1))$, where
$(\und b,\und \lambda)$ is a solution of \eqref{dynsys_unp} with energy
$\mathcal E_\mathrm{mo} =\mathcal{E}_\mathrm{mo}^\star$.
Unfortunately, as can be seen by backward integration from $s=s_1$ using the modulation
estimates \eqref{mod_est}, the difference
$\mathcal{E}_\mathrm{mo}(\und b,\und \lambda)
-\mathcal{E}_\mathrm{mo}(b,\lambda)$
grows like $\log(s)$ as $s\to\infty$. As a consequence, the energy of our solution of~\eqref{eq:nls-Hg} 
cannot be controlled by $\mathcal{E}_\mathrm{mo}$ alone. Instead, the full expansion $\mathcal E(b,\lambda)$
must be used and the choice of final data for $(b,\lambda)$ is more involved.
The following proposition will be proved in Appendix~\ref{sec:dynamical_systems}.

For a fixed parameter $\lambda_0>0$ such that $\mathcal E^\star\lambda_0+2\alpha_\star>0$, let
\begin{equation}\label{defofF}
\mathcal F(\lambda)
=\int_\lambda^{\lambda_0}\frac{\diff\mu}{\mu^{3/2}\sqrt{\mathcal E^\star\mu+2\alpha_\star}},
\quad \lambda\in(0,\lambda_0].
\end{equation}

\begin{proposition}
\label{prop:dynamic}
Let $\mathcal E^\star\in\mathbb R$. There exist $s_0 \geq 1$ and $c_1>0$ such that,
for any $s_1\ge s_0$ we can find $b_1,\lambda_1>0$ satisfying
\begin{equation*}
\mathcal F(\lambda_1)=s_1, \quad
\mathcal E(b_1,\lambda_1)=\mathcal E^\star, \quad \Big|\frac{\lambda_1^{1/2}}{\lambda_\mathrm{mo}(s_1)^{1/2}}-1\Big|\le \frac{c_1}{s_1},
\quad
\Big|\frac{b_1}{b_\mathrm{mo}(s_1)}-1\Big|\le \frac{c_1}{s_1}.
\end{equation*}
\end{proposition}

\subsection{Proof of the main theorem}
We now define the final data which will give rise to an approximate solution of our problem
by backward in time integration of~\eqref{eq:nls-Hg}.
Let $E^\star\in\R$. Define
\begin{equation}\label{energies}
\mathcal E^\star=C_Q^{-1}E^\star.
\end{equation}
Consider $t_1<0$ and sufficiently close to $0$. Define 
the associated rescaled final time $s_1$ (see~\eqref{eq:def_of_s_1}) by 
\begin{equation}
\label{eq:def_of_s_1_first}
s_1=\Big(\frac{4}{3\alpha_\star^2} \Big)^{1/3}|t_1|^{-1/3}.
\end{equation}  
Let $(b_1,\lambda_1)$ be given by Proposition~\ref{prop:dynamic} and
$  u_1$ be the radial solution of~\eqref{eq:nls-Hg} such that 
\begin{equation}\label{eq:definition_of_u_1}
  u_1(t_1,  x)=\frac{1}{\sqrt{\lambda_1}}  P_{b_1,\lambda_1}
\Big(\frac{  x}{\lambda_1}\Big)e^{-i\frac{b_1  x^2}{4\lambda_1^2}}.
\end{equation}
Let $I\subset\R$ be the maximal interval such that $t_1\in I$, $  u_1$ exists on $I$ and satisfies~\eqref{eq:dist_u_modulation}. Observe that, for $t_1$ close enough to $0$, 
the parameters $b_1, \lambda_1$
given by Proposition~\ref{prop:dynamic} are small, so that $P_{b_1,\lambda_1}$ is close to $Q$.
By Proposition~\ref{prop:cauchy}, we have 
$u_1 \in C^0(I,D_{\gamma}\cap \Sigma^2(\Gc)) \cap C^1(I,L^2(\Gc))$. 
Let $\theta$, $b$, $\lambda$ and $  h$ be given by Proposition~\ref{prop:modulation}. 
Note that $h$ also belongs to 
$C^0(I,D_{\gamma}\cap \Sigma^2(\Gc)) \cap C^1(I,L^2(\Gc))$. The asymptotics as $t\to0^-$ of 
the functions $\theta$, $b$, $\lambda$ and $  h$ follow from the next proposition, which will be proved in Section~\ref{unifest_proof.sec}.

\begin{proposition}[Uniform estimates in the $t$ variable]\label{prop:uniform_in_t}
Let $\kappa\in \N^*$ (see \eqref{def_of_Sigma}).
There exists $t_0 \in (-\infty,0)$ such that, for any $t_1 \in (t_0,0)$, 
the solution $u_1$ of \eqref{eq:nls-Hg} with final data
\eqref{eq:definition_of_u_1} is defined and satisfies the hypotheses of 
Proposition~\ref{prop:modulation} on $[t_0,t_1]$. Furthermore, its decomposition
given by Proposition~\ref{prop:modulation} satisfies, for all $t\in[t_0,t_1]$,
  \begin{gather}
            \big|b(t)-C_b|t|^{\frac13}\big|\lesssim |t|,\quad 
            \big|\lambda(t)-C_\lambda|t|^{\frac23}\big|\lesssim |t|^{5/3}, %
            \label{t_est_bl}\\
            \norm{  h(t)}_{L^2(\mathcal G)} \lesssim \abs t^{\frac {\kappa-1}{3}}, \quad 
            \norm{  h_y(t)}_{L^2(\mathcal G)} \lesssim 
            \abs t^{\frac {\kappa-1}{3}}, \quad 
            \norm{  y   h(t)}_{L^2(\mathcal G)} \lesssim \abs t^{\frac{\kappa-2}{3}}, 
            \label{t_est_h}\\
           |E(\tilde P(b,\lambda,\theta)(t))-E^\star|\lesssim |t|^{\frac{\kappa-5}{3}}, \label{t_est_nrj}
    \end{gather}
    where $C_b=2\big(\frac{3\alpha_\star^2}{4}\big)^{1/3}$,
    $C_{\lambda}=\frac{2}{\alpha_\star}\big(\frac{3\alpha_\star^2}{4}\big)^{2/3}$.
    All these estimates are independent of $t_1$.
\end{proposition}

The phase modulation parameter $\theta$ can be shown to satisfy $\theta(t)\sim C_\theta|t|^{-\frac13}$ for some $C_\theta>0$, but it plays no specific role in the proof of Theorem~\ref{th:blowup}. 

\begin{remark}\label{rem:energy_corr}
In Section~\ref{unifest_proof.sec}, Proposition~\ref{prop:uniform_in_t} will
be deduced from its counterpart in the $s$ variable, Proposition~\ref{prop:uniform_in_s},
which says in particular that 
$b(s)\sim b_\mathrm{mo}(s)$ and $\lambda(s)\sim \lambda_\mathrm{mo}(s)$ as $s\to\infty$.
Therefore, the terms corresponding to $(j,k)=(0,2)$ and $(j,k)=(2,1)$ in the
expansion \eqref{asympt_nrj} behave asymptotically as
\begin{equation*}
\eps_{0,2}+\eps_{2,1}\frac{b^2}{\lambda}\sim 
\eps_{0,2}+\eps_{2,1}\frac{b_\mathrm{mo}^2}{\lambda_\mathrm{mo}}
\sim \eps_{0,2}+2\alpha_\star\eps_{2,1}, \quad s\to\infty.
\end{equation*}
Hence,
\begin{equation*}\label{asympt_nrj2}
\mathcal{E}(b,\lambda)\sim\mathcal{E}_\mathrm{mo}(b,\lambda)
+\eps_{0,2}+2\alpha_\star\eps_{2,1}, \quad s\to\infty.
\end{equation*}
Thanks to Remark~\ref{rem:improved_estimates}, the first two estimates 
in Proposition~\ref{prop:uniform_in_t} can be improved to
\begin{equation*}
\big|b(t)-C_b|t|^{\frac13}\big|\lesssim |t|^{5/3},
\quad 
\big|\lambda(t)-C_\lambda|t|^{\frac23}\big|\lesssim |t|^{7/3}
\end{equation*}
by replacing the energy $\mathcal E^\star$ in~\eqref{defofF} with
$\underline{\mathcal E}^\star=\mathcal E^\star-(\eps_{0,2}+2\alpha_\star\eps_{2,1})$.
\end{remark}

We are now in position to prove our main result.

\begin{proof}[Proof of Theorem~\ref{th:blowup}, assuming 
Propositions~\ref{prop:profile},~\ref{prop:modulation}, 
\ref{prop:dynamic} and~\ref{prop:uniform_in_t}] 

For a given $E^\star\in \mathbb R$, let $\mathcal E^\star$ be defined by \eqref{energies}.
Choose an increasing sequence of times $(t_n)\subset(t_0,0)$ 
such that $t_n\to 0$ as $n\to\infty$ and corresponding rescaled times $s_n\to \infty$ defined through~\eqref{eq:def_of_s_1_first}. For each $n\in\N^*$, 
let $b_n$ and $\lambda_n$ as given by Proposition~\ref{prop:dynamic}
and $  u_n(t_n)$ defined by~\eqref{eq:definition_of_u_1}, with the change of notation $t_1\to t_n$, $s_1\to s_n$, $b_1\to b_n, \ \lambda_1 \to \lambda_n$.
For each $n\in\N^*$, the corresponding solution $  u_n$ of~\eqref{eq:nls-Hg}
satisfies Proposition~\ref{prop:uniform_in_t} on $[t_0,t_n]$, with all the estimates independent
of $n$. We will show that $(  u_n)$ converges to a solution $  u$ of 
\eqref{eq:nls-Hg} with the desired properties.

Let $\chi \in C^\infty([0,\infty),[0,1])$ be equal to 0 on $[0,1]$ and equal to 1 on $[2,\infty)$. 
For $R > 0$ we define the radial function $  \chi_R$ on $\mathcal G$ by $\chi(x/R)$ on each edge.  
Let $\epsilon > 0$. From the formula~\eqref{eq:definition_of_u_1} defining $u_n(t_n)$ we deduce that
there exists $R > 0$ such that 
\begin{equation*} \label{eq:un-delta}
\int_{\mathcal G} \abs{  u_n(t_n)}^2   \chi_R \diff   x \leq \epsilon. 
\end{equation*}
Since $u$ verifies~\eqref{eq:nls-Hg}, we have
\[
\frac \dif {\dif t} \int_{\mathcal G} \abs{  u_n}^2   \chi_R \diff   x 
= 2 \im \int_{\mathcal G} \partial_x   u_n \bar {  u}_n \partial_x   \chi_R \diff   x.
\]
Using the decomposition of the solution given by Proposition~\ref{prop:modulation},
the estimates of Proposition~\ref{prop:uniform_in_t} for the corresponding variables
$b_n,\lambda_n,  h_n$ and the exponential decay of $  P_{b_n,\lambda_n}$, 
direct calculations show that
\begin{align*}
\left | \frac \dif {\dif t} \int_{\mathcal G} \abs{  u_n}^2   \chi_R \diff x \right| 
& \lesssim \frac 1 {R \lambda_n(t)} \big( e^{-\frac {R}{2 \lambda_n(t)}} 
+\|yh_n(t)\|_{L^2(|y|\ge R/\lambda_n(t))}^2 +\|h_n(t)\|_{H^1(|y|\ge R/\lambda_n(t))}^2 \big)
\lesssim \frac {|t|^{\frac 23 (\kappa-3)}}R.
\end{align*}
Thus, integrating over $[t_0,t_n]$, we find a constant $C>0$ such that 
(choosing $R$ larger if necessary)
\[
\int_{\mathcal G} \abs{  u_n(t_0)}^2   \chi_R \diff   x \le
C\frac {|t_0|^{\frac 23 (\kappa-3)+1}}R
+ \int_{\mathcal G} \abs{  u_n(t_n)}^2   \chi_R \diff   x \leq 2 \epsilon.
\]
After extracting a subsequence if necessary, we obtain that 
the sequence $(  u_n(t_0))$ has a limit $  u_0$ in $L^2(\mathcal G)$ (since the sequence  is bounded  $H^1(\mathcal G)$). 
We denote by $  u$ the maximal solution of~\eqref{eq:nls-Hg} with initial condition $  u(t_0) =   u_0$. 

Let $\tau \in (t_0,0)$ and assume by contradiction that $  u$ is not defined on $[t_0,\tau]$. 
Let $n \in \N^*$ such that $t_n > \tau$. 
Let $C > 0$ be such that $\norm{  u_n(t)}_{H^1(\mathcal G)} \leq C$ for all $n \in \N^*$ 
and $t \in [t_0,\tau]$. 
By the blow-up alternative, 
there exists $\tau_1$ in $[t_0,\tau]$ such that 
$  u$ is defined on $[t_0,\tau_1]$ and $\norm{  u(\tau_1)}_{H^1(\mathcal G)} \geq 2 C$. 
For all $t \in [t_0,\tau_1]$ the sequence $(  u_n(t))$ goes to $  u(t)$ in $L^2(\Gc)$ (see Lemma~\ref{lem:conv-L2-solution}) and 
has a weak limit in $H^1(\Gc)$.
Thus, $(  u_n(t))$ goes weakly to $  u(t)$ in $H^1(\Gc)$. 
In particular, $\norm{  u(t)}_{H^1(\Gc)} \leq C$ for all $t \in [t_0,\tau_1]$. 
This gives a contradiction and proves that $  u$ is defined on $[t_0,\tau]$. 
Since $\tau$ is arbitrary, this implies that $  u$ is well defined on $[t_0,0)$.

Since $u_n(t)\to u(t)$ in $L^2(\mathcal G)$ for all $t\in[t_0,0)$, 
Proposition~\ref{prop:modulation} can be applied to $u$,
yielding modulation parameters and remainder $b_\infty(t), \lambda_\infty(t), \theta_\infty(t),   h_\infty$, defined on $[t_0,0)$.
Then, by standard arguments, for all $t\in[t_0,0)$, there holds
\[
\theta_n(t) \to \theta_\infty(t), \quad b_n(t) \to b_\infty(t), \quad 
\lambda_n(t) \to \lambda_\infty(t),
\]
and, weakly in $\Sigma^1(\Gc)$, 
\[
  h_n(t) \rightharpoonup   h_\infty(t), \quad  n \to \infty.
\]
By Proposition~\ref{prop:uniform_in_t}, we deduce that, as $t\to0^-$,
\[
b_\infty(t) \sim C_b |t|^{\frac 13}, \quad 
\lambda_\infty(t) \sim C_\lambda |t|^{\frac 23},
\]
\[
\norm{  h_\infty(t)}_{L^2(\mathcal G)} \lesssim \abs t^{\frac {\kappa-1}{3}}, \quad 
\norm{\partial_y  h_\infty(t)}_{L^2(\mathcal G)} \lesssim\abs t^{\frac {\kappa-1}{3}}, \quad 
\norm{  y   h_\infty(t)}_{L^2(\mathcal G)} \lesssim \abs t^{\frac{\kappa-2}{3}}.
\]
Using $  y=  x/\lambda_\infty$, the decomposition of $  u$ 
given by Proposition~\ref{prop:modulation} and the formula for $  P_{b_\infty,\lambda_\infty}$
in Proposition~\ref{prop:profile}, it follows by direct calculations that
\begin{align*}
\|  u(t)\|_{L^2(\mathcal G)}^2&=\int_\mathcal G |  P_{b_\infty,\lambda_\infty}(  y)+  h_\infty(t,  y)|^2\diff   y
\longrightarrow \|  Q\|_{L^2(\mathcal G)}^2, \quad t\to0^-,  \\
\|  u_x(t)\|_{L^2(\mathcal G)}^2
&=\lambda_\infty(t)^{-2}
\int_\mathcal G \left|-ib\frac{  y}{2}\big(  P_{b_\infty,\lambda_\infty}(  y)+  h_\infty(t,  y)\big)
+\partial_y\big(  P_{b_\infty,\lambda_\infty}(  y)+  h_\infty(t,  y)\big)\right|^2\dif   y. 
\end{align*}
As $t\to0^-$, this implies 
\begin{equation}
    \label{eq:constant_explicit}
  \|  u_x(t)\|_{L^2(\mathcal G)}^2\sim \lambda_\infty(t)^{-2}\|Q_y\|_{L^2(\mathcal G)}^2
\sim C_\lambda^{-2}\|Q_y\|_{L^2(\mathcal G)}^2 |t|^{-\frac 43}.
\end{equation}
This proves~\eqref{blowup_speed} and, since $\|  u\|_{L^2(\mathcal G)}^2$ is constant,
that $M(  u)=M(  Q)$.

To complete the proof, we now show that $E(  u)=E^\star$. 
By~\eqref{energy_nls_dynsys} and~\eqref{t_est_nrj}, there exists a function 
$\eps : [t_0,0) \to \R_+^*$ with $\lim_{t\to0^-}\eps(t)=0$ and such that, for all $n \in \N^*$,
\[
\big|\mathcal E \big( b_n(t),  \lambda_n(t) \big) - C_Q^{-1}E^\star\big| \le \eps (t),
\quad t\in[t_0,0).
\]
Letting $n \to \infty$ yields
\[
\big|\mathcal E \big( b_\infty(t),  \lambda_\infty(t) \big) 
- C_Q^{-1}E^\star\big| \le \eps (t),
\quad t\in[t_0,0).
\]
Hence, using again~\eqref{energy_nls_dynsys}, we conclude that
\[
E \big(   P_{ b_\infty(t),  \lambda_\infty (t) } \big) \longrightarrow E^\star,  \quad t\to0^-.
\]
It then follows from the above information about $b_\infty,\lambda_\infty$ and $h_\infty$ that
\[
E(  u(t)) \longrightarrow E^\star,  \quad t\to0^-.
\]
By conservation of the energy, we deduce that $E(  u(t)) = E^\star$ for all $t \in [t_0,0)$.
\end{proof}

\section{Cauchy problem}\label{sec:cauchy}

\subsection{Local well-posedness}

Since the operator $H_\gamma$ is selfadjoint, it generates a strongly continuous group  $e^{-itH_\gamma}$.
As we are working in a one-dimensional setting, the nonlinearity $|u|^4u$ is Lipschitz continuous from bounded sets of $H^1_D(\mathcal G)$ to $L^q(\mathcal G)$, $2\leq q\leq \infty$, 
and well-posedness of the Cauchy problem for~\eqref{eq:nls} in the energy space $H^1_D(\mathcal G)$ may be obtained 
following a classical line of arguments (see e.g.~\cite{AnCa19}). 
For any $t_0 \in \R$ and any initial data $u_0\in H^1_D(\mathcal G)$, there exists a unique maximal solution
\[
  u\in C\left((T_{\min},T_{\max}),H^1_D(\mathcal G)\right) \cap  C^1\left((T_{\min},T_{\max}),H^1_D(\mathcal G)^\star\right)
\]
such that $u(t_0)=u_0$. The energy $E$ and the mass $M$, defined in~\eqref{eq:def_E} and~\eqref{eq:def_M}, are preserved along the time evolution, i.e.~for any $t\in (T_{\min},T_{\max})$, we have
\[
E(u(t))=E(u_0),\quad M(u(t))=M(u_0).
\]
The blow-up alternative holds, i.e.~either $T_{\max}=+\infty$ (resp. $T_{\min}=-\infty$) or
\[
\lim_{t\to T_{\max}\text{ (resp. }T_{\min}\text{)}}\norm{u(t)}_{H^1(\mathcal G)}=\infty.
\]
There is continuous dependence with respect to the initial data, i.e.~for any $(u_{0,n})\subset H^1(\mathcal G)$ such that $u_{0,n}\to u_0$ in $H^1(\mathcal G)$ the associated solutions $(u_n)$ of~\eqref{eq:nls} are defined on $[T_*,T^*]$ and verify $u_n\to u$ in $C((T_*,T^*),H^1(\mathcal G))$ for any $T_*,T^*$ such that $T_{\min} < T_* <T^*<T_{\max}$.

If in addition $u_0\in D(H_\gamma)$, then $u$ verifies
\[
  u\in C\left((T_{\min},T_{\max}),D(H_\gamma)\right) \cap C^1\left((T_{\min},T_{\max}),L^2(\mathcal G)\right).
\]
Finally, if the initial data $u_0$ belongs to the virial space $\Sigma^1(\mathcal G)$, 
then the solution $u$ verifies (see~\cite{GoOh20})
\[
u\in C\left((T_{\min},T_{\max}),\Sigma^1(\mathcal G)\right)
\]
and the virial identity is satisfied:
\[
\frac{\dif^2}{\dif t^2}\norm{x  u(t)}_{L^2(\mathcal G)}^2=8E(  u(t))+4\gamma |  u(t,  0)|^2.
\]

\subsection{Global existence}
\label{sec:global}

We now prove some global existence results for the nonlinear Schr\"odinger equation
\eqref{eq:nls}. As in the classical $\mathbb R^d$ case, they are obtained using Gagliardo-Nirenberg type inequalities.

\begin{lemma}
[Gagliardo-Nirenberg inequalities on star graphs]
The following inequalities hold:
\begin{align}
    \|  u\|_{L^6(\mathcal G)}^6&\le \frac{3\max\left\{1,\frac{4}{N^2} \right\}}{\|\mathsf Q\|_{L^2(\mathbb R)}^4}\|  u_x\|_{L^2(\mathcal G)}^2\|  u\|_{L^2(\mathcal G)}^4,\quad   u\in H^1_D(\mathcal G),\label{GN}\\
    \|  u\|_{L^6(\mathcal G)}^6&\le \frac{12}{N^2\|\mathsf Q\|_{L^2(\mathbb R)}^4}\|  u_x\|_{L^2(\mathcal G)}^2\|  u\|_{L^2(\mathcal G)}^4,\quad   u\in H^1_{\mathrm{rad}}(\mathcal G).\label{GN_radial}
\end{align}
\end{lemma}

\begin{proof}
The inequality~\eqref{GN} is well-known to hold on the line $\mathbb R$, i.e.~in case $N=2$, and the extension to a general star graph is carried out as follows. 
The case $N=1$ can be deduced from the case $N=2$ by 
extending a function defined on $\R_+$ evenly to $\R$. This yields, on the half-line $\mathbb R_+$, 
\[
   \|u\|_{L^6(\R_+)}^6\le \frac{12}{\|\mathsf Q\|_{L^2(\mathbb R)}^4}\|u_x\|_{L^2(\R_+)}^2\|u\|_{L^2(\R_+)}^4,\quad u\in H^1(\R_+).
\]
Let $  u\in H^1_{\mathrm{rad}}(\mathcal G)$ and denote by $u:\mathbb R_+\to \mathbb R$ the function representing $  u$ on any branch of the graph. We have
\begin{equation*}
     \|  u\|_{L^6(\mathcal G)}^6=N \|u\|_{L^6(\R_+)}^6
 \le N\frac{12}{\|\mathsf Q\|_{L^2(\mathbb R)}^4}\|u_x\|_{L^2(\R_+)}^2\|u\|_{L^2(\R_+)}^4
 =\frac{12}{N^2\|\mathsf Q\|_{L^2(\mathbb R)}^4}\|  u_x\|_{L^2(\mathcal G)}^2\|  u\|_{L^2(\mathcal G)}^4,
\end{equation*}
which establishes~\eqref{GN_radial}. Consider now the $H^1_D(\mathcal G)$ setting and assume $N\geq 2$. Since functions in $H^1_D(\mathcal G)$ are continuous at the vertex, and the star graph with $N\geq 2$ contains at least two half-lines (which can be thought of as a full line), the constant in the Gagliardo-Nirenberg inequality in  $H^1_D(\mathcal G)$ can be at best the same as the one on $H^1(\mathbb R)$ (this can be seen rigorously from a symmetric rearrangement on the graph, see~\cite[Theorem 3.2]{AdSeTi17a}). Consider the sequence of functions $(  u^n)\subset H^1_D(\mathcal G)$ defined by 
\[
u^n_1(x_1)=\mathsf Q(x_1 - n),\quad u^n_j(x_j)=\mathsf Q(x_j+n), \quad j\neq 1.
\]
The sequence is built with a bump on one half-line and tails on the others, the bump going away from the vertex. As $n\to\infty$, the sequence optimizes the Gagliardo-Nirenberg inequality in  $H^1_D(\mathcal G)$ with a constant which is the same as the one on $H^1(\mathbb R)$ (it also illustrates the fact that an optimizer does not exist). Hence~\eqref{GN}.
\end{proof}

\begin{proposition}[Global well-posedness]
\label{prop:global-wellposedness}
Let $\gamma\in\R$, $  u_0$ be an initial data, $t_0$ an initial time and 
$  u$ be the corresponding solution  of~\eqref{eq:nls} such that $  u(t_0,\cdot)=  u_0$. 

If $  u_0\in H^1_D(\mathcal G)$ satisfies 
$\|  u_0\|_{L^2(\mathcal G)}^2<\min\left\{1,\frac N2\right\}\|\mathsf Q\|_{L^2(\R)}^2$,
then $  u$ is global in $H^1(\mathcal G)$. Furthermore, if 
$\gamma>0$, then for any solution with $\|  u_0\|_{L^2(\mathcal G)}^2=\min\left\{1,\frac N2\right\}\|\mathsf Q\|_{L^2(\mathbb R)}^2$, $|  u(t,0)|$ remains bounded.

If $  u_0\in H^1_{\mathrm{rad}}(\mathcal G)$ satisfies $\|  u_0\|_{L^2(\mathcal G)}^2<\frac N2\|\mathsf Q\|_{L^2(\mathbb R)}^2$,
then $  u$ is global in $H^1_{\mathrm{rad}}(\mathcal G)$. Furthermore, if 
$\gamma>0$, then for any solution with 
$\|  u_0\|_{L^2(\mathcal G)}^2=\frac N2\|\mathsf Q\|_{L^2(\mathbb R)}^2$,
$|  u(t,0)|$ remains bounded. 
\end{proposition}

\begin{proof}
The proof follows by combining~\eqref{GN} with the conservation laws of~\eqref{eq:nls}.
We have
\begin{align*}
E(  u_0)=E(  u(t))
&=\frac12\|  u_x(t)\|_{L^2(\mathcal G)}^2-\frac16\|  u(t)\|_{L^6(\mathcal G)}^6+\frac{\gamma}{2}|  u(t,0)|^2\\
&\ge \frac12\left[1-\max\left\{1,\frac{4}{N^2}\right\}\left(\frac{\|  u_0\|_{L^2(\mathcal G)}}{\|\mathsf Q\|_{L^2(\mathbb R)}}\right)^4\right]\|  u_x(t)\|_{L^2(\mathcal G)}^2
+\frac{\gamma}{2}|  u(t,0)|^2.
\end{align*}
If $\gamma>0$, it follows that $\|  u_x(t)\|_{L^2(\mathcal G)}$ remains bounded provided 
$\|  u_0\|_{L^2(\mathcal G)}^2<\min\left\{1,\frac N2\right\}\|\mathsf Q\|_{L^2(\mathbb R)}^2$, and global existence in $H^1(\mathcal G)$ follows by the
blow-up alternative. Moreover in this case, if $\|  u_0\|_{L^2(\mathcal G)}^2=\min\left\{1,\frac N2\right\}\|\mathsf Q\|_{L^2(\mathbb R)}^2$, we see
that $|  u(t,0)|^2$ must remain bounded.

If $\gamma<0$, the inequality $|  u(t,0)|^2\le 2\|  u_x(t)\|_{L^2(\mathcal G)}\|  u(t)\|_{L^2(\mathcal G)}$ yields
\[
|  u(t,0)|^2\le \ep\|  u_x(t)\|_{L^2(\mathcal G)}^2+\frac{4}{\ep}\|  u(t)\|_{L^2(\mathcal G)}^2
\]
for any $\ep>0$, and it follows that
\[
E(  u_0)\ge 
\frac12\left[1-\max\left\{1,\frac{4}{N^2}\right\}\left(\frac{\|  u_0\|_{L^2(\mathcal G)}}{\|\mathsf Q\|_{L^2(\R)}}\right)^4-|\gamma|\ep\right]\|  u_x(t)\|_{L^2(\mathcal G)}^2
+\frac{2\gamma}{\ep}\|  u_0\|_{L^2(\mathcal G)}^2.
\]
If $\|  u_0\|_{L^2(\mathcal G)}^2<\min\left\{1,\frac N2\right\}\|\mathsf Q\|_{L^2(\R)}^2$, we can choose $\ep>0$ so that
\[
1-\max\left\{1,\frac{4}{N^2}\right\}\left(\frac{\|  u_0\|_{L^2(\mathcal G)}}{\|\mathsf Q\|_{L^2(\R)}}\right)^4-|\gamma|\ep>0,
\]
showing that 
$\|  u_x(t)\|_{L^2(\mathcal G)}$ remains bounded. The second part of the Proposition follows from similar arguments,  using~\eqref{GN_radial} instead of~\eqref{GN}. This concludes the proof.
\end{proof}

\subsection{Well-posedness in weighted spaces}
We establish in this section the well-posedness of the Cauchy problem in weighted spaces. Recall that the weighted spaces $\Sigma^k$ are defined in ~\eqref{eq:def_Sigma_k}.
The following lemma will be useful in the sequel.

\begin{lemma}
\label{lem:sigma_k_est}
 Let $k\in\mathbb N$ and $\alpha,\beta\in\mathbb N$ such that $\alpha+\beta\le k$. Then for any $u\in\Sigma^k$, the following holds:
    \[
    \norm{x^\alpha \partial_x^\beta u}_{L^2(\Gc)}\lesssim \norm{u}_{\Sigma^k}.
    \]
\end{lemma} 

\begin{proof}
We prove the result by induction on $k$. The case $k=0$ is trivial, since $\norm{u}_{\Sigma^0}^2=2\norm{u}_{L^2}^2$.
Now fix $k\in\N$ and suppose the result is true for all $j\le k$:
\begin{equation}\label{eq:induction_hypothesis}
\forall \alpha',\beta'\in\N, \ \alpha'+\beta'\le j, \quad 
\norm{x^{\alpha'} \partial_x^{\beta'} u}_{L^2(\Gc)}\lesssim \norm{u}_{\Sigma^j}.
\end{equation}
Let $\alpha,\beta\in\N$ such that $\alpha+\beta\le k+1$. To complete the proof, we will show that
\begin{equation}\label{eq:induction_conclusion}
\norm{x^{\alpha} \partial_x^{\beta} u}_{L^2(\Gc)}\lesssim \norm{u}_{\Sigma^{k+1}}.
\end{equation}
Firstly, if $\alpha+\beta\le k$, then~\eqref{eq:induction_conclusion} follows from~\eqref{eq:induction_hypothesis} with 
$j=k$. So we suppose that $\alpha+\beta= k+1$.
Next, we observe that the cases $(\alpha,\beta)=(0,k+1)$ and $(\alpha,\beta)=(k+1,0)$ are trivial. So we suppose
$1\le \alpha,\beta \le k, \ \alpha+\beta= k+1$ for the remainder of the proof.

Naturally, we start by estimating $\|x\partial_x^ku\|_{L^2(\Gc)}$. We shall use the
short hand notation $\partial_x^\beta u \equiv u^{(\beta)}$, for $\beta=1,\dots,k+1$.
Integrating by parts, we obtain
\begin{align*}
\|x u^{(k)}\|_{L^2(\Gc)}^2
&=\intg x^2 u^{(k)} \bar u^{(k)} \diff x \\
&=-2\intg x u^{(k)}\bar{u}^{(k-1)}\diff x-\intg x^2 u^{(k+1)}\bar{u}^{k-1}\diff x.
\end{align*}
By the induction hypothesis, we have
\[
\Big| \intg x u^{(k)}\bar{u}^{(k-1)}\diff x \Big|
\le \|x u^{(k-1)}\|_{L^2(\Gc)}\|u^{(k)}\|_{L^2(\Gc)} 
\le \|u\|_{\Sigma^k}^2
\le \|u\|_{\Sigma^{k+1}}^2.
\]
On the other hand, for any $\eps_1>0$, 
\[
\Big| \intg u^{(k+1)}x^2 \bar{u}^{(k-1)}\diff x \Big| 
\le \eps_1^2 \|u^{(k+1)}\|_{L^2(\Gc)}^2+\eps_1^{-2}\|x^2 u^{(k-1)}\|_{L^2(\Gc)}^2.
\]
It follows that
\begin{equation*}
\|x u^{(k)}\|_{L^2(\Gc)}^2
\le 2 \|u\|_{\Sigma^{k+1}}^2 + \eps_1^2 \|u^{(k+1)}\|_{L^2(\Gc)}^2+\eps_1^{-2}\|x^2 u^{(k-1)}\|_{L^2(\Gc)}^2,
\end{equation*}
for any $\eps_1>0$.
With the same arguments, we obtain
\[
\|x^2 u^{(k-1)}\|_{L^2(\Gc)}^2
\le 4 \|u\|_{\Sigma^{k+1}}^2 + \eps_2^2 \|xu^{(k)}\|_{L^2(\Gc)}^2+\eps_2^{-2}\|x^3 u^{(k-2)}\|_{L^2(\Gc)}^2,
\]
for any $\eps_2>0$. Continuing this process, and using
\[
\|u^{(k+1)}\|_{L^2(\Gc)}^2\le \|u\|_{\Sigma^{k+1}}^2, \quad
\|x^{k+1} u\|_{L^2(\Gc)}^2\le \|u\|_{\Sigma^{k+1}}^2,
\]
we find that, for any sequence of positive numbers $\eps_1,\eps_2,\dots,\eps_k$, 
\begin{align*}
\|x u^{(k)}\|_{L^2(\Gc)}^2
&\le (2+\eps_1^2)\|u\|_{\Sigma^{k+1}}^2 + \eps_1^{-2}\|x^2 u^{(k-1)}\|_{L^2(\Gc)}^2 \\
\|x^2 u^{(k-1)}\|_{L^2(\Gc)}^2
&\le 4 \|u\|_{\Sigma^{k+1}}^2 + \eps_2^2 \|xu^{(k)}\|_{L^2(\Gc)}^2+\eps_2^{-2}\|x^3 u^{(k-2)}\|_{L^2(\Gc)}^2 \\
\|x^3 u^{(k-2)}\|_{L^2(\Gc)}^2
&\le 6 \|u\|_{\Sigma^{k+1}}^2 + \eps_3^2 \|x^2u^{(k-1)}\|_{L^2(\Gc)}^2+\eps_3^{-2}\|x^4 u^{(k-3)}\|_{L^2(\Gc)}^2 \\
&\vdotswithin{ = } \\
\|x^{k-1} u^{(2)}\|_{L^2(\Gc)}^2
&\le 2(k-1) \|u\|_{\Sigma^{k+1}}^2 + \eps_{k-1}^2 \|x^{k-2}u^{(3)}\|_{L^2(\Gc)}^2
+\eps_{k-1}^{-2}\|x^{k}u^{(1)}\|_{L^2(\Gc)}^2 \\
\|x^k u^{(1)}\|_{L^2(\Gc)}^2
&\le (2k+\eps_k^{-2})\|u\|_{\Sigma^{k+1}}^2 + \eps_k^2 \|x^{k-1}u^{(2)}\|_{L^2(\Gc)}^2.
\end{align*}

We will now deduce~\eqref{eq:induction_conclusion}, for all $1\le \alpha,\beta \le k, \ \alpha+\beta= k+1$,
from this system of inequalities. The main difficulty is that, for $l=2,\dots,k-1$, 
line number $l$ involves terms appearing in lines $l-1$ and $l+1$. 
This can be remedied by injecting inequality $l-1$ into line $l$. For instance, for $l=2$, 
we obtain
\[
\|x^2 u^{(k-1)}\|_{L^2(\Gc)}^2
\le 4 \|u\|_{\Sigma^{k+1}}^2
+ \eps_2^2 \big( (2+\eps_1^2)\|u\|_{\Sigma^{k+1}}^2 + \eps_1^{-2}\|x^2 u^{(k-1)}\|_{L^2(\Gc)}^2 \big)
+ \eps_2^{-2}\|x^3 u^{(k-2)}\|_{L^2(\Gc)}^2,
\]
so that
\[
(1-\eps_2^2\eps_1^{-2})\|x^2 u^{(k-1)}\|_{L^2(\Gc)}^2
\le 4 \|u\|_{\Sigma^{k+1}}^2 + \eps_2^2(2+\eps_1^2)\|u\|_{\Sigma^{k+1}}^2 + \eps_2^{-2}\|x^3 u^{(k-2)}\|_{L^2(\Gc)}^2.
\]
Hence,
\[
\|x^2 u^{(k-1)}\|_{L^2}^2
\le (1-\eps_2^2\eps_1^{-2})^{-1} 
\Big(\big(4+\eps_2^2(2+\eps_1^2)\big)\|u\|_{\Sigma^{k+1}}^2
+\eps_2^{-2}\|x^3 u^{(k-2)}\|_{L^2(\Gc)}^2\Big),
\]
provided $1-\eps_2^2\eps_1^{-2}>0$. Next, combining this inequality with line $3$ and isolating
the term $\|x^3 u^{(k-2)}\|_{L^2(\Gc)}^2$, we obtain
\begin{multline*}
\|x^3 u^{(k-2)}\|_{L^2(\Gc)}^2
\le \big(1-\eps_3^2\eps_2^{-2}(1-\eps_2^2\eps_1^{-2})^{-1}\big)^{-1}
\Big(6\|u\|_{\Sigma^{k+1}}^2
+\eps_3^2(1-\eps_2^2\eps_1^{-2})^{-1}\big(4+\eps_2^2(2+\eps_1^2)\big)\|u\|_{\Sigma^{k+1}}^2 \\
+\eps_3^{-2}\|x^4 u^{(k-3)}\|_{L^2(\Gc)}^2\Big),   
\end{multline*}
provided $1-\eps_3^2\eps_2^{-2}(1-\eps_2^2\eps_1^{-2})^{-1}>0$. 

Iterating this process,
we find that, for each $l=2,\dots,k-1$, there exists a positive constant 
$C_l(\eps)=C_l(\eps_1,\eps_2,\dots,\eps_l)$ such that
\begin{multline*}
\|x^l u^{(k+1-l)}\|_{L^2(\Gc)}^2
\le C_l(\eps)\|u\|_{\Sigma^{k+1}}^2 \\
+\big(1-\eps_l^2\eps_{l-1}^{-2}\big(1-\eps_{l-1}^2\eps_{l-2}^{-2}
(1-\dots (1-\eps_2^2\eps_1^{-2})^{-1})^{-1}\dots\big)^{-1}\eps_l^{-2}\|x^{l+1} u^{(k-l)}\|_{L^2(\Gc)}^2,
\end{multline*}
provided $\eps_1,\eps_2,\dots,\eps_k$ can be chosen so that
\begin{equation}\label{eq:induction_on_eps}
1-\eps_l^2\eps_{l-1}^{-2}\big(1-\eps_{l-1}^2\eps_{l-2}^{-2}
(1-\dots (1-\eps_2^2\eps_1^{-2})^{-1})^{-1}>0 \quad \forall l=2,\dots,k.
\end{equation}
It is easy to verify that the following choice works:
\[
\eps_1=1, \quad \eps_l=2^{-(l-1)} \ \forall l=2,\dots,k.
\]
The argument stops at the last step, when $l=k$, where one merely obtains
\[
\|x^k u^{(1)}\|_{L^2(\Gc)}^2 \le C_k(\eps)\|u\|_{\Sigma^{k+1}}^2,
\]
for a positive constant $C_k(\eps)=C(\eps_1,\eps_2,\dots,\eps_k)$.
This final estimate allows one to return iteratively to all previous estimates and close them 
to obtain~\eqref{eq:induction_conclusion}.
\end{proof}

We now show that weighted spaces are preserved by the flow of~\eqref{eq:nls}. 
\begin{lemma} \label{lem:xk}
Let $k \in \N$. Let $u_0 \in \Sigma^2 \cap D(H_\gamma)$, $t_0 \in \R$ and let $u$ be the maximal solution of~\eqref{eq:nls} with $u(t_0) = u_0$. If $x^{k/2} u_0 \in L^2(\Gc)$ then $x^{k/2} u(t) \in L^2(\Gc)$ for all $t \in (T_{\min},T_{\max})$. Moreover $\|x^{\frac k 2}u\|_{L^2(\Gc)}$ is locally bounded in $(T_{\min},T_{\max})$.
\end{lemma}

\begin{proof}
We prove the result by induction. Let $k \geq 1$ and assume that the result is proved for $k-1$.
Let $I \subset (T_{\min},T_{\max})$ be a compact interval.

Let $\chi \in C^\infty([0,\infty) , [0,1])$ be a cut-off function such that $\chi(x)=1$ for $x\in[0,1]$ and $\chi(x)=0$ for $x\geq 2$. For $n \in \N^*$ and $x \in \Gc$ we set 
\[
\phi_n(x) = \left( \chi\left( \frac {x_j}n \right) x_j^k \right)_{1\leq j \leq N}.
\]
We have 
\begin{align*}
\partial_t \int_\Gc \phi_n |u(t)|^2 \diff x 
& = 2 \int_\Gc \phi_n \re \big( \bar u(t) u_t (t) \big) \diff x\\
& = 2 \int_\Gc \phi_n \im \big(\bar u(t) \big( H_\gamma u(t) - |u(t)|^4 u(t) \big) \big) \diff x\\
& =  -2 \int_\Gc \phi_n \im \big(\bar u(t)  \partial_{xx} u(t) \big) \diff x\\
& = 2 \int_\Gc \phi_n' \im \big(\bar u(t)  \partial_{x} u(t) \big) \diff x.
\end{align*}
By the Cauchy-Schwarz inequality we have on $I$
\[
\left| \partial_t \int_\Gc \phi_n |u(t)|^2 \diff x  \right| \leq 2 \|\phi_n' u\|_{L^2(\Gc)} \|\partial_x u\|_{L^2(\Gc)}.
\]
By the support properties of $\chi$, we have 
\[
|\phi_n'(x)|\leq \frac1n\left|\chi'\left(\frac xn\right)\right|x^k+\chi\left(\frac xn\right)x^{k-1}\leq 2\norm{\chi'}_{L^\infty}+x^{k-1}.
\]
Therefore, by induction assumption, there exists $C >0$ such that 
\[
\left| \partial_t \int_\Gc \phi_n |u(t)|^2 \diff x  \right| \leq C.
\]
Then for $t \in I$ we have 
\[
\int_\Gc \phi_n |u(t)|^2 \diff x \leq \|x^k u_0\|_{L^2(\Gc)} + C|I|.
\]
By the monotone convergence theorem we finally get 
\[
\int_\Gc x^k |u(t)|^2 \diff x \leq \|x^k u_0\|_{L^2(\Gc)} + C|I|,
\]
which proves the desired result. 
\end{proof}

Applying Lemma~\ref{lem:xk} we get with Lemma~\ref{lem:sigma_k_est} the following result in $\Sigma^2$.

\begin{proposition}
Let $u_0 \in D(H_\gamma)$, $t_0 \in \R$ and let $u$ be the maximal solution of~\eqref{eq:nls} with $u(t_0) = u_0$. If $u_0 \in \Sigma^2$ then $u\in C\left((T_{\min},T_{\max}),\Sigma^2\right)$.
\end{proposition}

\begin{proof}
    Applying Lemma~\ref{lem:xk} we get with Lemma~\ref{lem:sigma_k_est} that $u(t)\in \Sigma^2(\Gc)$ for all $t\in(T_{\min},T_{\max})$. To obtain continuity in $\Sigma^2(\Gc)$, it suffices to repeat the proof of Lemma~\ref{lem:xk} with $u(t)-u(t_0)$ instead of $u(t)$.
\end{proof}

\subsection{Dependency on the initial data}

We now give the proof of the modified continuous dependency property presented in Lemma~\ref{lem:conv-L2-solution}. 

\begin{proof}[Proof of Lemma~\ref{lem:conv-L2-solution}]
Since $  u_n$ is solution of~\eqref{eq:nls}, it verifies the Duhamel formula. For any $t\in J$, we have
\[
  u_n(t)=U(t)  u_{n,0}+i\int_{t_0}^t U(t-s)|  u_n(s)|^4  u_n(s)ds,
\]
where by $U(t)$ we denote the Schr\"odinger propagator on the graph, i.e.~$U(t)=e^{itH_\gamma}$.
A similar formula holds for $  u$. Therefore, 
\[
  u_n(t)-  u(t)=U(t)  (  u_{n,0}-  u_0)+i\int_{t_0}^t U(t-s)(|  u_n(s)|^4  u_n(s)-|  u(s)|^4  u(s))ds.
\]
Since $U$ is an isometry on $L^2(\mathcal G)$, we have 
\[
\norm{U(t)(  u_{n,0}-  u_0)}_{L^2(\mathcal G)}= 
\norm{  u_{n,0}-  u_0}_{L^2(\mathcal G)}.
\]
In addition, for the nonlinear part, we have
\begin{align*}
    \norm*{\int_{t_0}^t U(t-s)(|  u_n(s)|^4  u_n(s)-|  u(s)|^4  u(s))ds}_{L^2(\mathcal G)}
&\leq \int_{t_0}^t\norm*{|  u_n(s)|^4  u_n(s)-|  u(s)|^4  u(s)}_{L^2(\mathcal G)}\\
&\leq 
   C M\int_{t_0}^t\norm*{  u_n(s)-  u(s)}_{L^2(\mathcal G)}ds,
\end{align*}
where
\[
M:= \|  u\|_{L^\infty(J,H^1(\mathcal G))}^4+\sup_{n \in \N^*} \|  u_n\|_{L^\infty(J,H^1(\mathcal G))}^4<\infty.
\]
Therefore, 
\[
\norm*{  u_n(s)-  u(s)}_{L^2(\mathcal G)}\leq \norm{  u_{n,0}-  u_0}_{L^2(\mathcal G)}+C M\int_{t_0}^t\norm*{  u_n(s)-  u(s)}_{L^2(\mathcal G)}ds.
\]
The desired conclusion follows for Gronwall's argument. 
\end{proof}

\section{Linearized operators}\label{generalized_kernel.sec}

\newcommand{\Lop}{L}
\newcommand{\Lf}{\mathsf{L}}

In this section we record some useful properties of the linearized operators 
which will appear in our analysis (see e.g.~\eqref{eq:lin_op_h}).
We shall work here with {\em real-valued} radial functions, hence $L^2_\mathrm{rad}(\Gc)\equiv L^2_\mathrm{rad}(\Gc,\R)$ and $H^1_\mathrm{rad}(\Gc)\equiv H^1_\mathrm{rad}(\Gc,\R)$
for the rest of this section.

For $\beta \in (-N,N)$, let us define on $L^2_{\mathrm{rad}}(\Gc)$ the following operators, with domain $D_\beta\cap L^2_{\mathrm{rad}}(\Gc)$ (see \eqref{dom:Hgamma}):
\begin{equation*}
   \Lop_{+,\beta}= H_\beta+1-5 Q_\beta^4,\quad \Lop_{-,\beta} = H_\beta+1- Q_\beta^4.
\end{equation*}

We denote by $\Lf_{\pm,\beta}$ the corresponding bounded operators from $H^1_{\mathrm{rad}}(\mathcal G)$ to $H^1_{\mathrm{rad}}(\mathcal G)^\star$, such that 
\[
\dual{\Lf_{\pm ,\beta} \phi}{\psi} = \scalar{\Lop_{\pm,\beta}\phi}{\psi}_{L^2}, \quad \phi \in D_\beta, \ \psi \in H^1_{\mathrm{rad}}(\mathcal \Gc). 
\]
For instance for $\Lf_{+,\beta}$ we have, for $\phi,\psi \in H^1_{\mathrm{rad}}(\Gc)$,
\[
\dual{\Lf_{+,\beta} \phi}{\psi} = \scalar{\phi'}{\psi'}_{L^2} + \beta \phi(0) \overline{\psi(0)} + \scalar{\phi-5Q_\beta^4 \phi}{\psi}_{L^2}.
\]

\subsection{Unperturbed linearized operators}\label{sec:unpert_op}

In this subsection, we establish some results about the unperturbed linearized operators $\Lop_\pm = \Lop_{\pm,0}$.

We denote by $\sigma(A)$ and $\sigma_\textup{ess}(A)$ the spectrum and essential spectrum, respectively, of a linear operator
$A$ on $L^2_{\mathrm{rad}}(\Gc)$. The following spectral properties of the operators $\Lop_\pm$ are well-known 
in the context of radial functions on the line
(see e.g.~\cite{ChGuNaTs07,We85} and references therein)
and it is straightforward to transpose them to
$L^2_{\mathrm{rad}}(\Gc)$.
We recall that $\Lambda$ is the generator of dilations on $\Gc$ (see the definition in~\eqref{eq:dilation}).

\begin{lemma}\label{nondeg.lem}
The operators $\Lop_\pm$ have the following properties.
\begin{enumerate}[label={\upshape(\roman*)}]
\item $\Lop_\pm$ are selfadjoint and bounded from below.
\item $\sigma_\textup{ess}(\Lop_\pm)=[1,\infty)$.
\item $-8$ is the only eigenvalue of $\Lop_+$, with
$\ker(\Lop_++8I)=\Span(  Q^3)$.
\item $0$ is the only eigenvalue of $\Lop_-$, with $\ker(\Lop_-)=\Span(  Q)$. 
\item Setting $  \rho=\Lop_+^{-1}(  y^2  Q)$, we
have the relations
\begin{equation*}\label{kercycle}
\Lop_-  Q=0, \quad \Lop_+\Lambda   Q=-2  Q, \quad \Lop_-  y^2  Q=-4\Lambda   Q, \quad \Lop_+  \rho=  y^2  Q.
\end{equation*}
\end{enumerate}\label{Lpm.lem}
\end{lemma}

From these results on $\Lop_\pm$, we deduce the following properties of $\Lf_\pm$.

\begin{proposition} \label{prop:Ran-Lpm}
The operators $\Lf_\pm$ have the following properties.
\begin{enumerate}[label={\upshape(\roman*)}]
\item $\Lf_+ : H^1_{\mathrm{rad}}(\Gc) \to H^1_{\mathrm{rad}}(\Gc)^\star$ is bijective.
\item $\ker(\Lf_-) = \Span(  Q)$ and 
$\rge(\Lf_-) = \{   \f \in H^1_{\mathrm{rad}}(\Gc)^\star \, : \,   \f(  Q) = 0 \}$.

\end{enumerate}
\end{proposition}

\begin{proof}
We have $\Span(  Q) = \ker(\Lop_-) \subset \ker(\Lf_-)$, and if 
$  \varphi \in \ker(\Lf_-)$ we have $  \varphi \in D(\Lop_-)$ and $\mathsf{\Lf_-}   \varphi = 0$. 
This proves that $\ker(\Lf_-) = \Span(  Q)$. 

Since $\Lf_- = (\mathsf{Id}_{H^1_{\mathrm{rad}}(\Gc)^\star} - K) (-\partial_y^2 + 1)$ where $K =   Q^4 (-\partial_y^2 + 1)^{-1}$ is a compact operator on $H^1_{\mathrm{rad}}(\Gc)^\star$, its range is closed. 
Then $\rge(\Lf_-) = \ker(\Lf_-)^\bot$ (with $\perp$ denoting here $H^1(\Gc)^\star$ -- $H^1(\Gc)$ orthogonality) and the second statement of the proposition follows. 

The first statement about $\Lf_+$ is proved similarly.
\end{proof}

Next, we give some useful integral identities.

\begin{lemma}\label{intid.lem}
Let $  Q$, $\Lambda   Q$ and $  \rho$ be as in Lemma~\ref{prop:Ran-Lpm}. Then:
\begin{enumerate}[label={\upshape(\roman*)}]
\item $\int_\Gc   y^2  Q\Lambda   Q \diff   y
= - \int_\Gc   y^2  Q^2 \diff   y$;
\item $\int_\Gc   Q\Lambda   Q \diff  y = 0$;
\item $\int_\Gc   Q  \rho \diff   y = \frac12 \int_\Gc   y^2  Q^2 \diff   y$. 
\end{enumerate}
\end{lemma}

\begin{proof}
(i) For real parameters $\mu>-1$ and $r\geq1$, we will show that
\begin{equation}\label{genintid}
\int_\Gc   y^\mu   Q^r \Lambda   Q \diff   y
= \left(\frac{1}{2}-\frac{\mu+1}{r+1}\right) 
\int_\Gc   y^\mu   Q^{r+1}\diff   y,
\end{equation}
from which (i) follows. Now, to prove~\eqref{genintid}, we only need
to show that 
\begin{equation}\label{auxintid}
\int_\Gc   y^\mu   Q^ry\,  Q_y \diff   y 
= -\frac{\mu+1}{r+1}\int_\Gc   y^\mu   Q^{r+1} \diff   y.
\end{equation}
Integrating by parts (no vertex terms arise since $\mu+1>0$), we have
\begin{align*}
\int_\Gc   y^{\mu+1}    Q^r  Q_y \diff   y
&= -\int_\Gc [(\mu+1)  Q^r  y^\mu + r  y^{\mu+1}  Q^{r-1}  Q_y]  Q \diff   y \\
&= - (\mu+1) \int_\Gc   y^\mu   Q^{r+1} \diff   y-r\int_\Gc   Q^r  Q_y  y^{\mu+1} \diff   y,
\end{align*}
which is equivalent to~\eqref{auxintid}. This completes the proof of~\eqref{genintid}.

(ii) The second statement follows directly from~\eqref{genintid} with $\mu=0$ and $r=1$, but 
the following argument is more instructive.
Since the $L^2$ scaling $  Q_\lambda(y)=\lambda^{\frac{1}{2}}  Q(\lambda y)$ leaves
the $L^2$ norm invariant, we have that
$$
0=\frac{\dif}{\dif\lambda}\Vert   Q_\lambda\Vert_{L^2(\mathcal G)}^2
=2\int_\Gc   Q_\lambda \frac{\partial   Q_\lambda}{\partial\lambda}
\diff   y, \quad \forall\lambda>0.
$$
The result follows by letting $\lambda=1$. 

(iii) Using Lemma~\ref{Lpm.lem} and (i), we have the identities
$$
\int_\Gc   Q  \rho\diff   y=
-\frac12\int_\Gc \Lop_+\Lambda   Q  \rho\diff   y=
-\frac12\int_\Gc \Lambda   Q \Lop_+  \rho\diff   y=
-\frac12\int_\Gc \Lambda   Q  y^2  Q\diff   y=
\frac12 \int_\Gc   y^2  Q^2\diff   y.
$$
The proof is complete.
\end{proof}

We now state well-known coercivity properties of the operators $\Lf_\pm$, which we prove 
for the reader's convenience. We start with positivity properties.

\begin{lemma}\label{positivity.lem}
Denoting by $\perp$ the orthogonality with respect to $(\cdot,\cdot)_{L^2}$,
the operators $\Lf_\pm$ satisfy the following positivity relations in $H^1_\mathrm{rad}(\mathcal G)$:
\begin{align}
\langle \Lf_-  v,  v\rangle_{H^1(\mathcal G)^\star\times H^1(\mathcal G)}
& \gtrsim \|  v\|_{H^1(\mathcal G)}^2 \ \text{ on } {  \rho}^\perp, \label{pos_rel1} \\
\langle \Lf_+  v,  v\rangle_{H^1(\mathcal G)^\star\times H^1(\mathcal G)} \label{pos_rel2}
& \gtrsim \|  v\|_{H^1(\mathcal G)}^2 \ \text{ on } \{  Q,  y^2  Q\}^\perp. 
\end{align}
\end{lemma}

\begin{proof}

To prove~\eqref{pos_rel1}, we first observe that Lemma~\ref{Lpm.lem} implies 
\begin{equation}\label{pos_rel0}
\langle \Lf_-  w,  w\rangle
 \ge \|  w\|_{L^2}^2, \ \forall   w\in Q^\perp.
\end{equation}

Let $  v \in {  \rho}^\bot$. Let $  w\in {  Q}^\perp$ and $t \in \R$ such that $  v=  w+tQ$. Since $(  Q,  \rho)_{L^2}\neq0$ (see Lemma~\ref{intid.lem}), we necessarily have 
$$
t=-\frac{(  w,  \rho)_{L^2}}{(  Q,  \rho)_{L^2}},
$$
and hence
$$
\|  v\|_{L^2}^2=\|  w\|_{L^2}^2+2t(  w,  Q)_{L^2}+t^2\|  Q\|_{L^2}^2
=\|  w\|_{L^2}^2+\frac{(  w,  \rho)_{L^2}^2}{(  Q,  \rho)_{L^2}^2}\|  Q\|_{L^2}^2
\leq \|  w\|_{L^2}^2\left(1+\frac{\|  \rho\|_{L^2}^2\|  Q\|_{L^2}^2}{(  Q,  \rho)_{L^2}^2}\right).
$$
Setting 
\[
C_1 = \left(1+\frac{\|  \rho\|_{L^2}^2\|  Q\|_{L^2}^2}{(  Q,\rho)_{L^2}^2}\right)^{-1} > 0,
\]
it then follows by~\eqref{pos_rel0} that
\begin{equation}\label{pos_rel0'}
\langle \Lf_-  v,  v\rangle
= \langle \Lf_-  w,  w\rangle \ge \|  w\|_{L^2}^2\ge C_1\|  v\|_{L^2}^2.
\end{equation}
To deduce~\eqref{pos_rel1}, we argue by contradiction. Suppose there exists a sequence
$\{  v_n\}$ in $H^1_\mathrm{rad}(\mathcal G) \cap {  \rho}^\bot$ such that $\|  v_n\|_{H^1} = 1$ for all $n$ and 
$\langle \Lf_-  v_n,  v_n\rangle\to0$ as $n\to\infty$. Then~\eqref{pos_rel0'} implies $\|  v_n\|_{L^2}\to0$, so $\|\partial_y  v_n\|_{L^2}\to1$. Then
\begin{align*}
\langle \Lf_-  v_n,  v_n\rangle
&=\|\partial_y  v_n\|_{L^2}^2+\|  v_n\|_{L^2}^2
-\int_\Gc   Q^{p-1}  v_n^2\diff   y \\
&\ge \|\partial_y  v_n\|_{L^2}^2+\|  v_n\|_{L^2}^2\big(1-\|  Q\|_{L^\infty}^{p-1}\big) \to1.
\end{align*}
This contradiction concludes the proof of~\eqref{pos_rel1}.

The proof of~\eqref{pos_rel2} follows in the same way from 
\begin{equation*}
\langle \Lf_+  v,  v\rangle \gtrsim \|  v\|_{L^2}^2 \ \text{ on } \{  Q,  y^2  Q\}^\perp.
\end{equation*}
However, the proof of this inequality is much more involved than that
of~\eqref{pos_rel0'}, see~\cite{We85}.
\end{proof}

\begin{lemma}\label{coercivity.lem}
There exist $\mu_-,\mu_+>0$ such that, for all $  v\in H^1_\mathrm{rad}(\mathcal G)$,
\begin{equation}\label{coerc1}
\langle \Lf_-  v,  v\rangle_{H^1(\mathcal G)^\star\times H^1(\mathcal G)} \ge \mu_- \|  v\|_{H^1(\mathcal G)}^2
-\mu_-^{-1}(  v,  \rho)_{L^2(\mathcal G)}^2
\end{equation}
and 
\begin{equation}\label{coerc2}
\langle \Lf_+  v,  v\rangle_{H^1(\mathcal G)^\star\times H^1(\mathcal G)} \ge \mu_+ \|  v\|_{H^1(\mathcal G)}^2
-\mu_+^{-1}\big[(  v,  Q)_{L^2(\mathcal G)}^2+(  v,  y^2  Q)_{L^2(\mathcal G)}^2\big].
\end{equation}
There exists $\mu>0$ such that 
\begin{multline}\label{coerc3}
\langle \Lf_+  v,  v\rangle_{H^1(\mathcal G)^\star\times H^1(\mathcal G)}+\langle \Lf_-  v,  v\rangle_{H^1(\mathcal G)^\star\times H^1(\mathcal G)} \\
\ge \mu \|  v\|_{H^1(\mathcal G)}^2 -\mu^{-1}\big[(  v,  Q)_{L^2(\mathcal G)}^2+(  v,  y^2  Q)_{L^2(\mathcal G)}^2+(  v,  \rho)_{L^2(\mathcal G)}^2\big].
\end{multline}
\end{lemma}

\begin{proof}
We will prove~\eqref{coerc2}.
The proof of~\eqref{coerc1} is similar and will therefore be omitted. 
Estimate~\eqref{coerc3} is a direct consequence of~\eqref{coerc1} and~\eqref{coerc2}.
We will use the same shorthand notation for inner / duality products and norms as in the proof of Lemma~\ref{positivity.lem}.

Any $  v\in H^1_\mathrm{rad}(\mathcal G)$ can be written as
$$
  v=  w+s  Q+t  y^2  Q, \quad   w\in\{  Q,  y^2  Q\}^\perp,
$$
with
$$
s=\frac {(  v,  Q)_{L^2} \|   y^2   Q\|_{L^2}^2 - (  v,  y^2  Q)_{L^2}(  Q,  y^2  Q)_{L^2}}
{\|  Q\|_{L^2}^2 \|  y^2   Q\|_{L^2}^2 - (  Q,  y^2  Q)_{L^2}^2}, \quad 
t=\frac {(  v,  y^2  Q)_{L^2} \|  Q\|_{L^2}^2 - (  v,  Q)_{L^2}(  Q,  y^2  Q)_{L^2}}
{\|  Q\|_{L^2}^2 \|  y^2   Q\|_{L^2}^2 - (  Q,  y^2  Q)_{L^2}^2}.
$$

Then
\begin{align*}
\|  v\|_{H^1}^2
&=\|  w\|_{H^1}^2+s^2\|  Q\|_{H^1}^2+t^2\|  y^2  Q\|_{H^1}^2+2st(   Q,  y^2  Q)_{H^1}
+2s(  w,  Q)_{H^1}+2t(   w,  y^2  Q)_{H^1} \\
&\le \|  w\|_{H^1}^2+s^2\|  Q\|_{H^1}^2+t^2\|  y^2  Q\|_{H^1}^2
+(s^2+t^2)|(  Q,  y^2  Q)_{H^1}| \\
&\hspace{2cm}+2s\|  w\|_{H^1}\|  Q\|_{H^1}+2t\|  w\|_{H^1}\|  y^2  Q\|_{H^1} \\
&\le \|  w\|_{H^1}^2+s^2\|  Q\|_{H^1}^2+t^2\|  y^2  Q\|_{H^1}^2
+(s^2+t^2)|(  Q,  y^2  Q)_{H^1}| \\
&\hspace{2cm}+s^2+\|  w\|_{H^1}^2\|  Q\|_{H^1}^2+t^2+\|  w\|_{H^1}^2\|  y^2  Q\|_{H^1}^2 \\
&\le \big(1+\|  Q\|_{H^1}^2+\|  y^2  Q\|_{H^1}^2\big)\|  w\|_{H^1}^2
+\big(1+\|  Q\|_{H^1}^2+|(  Q,  y^2  Q)_{H^1}|\big)s^2 \\
&\hspace{2cm}
+\big(1+\|  y^2  Q\|_{H^1}^2+|(  Q,  y^2  Q)_{H^1}|\big)t^2.
\end{align*}
Hence, there exist constants $A,B,C>0$ such that
\begin{equation}\label{norm_lowerbound}
\|  w\|_{H^1}^2 \ge A \|  v\|_{H^1}^2-B(  v,  Q)^2-C(  v,  y^2  Q)^2.
\end{equation}
Using similar calculations,~\eqref{pos_rel2} yields a constant $K>0$ such that,
for any $\eps>0$,
\begin{align*}
&\langle \Lf_+  v,  v\rangle \\
&=\langle \Lf_+  w,  w\rangle
+s^2\langle \Lf_+  Q,  Q\rangle+t^2\langle \Lf_+  y^2  Q,  y^2  Q\rangle
+2st\langle \Lf_+  Q,  y^2  Q\rangle
+2s\langle   w,\Lf_+  Q\rangle+2t\langle   w,\Lf_+  y^2  Q\rangle \\
&\ge K\|  w\|_{H^1}^2-\big(|\langle \Lf_+  Q,  Q\rangle|+|\langle \Lf_+  Q,  y^2  Q\rangle|\big)s^2
-\big(|\langle \Lf_+  y^2  Q,  y^2  Q\rangle|+|\langle \Lf_+  Q,  y^2  Q\rangle|\big)t^2 \\
&\hspace{2cm}
-2\eps^{-1}|s|\eps\|  w\|_{L^2}\|\Lf_+  Q\|_{L^2}-2\eps^{-1}|t|\eps\|  w\|_{L^2}\|\Lf_+  y^2  Q\|_{L^2} \\
&\ge K\|  w\|_{H^1}^2-\big(|\langle \Lf_+  Q,  Q\rangle|+|\langle \Lf_+  Q,  y^2  Q\rangle|\big)s^2
-\big(|\langle \Lf_+  y^2  Q,  y^2  Q\rangle|+|\langle \Lf_+  Q,  y^2  Q\rangle|\big)t^2 \\ 
&\hspace{2cm}
-\big(\eps^{-2}s^2+\eps^2\|\Lf_+  Q\|_{L^2}^2\|  w\|_{H^1}^2\big)
-\big(\eps^{-2}t^2+\eps^2\|\Lf_+  y^2  Q\|_{L^2}^2\|  w\|_{H^1}^2\big) \\
&\ge \Big[K-\eps^2\big(\|\Lf_+  Q\|_{L^2}^2+\|\Lf_+  y^2  Q\|_{L^2}^2\big)\Big]\|  w\|_{H^1}^2 \\ 
&\hspace{1cm}
-\big(\eps^{-2}+|\langle \Lf_+  Q,  Q\rangle|+|\langle \Lf_+  Q,  y^2  Q\rangle|\big)s^2
-\big(\eps^{-2}+|\langle \Lf_+  y^2  Q,  y^2  Q\rangle|+|\langle \Lf_+  Q,  y^2  Q\rangle|\big)t^2.
\end{align*}
Therefore, choosing $\eps>0$ small enough, there exist constants $K_1,K_2,K_3>0$ such that
$$
\langle \Lf_+  v,  v\rangle \ge K_1 \|  w\|_{H^1}^2 -K_2 (  v,  Q)^2
- K_3 (  v,  y^2  Q)^2.
$$
Combining this with~\eqref{norm_lowerbound} concludes the proof of~\eqref{coerc2}.
\end{proof}

\subsection{Linearized operators with vertex condition}

We collect in the following proposition the properties of $\Lop_{\pm,\beta}$ which will be useful for our analysis. 

\begin{proposition} \label{prop:spectre-Lpm} Let $\beta \in (-N,N)$. 
\begin{enumerate}[label={\upshape(\roman*)}]
\item $\Lop_{\pm,\beta}$ is selfadjoint and bounded from below.
\item $\sigma_\textup{ess}(\Lop_{\pm,\beta})=[1,\infty)$.
\item $\Lop_{+,\beta}$ has a unique negative eigenvalue $\lambda_{+,\beta}$ and the rest of the spectrum is (strictly) positive. 
\item $0 \in \sigma(\Lop_{-,\beta})$, $\ker(\Lop_{-,\beta}) = \Span(  Q_\beta)$ and the rest of the spectrum is (strictly) positive.
\end{enumerate}
\end{proposition}

\begin{proof}
The operators $\Lop_{\pm,\beta}$ are symmetric and bounded perturbations of $H_\beta$, so they are selfadjoint and bounded from below.

The operator $H_0 + 1$, with domain $D_0$, is selfadjoint and we have $\sigma(H_0+1) = \sigma_\textup{ess}(H_0+1) = [1,+\infty)$. On the other hand we have on $L^2_{\mathrm{rad}}(\Gc)$
\begin{align*}
(\Lop_{+,\beta}-i)^{-1} - (H_0+1 -i)^{-1}
& = (\Lf_{+,\beta}-i)^{-1} - (\mathsf H_0+1 -i)^{-1}\\
& = -(\Lf_{+,\beta}-i)^{-1} (\beta \delta - 5 Q_\beta^4 ) (H_0+1 -i)^{-1}.
\end{align*}
Since $(H_0+1 -i)^{-1}$ is bounded from $L^2_{\mathrm{rad}}(\Gc)$ to $D_0$, $(\beta \delta - 5 Q_\beta^4 )$ is compact from $D_0$ to $H^1_{\mathrm{rad}}(\Gc)^\star$ and $(\Lf_{+,\beta}-i)^{-1}$ is bounded from $H^1_{\mathrm{rad}}(\Gc)^\star$ to $L^2_{\mathrm{rad}}(\Gc)$, we obtain that $(\Lop_{+,\beta}-i)^{-1} - (H_0+1 -i)^{-1}$ is a compact operator on $L^2_{\mathrm{rad}}(\Gc)$. Then, by Weyl's Theorem, $\sigma_\textup{ess}(\Lop_{+,\beta}) = \sigma_\textup{ess}(\Lop_{+,0}) = [1,+\infty)$. We similarly have $\sigma_\textup{ess}(\Lop_{-,\beta}) = [1,+\infty)$.

We check that 0 is never an eigenvalue of $\Lop_{+,\beta}$. Differentiating \eqref{eq:Qbeta}, there holds  
$-(Q_\beta')'' + Q_\beta' - 5 Q_\beta^4 Q'_\beta=0$ on each edge. Now let $u  \in \ker(\Lop_{+,\beta})$. We have $(u'Q_\beta' - Q_\beta'' u)' = 0$, so there exists $\eta \in \R$ such that 
\[
\Big(\frac u{Q_\beta'}\Big)' = \frac \eta {(Q_\beta')^2}.
\]
Since an antiderivative of $(Q_\beta')^{-2}$ grows faster than $(Q_\beta')^{-1}$, this implies that $\eta = 0$, so $u$ is proportional to $Q_\beta'$. 
However, we can show that $Q_\beta'$ is not in $D_\beta$. By \eqref{eq:Qbeta}, we have $Q''_\beta(0)=Q_\beta(0)-Q_\beta(0)^5$. Then, computing $Q_\beta'$ gives $Q_\beta'(x)^2=Q_\beta(x)^2-1/3Q_\beta(x)^6$ for all $x\in\mathcal G$. In particular, combined with the jump
 condition $NQ_\beta'(0)=\beta Q_\beta(0)$ at $0$, this gives $ Q(0)^4=3\left(1-\beta^2/N^2\right)$. Hence, the jump ratio at $0$ for $Q_\beta'$ is given by
\[
\frac {Q_\beta''(0)}{Q_\beta'(0)} = \frac {Q_\beta(0)}{Q_\beta'(0)} \big(1 -Q_\beta(0)^4 \big) = \frac N \beta \left(1-3 \left(1 - \frac {\beta^2}{N^2} \right) \right) = \frac {3\beta}N - \frac {2N}{\beta} \neq \frac N \beta.
\]
Therefore, $Q_\beta'\not\in D_\beta$.
Thus, $u = 0$ and $\ker(\Lop_{+,\beta}) = \{0\}$ for all $\beta\in (-N,N)$. Since 0 is not in the essential spectrum, it is in the resolvent set of $\Lop_{+,\beta}$.

It is known and can easily be verified that $\Lop_{+,0}$ has a unique eigenvalue, equal to $-8$, with the explicit eigenfunction $Q^3$ . Since $\Lop_{+,\beta}$ depends analytically on $\beta$ (family of type B in the sense of Kato~\cite{kato}), we get by regularity of the spectrum and semi-boundedness from below that $\Lop_{+,\beta}$ has a unique negative eigenvalue for all $\beta \in (-N,N)$.

For the last statement, we see from~\eqref{eq:Qbeta} that $Q_\beta \in \ker(\Lop_{-,\beta})$ for all $\beta \in (-N,N)$. As $Q_\beta>0$, $0$ is a simple eigenvalue, so $\ker(\Lop_{-,\beta}) = \Span(Q_\beta)$. 
Moreover, it is the first eigenvalue of $\Lop_{-,\beta}$, so the rest of the spectrum is positive.
\end{proof}

For $y \in \Gc$ we set 
\[
\Lambda_\beta Q_\beta(y)= \Lambda Q(y-\tau_\beta) + \left(\tau_\beta+ \frac{\beta N}{2(N^2-\beta^2)}\right)Q'(y-\tau_\beta).
\]

\begin{lemma} \label{lem:LambdaQ}
We have $\Lambda_\beta Q_\beta \in D_\beta$ and 
\begin{equation*}
    \Lop_{+,\beta} \Lambda_\beta Q_\beta= - 2 Q_\beta.
\end{equation*}
\end{lemma}

\begin{proof}
We adapt to our setting the proof of the case $\beta = 0$. For $\abs \omega < \frac N {\abs \beta}$ and $y \in \Gc$ we set 
\[
Q^\omega_\beta(y)= \sqrt\omega Q(\omega y-\tau^\omega_\beta),\quad \tau^\omega_\beta=\frac{1}{2}\tanh^{-1}\left(\frac{\beta}{N\omega}\right).
\]
Then $Q_\beta^\omega \in D_\beta$ and the function $\Lambda_\beta  Q_\beta$ has been defined so that 
\[
\Lambda_\beta  Q_\beta=\partial_\omega Q^\omega_\beta\big|_{\omega=1}.
\]
The fact that $\Lambda_\beta  Q_\beta$ belongs to the domain $D_\beta$ can be obtained either by direct calculations, or by observing that $Q_\beta^\omega$ is smooth in $(\beta,\omega)$ 
away from $\omega=0$, which allows one to differentiate with respect to $\omega$ the condition verified by $Q_\beta^\omega$ at the vertex, 
$N\partial_{y}Q_\beta^\omega(0^+)= \beta Q_\beta^\omega(0^+)$.

Finally, we have
\begin{equation}
    \label{eq:eq_for_omega}
H_\beta Q^\omega_\beta+\omega^2 Q^\omega_\beta-|Q^\omega_\beta|^4 Q^\omega_\beta=0,
\end{equation}
which gives the result after differentiation at $\omega = 1$.
\end{proof}

\begin{lemma} \label{lem:C1exp} 
Let $\beta \in (-N,N)$. Let $\eta \in \R$, $k \in \N$ and $g \in C^k_{\exp}$. Let $u \in H^1_{\mathrm{rad}}(\Gc)$ be a solution of 
\begin{equation} \label{eq:Lug}
\Lf_{\pm,\beta} u = g + \eta \delta.
\end{equation}
Then $u \in C^{k+2}_{\exp}$. Moreover, if $\eta = 0$ then $u$ also belongs to $D_\beta$. 
\end{lemma}

\begin{proof}
Assume that $u$ is solution of \eqref{eq:Lug}. Since the quadratic form 
\[
u \mapsto \norm{u'}_{L^2(\Gc)}^2 + \beta |u(0)|^2 + \norm u_{L^2(\Gc)}^2
\]
is coercive on $H^1_{\mathrm{rad}}(\Gc)$, $u$ is the unique function in $H^1_{\mathrm{rad}}(\Gc)$ which satisfies on each edge
\[
-u'' +  u = \tilde {  g}, \quad   \tilde {  g} =   g - \mu   Q_\beta^4   u,
\]
with $\mu=-1$ or $\mu=-5$, together with the vertex condition
\[
-N u'(0)  + \beta u(0) = \eta.
\]
Then $u$ is explicitely given by 
\[
u(y) = \frac \eta {N + \beta} e^{-y} + \frac 12 \int_0^\infty e^{-\abs{y-z}} \tilde g(z) \diff z + \frac 12 \frac {N-\beta}{N+\beta} e^{-y} \int_0^\infty e^{-z} \tilde g(z) \diff z.
\]
From this expression and the fact that $Q_\beta$ is smooth and exponentially decaying on each edge, we deduce all the conclusions of the lemma.
\end{proof}

Similarly to Section~\ref{sec:unpert_op}, it follows from Proposition~\ref{prop:spectre-Lpm} that
$L_{+,\beta}$ is invertible and $L_{-,\beta}$ is invertible on $Q_\beta^\perp$.
We will denote by $R_{-,\beta}$ the unique bounded operator on $L^2(\Gc)$ which maps $Q_\beta$ to 0 and any $\f \in Q_\beta^\perp = \rge(L_{-,\beta})$ to the unique $\psi \in Q_\beta^\perp = \ker(L_{-,\beta})^\perp$ such that $L_{-,\beta} \psi = \f$. In particular, $\rge(R_{-,\beta}) \subset D_\beta$.

\begin{lemma} \label{lem:inversion-Lbeta}
Let $\nu \in \N$. Let $\beta \mapsto u_\beta$ be a smooth map from $(-N,N)$ to $L^2_{\mathrm {rad}}(\Gc)$ such that $\partial_\beta^\ell u_\beta \in C^\nu_{\exp}$ for all $\ell \in \N$ and $\beta \in (-N,N)$. 
Then the maps $\beta \mapsto \Lop_{+,\beta}^{-1} u_\beta$ and $\beta \mapsto R_{-,\beta} u_\beta$ are smooth from $(-N,N)$ to $L^2_{\mathrm {rad}}(\Gc)$. Moreover, we have $\partial_\beta^\ell (L_{+,\beta}^{-1} u_\beta),\partial_\beta^\ell (R_{-,\beta} u_\beta) \in C^\nu_{\exp}$ for all $\beta \in (-N,N)$ and $\ell \in \N$.
\end{lemma}

We introduce some notation for the proof. For $\beta \in (-N,N)$, we set $H^1_{\beta,\perp} = H^1_{\mathrm {rad}} \cap Q_\beta^\perp$ (the orthogonal complement is still understood in the sense of $L^2(\Gc)$). We denote by $\mathcal I_\beta : H^1_{\beta,\perp} \to H^1_{\mathrm {rad}}$ the natural embedding, and we set $\Lf_{-,\beta}^\perp = \mathcal I_\beta^* \Lf_{-,\beta} \mathcal I_\beta : H^1_{\beta,\perp} \to (H^1_{\beta,\perp})^\star$. The corresponding operator given by the representation theorem is the restriction $L_{-,\beta}^\perp$ of $L_{-,\beta}$ to $Q_\beta^\perp$, with domain $D_\beta \cap Q_\beta^\perp$.

We set 
\[
\mathsf R_{-,\beta} = \mathcal I_\beta (\Lf_{-,\beta}^\perp)^{-1} \mathcal I_\beta^\star : (H^1_{\mathrm {rad}})^\star \to H^1_{\mathrm {rad}}.
\]
For $u \in L^2(\Gc)$ we have $\mathsf R_{-,\beta} u = R_{-,\beta} u$.

Finally, let $\Pi_\beta$ be the orthogonal projection of $L^2(\Gc)$ onto $\Span(Q_\beta)$,
and $\Pi_\beta^\perp = 1 - \Pi_\beta$. 

\begin{proof}
We begin with $R_{-,\beta}$. Let $\beta \in (-N,N)$. There exist $\eta_0 > 0$ and $C> 0$ such that for all $\eta \in [-\eta_0,\eta_0]$
\begin{equation} \label{eq:R-bounded}
\|R_{-,\beta+\eta}\|_{\mathcal L(L^2)} \leq C , \quad  \| \mathsf R_{-,\beta+\eta} \|_{\mathcal L((H^1_{\mathrm{rad}})^\star, H^1_{\mathrm{rad}})} \leq C.
\end{equation}
In \eqref{eq:R-bounded}, the first inequality follows from the strict positivity of the spectrum of $L_{-,\beta}^\perp$. The second inequality is a direct consequence of the first one (see the argument after \eqref{pos_rel0'} in the proof of Lemma \ref{positivity.lem}).
Let $u \in L^2(\Gc)$ and $\eta \in [-\eta_0,\eta_0]$. We have 
\begin{equation} \label{eq:R}
\begin{aligned}
(R_{-,\beta+\eta} - R_{-,\beta}) u 
& = (R_{-,\beta+\eta} - R_{-,\beta}) \Pi_\beta u \\
& + \Pi_{\beta+\eta} (R_{-,\beta+\eta} - R_{-,\beta}) \Pi_\beta^\perp u \\
& + \Pi_{\beta+\eta}^\perp (R_{-,\beta+\eta} - R_{-,\beta}) \Pi_\beta^\perp u.
\end{aligned}
\end{equation}
For the first two terms, we observe that
\begin{eqnarray} \label{eq:R1}
\lefteqn{(R_{-,\beta+\eta} - R_{-,\beta}) \Pi_\beta u + \Pi_{\beta+\eta} (R_{-,\beta+\eta} - R_{-,\beta}) \Pi_\beta^\perp u} \\
\nonumber && = R_{-,\beta+\eta} \Pi_\beta u - \Pi_{\beta+\eta} R_{-,\beta} \Pi_\beta^\perp u \\
\nonumber && = R_{-,\beta+\eta} (\Pi_\beta- \Pi_{\beta+\eta}) u - (\Pi_{\beta+\eta}-\Pi_\beta) R_{-,\beta}  u.
\end{eqnarray}
For the third term in the right-hand side of \eqref{eq:R}, we have
\begin{equation} \label{eq:R2}
\begin{aligned}
\Pi_{\beta+\eta}^\perp (R_{-,\beta+\eta} - R_{-,\beta}) \Pi_\beta^\perp u
& = \Pi_{\beta+\eta}^\perp (\mathsf R_{-,\beta+\eta} - \mathsf R_{-,\beta}) \Pi_\beta^\perp u \\
& = \mathsf R_{-,\beta+\eta} \Lf_{-,\beta+\eta} (\mathsf R_{-,\beta+\eta} - \mathsf R_{-,\beta}) \Lf_{-,\beta} \mathsf R_{-,\beta} u \\
& = - \mathsf R_{-,\beta+\eta} (\Lf_{-,\beta+\eta} - \Lf_{-,\beta}) \mathsf R_{-,\beta} u \\
& = - \eta \mathsf R_{-,\beta+\eta} \delta \mathsf R_{-,\beta} u.
\end{aligned}
\end{equation}
Since the map $\beta \mapsto \Pi_{\beta}$ is smooth, we first deduce with \eqref{eq:R-bounded}-\eqref{eq:R2} that the map $\beta \mapsto R_{-,\beta}$ is continuous. Then, after dividing these equalities by $\eta$, we deduce that it is also differentiable with 
\begin{equation} \label{eq:der-R-}
\partial_\beta R_{-,\beta} = - R_{-\beta} \partial_\beta \Pi_\beta - \partial_\beta \Pi_\beta R_{-\beta} - \mathsf R_{-,\beta} \delta \mathsf R_{-,\beta}.
\end{equation}
By iteration, we prove that the map $\beta \mapsto R_{-,\beta} u_\beta$ is smooth on $(-N,N)$.

Since $\rge(\Pi_\beta) = \Span(Q_\beta)$, then $\Pi_\beta$ and all its derivatives with respect to $\beta$ leave $C^\nu_{\exp}$ invariant.

Moreover, by Lemma \ref{lem:C1exp}, the operators $R_{-,\beta}$ and $\mathsf R_{-,\beta} \delta \mathsf R_{-,\beta}$ map $C^\nu_{\exp}$ into itself. By iterating \eqref{eq:der-R-}, we deduce that all the derivatives of $R_{-,\beta}$ with respect to $\beta$ leave $C^\nu_{\exp}$ invariant.

This concludes the proof for $R_{-,\beta}$. For the usual resolvent $L_{+,\beta}^{-1}$, we replace $\Pi_\beta$ by 0 and we recover basic properties for a resolvent. In this case, the derivatives are simply given by 
\[
\partial_\beta^\ell (L_{+,\beta}^{-1}) = \Lf_{+,\beta}^{-1} \big( \delta \Lf_{+,\beta}^{-1} \big)^{\ell},
\]
and we get as before that they leave $C^\nu_{\exp}$ invariant.
\end{proof}

\section{Construction of the profile} \label{sec:profile}

\newcommand{\Pol}{\mathrm{Pol}_\kappa(b,\lambda)}

In this section, we prove Proposition~\ref{prop:profile}. The terms $P_{j,k,\beta}$ and $\alpha_{j,k,\beta}$ will be constructed by induction on $(j,k) \in \Theta_\kappa$ according to the following order. For $(j_1,k_1),(j_2,k_2) \in \mathbb Z^2$ we say that $(j_1,k_1) < (j_2,k_2)$ if either $k_1 < k_2$ or $k_1 = k_2$ and $j_1< j_2$.

Let $\lambda_0 = \frac {N}{\abs{\gamma}}$. Given an interval $J$ of $\R$, $b \in C^1(J,\R)$, $\lambda \in C^1(J,(0,\lambda_0))$ and $\alpha \in C^1(J,\R)$, we denote by $\OC$ any family  
$(  u_\beta(s))_{\beta \in (-N,N)}\subset L^2(\Gc)$ satisfying the following property:
for any $\nu \in \N$ there exist $\rho > 0$ and $C > 0$ such that, for all $s \in J$ there holds $ {u}_{\gamma\lambda(s)}(s) \in C^\nu_{\exp}$ and
\begin{equation} \label{def:Obla}
\norm{  u_{\gamma\lambda(s)}}_{C^\nu_{\exp,\rho}} \leq C \lambda \left(\left|b+ \frac{\lambda_s}{\lambda}\right|+\left|b_s+b^2-\alpha\right| \right)
+ C(b^2+\lambda)^{\kappa}.
\end{equation}

With this notation, the main estimate~\eqref{remainder_est} of Proposition~\ref{prop:profile} reads
\begin{equation} \label{remainder_est-2}
\Psi_\kappa = \OC.
\end{equation}

\begin{lemma}
Let $(  P_{j,k,\beta})$ and $(\alpha_{j,k,\beta})$ be as in Proposition~\ref{prop:profile}~(i)~and~(ii). Let $P$ and $\alpha$ be defined by~\eqref{def:P} and~\eqref{def:alpha}. Let $b \in C^1(J,\R)$, $\lambda \in C^1(J,(0,\lambda_0))$ and $\alpha \in C^1(J,\R)$, and let $\Psi_\kappa$ be defined by~\eqref{def:Psi-K}. Then we can write 
\begin{equation} \label{eq:Psi-OC}
\Psi_\kappa = \sum_{(j,k) \in \Theta_\kappa} 
(ib)^j \lambda^k\Psi_{j,k,\gamma \lambda} + \OC,
\end{equation}
for $\Psi_{j,k,\beta}$, $(j,k,\beta) \in \Theta_\kappa \times (-N,N)$ as follows. Given $(m,k) \in \N \times \N^*$ with $m+k <  \kappa$, there exist $\tilde \Psi_{2m,k,\beta }$ and $\tilde \Psi_{2m+1,k,\beta}$ in $L^2(\Gc)$
\begin{enumerate}[label={\upshape(\roman*)}]
    \item with all derivatives with respect to $\beta$ in $C^\infty_{\exp}$,
    \item which depend explicitly on the $P_{j',k',\beta }$ and $\alpha_{j',k',\beta }$ for $(j',k') < (2m,k)$
\end{enumerate}
such that
\begin{align}
\label{eq:Psi-2m}
\Psi_{2m,k,\beta } & = -\Lop_{+,\beta } P_{2m,k,\beta } + \alpha_{2m,k,\beta } \frac {y^2}{4} Q_{\beta }  + \tilde \Psi_{2m,k,\beta },\\
\label{eq:Psi-2m1}
\Psi_{2m+1,k,\beta } & = -\Lop_{-,\beta } P_{2m+1,k,\beta } - (2m + k) P_{2m,k,\beta } + \tilde \Psi_{2m+1,k,\beta }.
\end{align}
In particular we have 
\begin{equation} \label{eq:tilde-Psi11}
\tilde \Psi_{0,1,\beta} = 0, \quad \tilde \Psi_{1,1,\beta} = - \gamma \partial_\beta Q_{\beta } =  \frac {N\gamma} {2(N^2-\beta^2)} Q_\beta'.
\end{equation}
\end{lemma}

\begin{proof}
Up to a rest of size $\OC$, 
we can always replace the derivatives $\lambda_s$ and $b_s$ by $(-\lambda b)$ and $(ib)^2 + \alpha$, respectively. For $(j,k) \in \Theta_\kappa$, we have 
\begin{eqnarray*}%
\lefteqn{i \partial_s \big( (ib)^j \lambda^k   P_{j,k,\gamma \lambda } \big)= -j (ib)^{j-1} b_s \lambda^k   P_{j,k,\gamma \lambda } + i (ib)^j k \lambda^{k-1} \lambda_s   P_{j,k,\gamma \lambda } + i\gamma (ib)^j \lambda^k \lambda_s \partial_\beta P_{j,k,\gamma \lambda} }\notag\\
&& \notag \\
&& = - j (ib)^{j+1}  \lambda^k   P_{j,k,\gamma \lambda } -j (ib)^{j-1} \alpha \lambda^k   P_{j,k,\gamma \lambda } - (ib)^{j+1} k \lambda^{k}    P_{j,k,\gamma \lambda} - \gamma (ib)^{j+1} \lambda^{k+1} \partial_\beta P_{j,k,\gamma \lambda}\notag\\ 
&& \qquad + \OC\notag\\
&& = - (j+k) (ib)^{j+1}  \lambda^k   P_{j,k,\gamma \lambda }  - \sum_{\substack{(p,q) \in \Theta_\kappa \\ p \text{ even}}} j \alpha_{p,q,\gamma \lambda} (ib)^{j-1+ p} \lambda^{k+q}   P_{j,k,\gamma \lambda } - \gamma (ib)^{j+1} \lambda^{k+1} \partial_\beta P_{j,k,\gamma \lambda }\notag\\
&& \qquad  + \OC.
\end{eqnarray*}
Notice that some terms in the sum satisfy~\eqref{def:Obla} and could be put in the rest. 
On the other hand, there exists a family $(\Phi_{j,k,\beta})$ with $\Phi_{j,k,\beta} \in L^2(\Gc)$ 
which depends only on $Q_{\beta}$ and the $P_{j,k,\beta}$, such that
\begin{equation}\label{eq:expansion_of_f}
 \abs{P}^4 P = Q_{\gamma \lambda }^5 + \sum_{(j,k) \in \Theta_\kappa} (ib)^j \lambda^k \Phi_{j,k,\gamma \lambda } + \OC.   
\end{equation}
In particular, $\Phi_{0,1,\beta} = 5 Q_{\beta}^4 P_{0,1,\beta}$.
More generally, we have 
\[
\Phi_{j,k,\beta } = \big( 3 + 2 (-1)^j \big) Q_{\beta }^4 P_{j,k,\beta }  + \tilde \Phi_{j,k,\beta },
\]
for some $\tilde \Phi_{j,k,\beta} \in L^2(\Gc)$ whose derivatives with respect to $\beta$ are in $C^\infty_{\exp}$ and which only depends on the $P_{j',k',\beta }$ with $(j',k') <  (j,k)$. In particular, $\tilde \Phi_{0,1,\beta} = \tilde \Phi_{1,1,\beta} = 0$.

We set $\Theta_\kappa^0 = \Theta_\kappa \cup \{(0,0)\}$ and $P_{0,0,\gamma \lambda } = Q_{\gamma \lambda }$. We also set $P_{j,0,\gamma \lambda } = 0$ for $j \in \N^*$.
Then we have~\eqref{eq:Psi-OC} where, for $(j,k,\beta) \in \Theta_\kappa \times (-N,N)$, 
\begin{equation*} %
\begin{aligned}
\Psi_{j,k,\beta } 
& = -(j-1+k) P_{j-1,k,\beta } + \sum_{\substack{(p,q) \in \Theta_\kappa \\ p \text{ even}}} (j+1-p) \alpha_{p,q} P_{j+1-p,k-q,\beta } - \gamma \partial_\beta P_{j-1,k-1,\beta }\\
& + H_{\beta }P_{j,k,\beta } - P_{j,k} + \big( 3 + 2 (-1)^j \big) Q_{\beta }^4 P_{j,k,\beta }  + \tilde \Phi_{j,k,\beta } + \sum_{\substack{(p_1,q_1) \in \Theta_\kappa , (p_2,q_2)\in \Theta_\kappa^0 \\ p_1+p_2 = j, q_1+q_2 = k \\ p_1 \text{ even}}} \alpha_{p_1,q_1,\beta } \frac {y^2}{4} P_{p_2,q_2,\beta }.
\end{aligned}
\end{equation*}
It can be checked that this is indeed of the form~\eqref{eq:Psi-2m}-\eqref{eq:Psi-2m1}, and that \eqref{eq:tilde-Psi11} holds when $m = 0$.
\end{proof}

We now show that $(P_{j,k,\beta})$ and $(\alpha_{j,k,\beta})$ can be defined so that \eqref{remainder_est} holds. We construct $P_{j,k,\beta}$ and $\alpha_{j,k,\beta}$ by induction, so that the whole sum in~\eqref{eq:Psi-OC} vanishes, which implies \eqref{remainder_est-2}. To this aim, we shall use the expressions~\eqref{eq:Psi-2m}-\eqref{eq:Psi-2m1} and the properties of the operators $\Lf_{\pm,\beta}$. In particular, we will obtain the explicit 
formula for $\alpha_{0,1,0}$, as stated in part~(iii).

\begin{proof}[Proof of~(iii) and~\eqref{remainder_est}]
Let $(m,k) \in \N \times \N^*$ with $m+k <  \kappa$. Assume that we have defined $\alpha_{j',k',\beta }$ (if $j'$ is even) and $P_{j',k',\beta }$ for all $(j',k') \in \Theta_\kappa$ with $(j',k') < (2m,k)$.

By~\eqref{eq:Psi-2m} and Proposition~\ref{prop:spectre-Lpm}, for any choice of $\alpha_{2m,k,\beta } \in \R$ there exists $P_{2m,k,\beta } \in D_{\beta}$ such that $\Psi_{2m,k,\beta} = 0$. Moreover $\partial_\beta^\ell P_{2m,k,\beta} \in C^\infty_{\exp}$ for all $\ell \in \N$ by Lemma~\ref{lem:inversion-Lbeta}.

We choose $\alpha_{2m,k,\beta }$ so that (see~\eqref{eq:Psi-2m1} and Proposition~\ref{prop:spectre-Lpm} again)
\begin{equation} \label{eq:P2mk}
(2m+k) P_{2m,k,\beta } - \tilde \Psi _{2m+1,k,\beta } \in \mathrm{Ran}(\Lf_{-,\beta }) = 
\Span(Q_{\beta })^\bot.
\end{equation}
By Lemma~\ref{lem:LambdaQ} and the selfadjointness of $\Lf_{+,\beta }$, this condition reads
\begin{equation*}
\frac{2m+k}2 \scalar{\Lf_{+,\beta } P_{2m,k,\beta }}{\Lambda_{\beta } Q_{\beta }} + \scalar{ \tilde \Psi _{2m+1,k,\beta }}{ Q_{\beta }} = 0,
\end{equation*}
which gives
\begin{align*}
\alpha_{2m,k,\beta } 
 = - \frac {4} {\scalar{y^2 Q_{\beta }}{ \Lambda_{\beta } Q_{\beta }}} 
 \left(\frac 2 {2m+k}  \scalar{ \tilde \Psi_{2m+1,k,\beta}} {Q_{\beta }} 
 + \scalar{\tilde \Psi_{2m,k,\beta }}{\Lambda_{\beta } Q_{\beta }}  \right) . 
\end{align*}
We can then choose $P_{2m+1,k,\beta } \in D_\beta \cap \Span(Q_{\beta })^\bot$ such that $\Psi_{2m+1,k,\beta } = 0$. And $\partial_\beta^\ell P_{2m+1,k,\beta} \in C^\infty_{\exp}$ for all $\ell \in \N$ by Lemma~\ref{lem:inversion-Lbeta}. We have thus constructed $(P_{j,k,\beta})$ and $(\alpha_{j,k,\beta})$ 
so that~\eqref{remainder_est} holds.

We finally check (iii). For $m = 0$ and $k=1$, \eqref{eq:tilde-Psi11} yields
\[
\alpha_{0,1,\beta } =  -\frac {4N \gamma \scalar{ Q_{\beta }'}   {Q_{\beta}}} {(N^2-\beta^2)\scalar{y^2 Q_{\beta }}{ \Lambda_{\beta } Q_{\beta }}}.
\]
Hence, by Lemma~\ref{intid.lem},
\[
\alpha_{0,1,0} =  
-\frac {4\gamma \scalar{ Q'} {Q}} {N \scalar{y^2 Q}{ \Lambda Q}} =   
- 2 \gamma \frac {Q(0)^2}{\|yQ\|_{L^2(\Gc)}^2}=\alpha_\star,
\]
as expected. 
\end{proof}

\begin{proof}[Proof of~\eqref{deriv_energy} and~\eqref{energy_nls_dynsys}-\eqref{asympt_nrj}]
We start by proving~\eqref{deriv_energy}. Recall that the energy $E$ is defined as
\begin{align*}
E(u)
&=\frac12\int_{\Gc}|u_x|^2\diff x+\frac{\gamma}{2}|u(0)|^2-\frac16\int_{\Gc}|u|^6\diff x,
\quad u\in H^1_D(\Gc).
\end{align*}
Letting 
\[
u(x)=\lambda^{-1/2}v(\lambda^{-1}x), \quad \lambda>0,
\]
we have that $u\in D(H_\gamma) \Leftrightarrow v\in D(H_{\gamma\lambda})$ and 
the change of variables $y=\lambda^{-1}x$ yields
\begin{align}\label{eq:ch_var_energy}
E(u)
&=\frac{1}{\lambda^2}\left[\frac{1}{2}\int_{\Gc}|v_y(y)|^2\diff y+\frac{\gamma\lambda}{2}|v(0)|^2
-\frac{1}{6}\int_{\Gc}|v(y)|^6\diff y \right] =: \frac{1}{\lambda^2}\wt E(\lambda,v).
\end{align}
Note that
$\wt E$ is Fr\'echet differentiable with respect to its second variable 
at any $v \in D_{\gamma \lambda}$, with derivative 
\begin{equation}\label{eq:id_E_H}
D_p \wt E(\lambda,v) = -H_{\gamma\lambda}v-f(v).
\end{equation}

We now consider 
\begin{equation}\label{eq:ch_var_small_p}
p:=e^{-ib\frac{y^2}{4}} P \in D(H_{\gamma\lambda}).
\end{equation}
Then
$$
\tilde P = \lambda^{-1/2}e^{i(\theta-b\frac{y^2}{4})}P
=\lambda^{-1/2}e^{i\theta}p
$$
and we deduce from~\eqref{eq:ch_var_energy} (with $v=e^{i\theta}p$) that
\begin{equation}\label{eq:lambda^2_energy}
E(\tilde P)=\frac{1}{\lambda^2}\wt E(\lambda,p).
\end{equation}
Hence,
\begin{align*}
\frac{\dif}{\dif s}E(\tilde P(s))
&=\Big(-2\frac{\lambda_s}{\lambda^3}\Big)\wt E(\lambda,p)
+\frac{1}{\lambda^2}\Big[D_\lambda \wt E(\lambda,p)\lambda_s 
+ ( D_{p}\wt E(\lambda,p),p_s)\Big] \\
&=\frac{1}{\lambda^2}\Big[-2\frac{\lambda_s}{\lambda}\wt E(\lambda,p)
+\frac{\gamma\lambda}{2}\frac{\lambda_s}{\lambda}|p(0)|^2
+( D_{p}\wt E(\lambda,p),p_s)\Big].
\end{align*}
To compute the last term, we shall use the equation satisfied by $p$, 
which we derive from~\eqref{def:Psi-K}, using~\eqref{eq:id_E_H}:
\begin{align*}
ip_s
&=-H_{\gamma\lambda}p+p-f(p)+i\frac{\lambda_s}{\lambda}\Lambda p
-i\Big(\frac{\lambda_s}{\lambda}+b\Big)\Lambda p+(b_s+b^2-\alpha)\frac{y^2}{4}p
+e^{-ib\frac{y^2}{4}}\Psi_\kappa \\
&=D_p \wt E(\lambda,p)+p+i\frac{\lambda_s}{\lambda}\Lambda p
-i\Big(\frac{\lambda_s}{\lambda}+b\Big)\Lambda p+(b_s+b^2-\alpha)\frac{y^2}{4}p
+e^{-ib\frac{y^2}{4}}\Psi_\kappa.
\end{align*}
Since 
$( D_p \wt E(\lambda,p), iD_p\wt E(\lambda,p))=( D_{p}\wt E(\lambda,p),ip )=0$, 
it follows that
\begin{align*}
( D_{p}\wt E(\lambda,p),p_s )
=& \, \frac{\lambda_s}{\lambda}( D_{p}\wt E(\lambda,p),\Lambda p ) \\
& - \Big(\frac{\lambda_s}{\lambda}+b\Big)( D_{p}\wt E(\lambda,p),\Lambda p ) 
-\Big( D_{p}\wt E(\lambda,p),i(b_s+b^2-\alpha)\frac{y^2}{4}p \Big) \\
& - \Big( D_{p}\wt E(\lambda,p),ie^{-ib\frac{y^2}{4}}\Psi_\kappa \Big).
\end{align*}
Hence,
\begin{align}\label{eq:1_large_term}
\frac{\dif}{\dif s}E(\tilde P(s))
=& \,\frac{1}{\lambda^2}
\bigg[
\frac{\lambda_s}{\lambda}\Big(-2\wt E(\lambda,p)
+\frac{\gamma\lambda}{2}|p(0)|^2+( D_{p}\wt E(\lambda,p),\Lambda p ) \Big) \notag \\
&- \Big(\frac{\lambda_s}{\lambda}+b\Big)( D_{p}\wt E(\lambda,p),\Lambda p ) 
-\Big( D_{p}\wt E(\lambda,p),i(b_s+b^2-\alpha)\frac{y^2}{4}p \Big) \nonumber \\
& - \Big( D_{p}\wt E(\lambda,p),ie^{-ib\frac{y^2}{4}}\Psi_\kappa \Big) 
\bigg].
\end{align}
To conclude the proof of~\eqref{deriv_energy}, we will show that 
\begin{equation}\label{eq:3_terms_cancelout}
-2\wt E(\lambda,p)
+\frac{\gamma\lambda}{2}|p(0)|^2+( D_{p}\wt E(\lambda,p),\Lambda p ) = 0.
\end{equation} 
Estimate~\eqref{deriv_energy} is then a direct consequence of~\eqref{remainder_est}.
Using again~\eqref{eq:id_E_H},
the left-hand side of~\eqref{eq:3_terms_cancelout} reads
\begin{equation}\label{eq:to_cancel}
-\intg|p_y|^2\diff y - \gamma\lambda|p(0)|^2 + \frac13 \intg|p|^6\diff y
+\frac{\gamma\lambda}{2} |p(0)|^2
-(H_{\gamma\lambda}p,\Lambda p)
-\re\intg|p|^4p\overline{\Lambda p} \diff y. 
\end{equation}
An integration by parts shows that
\begin{equation*}
\re\intg p_{yy}y\overline{p_y}\diff y=-\frac12 \intg|p_y|^2\diff y,
\end{equation*}
so that
\begin{align}\label{eq:cancel_1}
( H_{\gamma\lambda}p,\Lambda p)
=\Big( H_{\gamma\lambda}p,\frac12 p + yp_y\Big) 
&=\frac12\intg|p_y|^2\diff y+\frac{\gamma\lambda}{2}|p(0)|^2 
-\re\intg p_{yy}y\overline{p_y}\diff y \notag \\
&=\intg|p_y|^2\diff y+\frac{\gamma\lambda}{2}|p(0)|^2.
\end{align}
Another integration by parts yields
\begin{equation*}
\re\intg|p|^4py\overline{p_y} \diff y = -\frac16\intg|p|^6\diff y
\end{equation*}
and it follows that
\begin{equation}\label{eq:cancel_2}
\re\intg|p|^4p\overline{\Lambda p} \diff y 
=\frac12 \intg|p|^6\diff y + \re\intg|p|^4py\overline{p_y} \diff y
= \frac13 \intg|p|^6\diff y.
\end{equation}
\eqref{eq:cancel_1} and~\eqref{eq:cancel_2} show that all terms
in~\eqref{eq:to_cancel} cancel out, which completes the proof of~\eqref{deriv_energy}.

We next prove~\eqref{energy_nls_dynsys}-\eqref{asympt_nrj}. First, 
by~\eqref{eq:ch_var_small_p} and~\eqref{eq:lambda^2_energy},
\begin{align*}
\lambda^2 E(\tilde P) &=\wt E\big(\lambda,e^{-ib\frac{y^2}{4}}P\big) \\
&=\frac12\intg \Big( b^2\frac{y^2}{4}|P|^2+|P_y|^2+\re(ibyP_y\bar{P})\Big)\diff y
+\frac{\gamma\lambda}{2}|P(0)|^2 - \frac16\intg |P|^6\diff y \\
&=\frac12\intg|P_y|^2\diff y + \frac{b^2}{8}\intg y^2|P|^2\diff y 
-\frac{b}{2}\im \intg yP_y\bar{P}\diff y
+\frac{\gamma\lambda}{2}|P(0)|^2 - \frac16\intg |P|^6\diff y.
\end{align*}
We shall now plug in 
\[
P=Q_{\gamma \lambda}+\lambda Z,
\]
where, according to~\eqref{def:P},
\[
Z = \sum_{(j,k) \in \Theta_\kappa} (ib)^j \lambda^{k-1}  P_{j,k,\gamma \lambda }.
\]
We find that
\begin{align*}
\wt E\big(\lambda,e^{-ib\frac{y^2}{4}}P\big)
=&\,\frac12\intg|\partial_yQ_{\gamma\lambda}|^2\diff y+\lambda\intg\partial_yQ_{\gamma\lambda}\re\big(\bar Z_y\big)\diff y
+\frac{\lambda^2}{2}\intg|Z_y|^2\diff y \\
& +\frac{b^2}{8}\intg y^2 Q_{\gamma\lambda}^2\diff y + \frac{b^2}{4}\lambda\intg y^2 Q_{\gamma\lambda}\re\big(\bar{Z}\big)
+\frac{b^2}{8}\lambda^2\intg y^2|Z|^2\diff y \\
& -\frac{b}{2}\lambda\im \intg y\big(Z_yQ_{\gamma\lambda}+\partial_yQ_{\gamma\lambda}\bar{Z}\big)\diff y
-\frac{b}{2}\lambda^2 \im\intg y Z_y\bar{Z}\diff y \\
& +\frac{\gamma\lambda}{2}Q_{\gamma\lambda}(0)^2 + \gamma\lambda^2 Q_{\gamma\lambda}(0)\re\big(\bar{Z}(0)\big) 
+\frac{\gamma\lambda^3}{2}|Z(0)|^2 -  \frac16\intg |Q_{\gamma \lambda}+\lambda Z|^6\diff y.
\end{align*}
Since $Q_{\gamma \lambda}$ satisfies 
\begin{equation}\label{eq:ground_state_eq}
H_{\gamma\lambda}Q_{\gamma\lambda}+Q_{\gamma\lambda}-Q_{\gamma\lambda}^5=0,
\end{equation}
it is a critical point of the functional
\[
A(v):=
\frac{1}{2}\int_{\Gc}|v_y(y)|^2\diff y + \frac{\gamma\lambda}{2}|v(0)|^2 + \frac{1}{2}\int_{\Gc}|v(y)|^2\diff y
-\frac{1}{6}\int_{\Gc}|v(y)|^6\diff y, \quad v\in H^1_D(\Gc).
\]
Letting $Q_{\gamma \lambda}^{\mu}(y)=\mu^{1/2}Q_{\gamma \lambda}(\mu y), \ \mu>0$, 
we deduce that
\[
\frac{\dif}{\dif\mu}\Big\vert_{\mu=1}A(Q_{\gamma \lambda}^\mu)
=A'(Q_{\gamma \lambda})\frac{\dif Q_{\gamma \lambda}^\mu}{\dif\mu}\Big\vert_{\mu=1}=0.
\]
An explicit computation of the left-hand side of this identity yields
\begin{equation}\label{eq:nehari+pohozaev}
\frac12\intg|\partial_yQ_{\gamma\lambda}|^2\diff y+\frac{\gamma\lambda}{4}Q_{\gamma\lambda}(0)^2
- \frac16\intg Q_{\gamma\lambda}^6\diff y = 0.
\end{equation}
On the other hand, by~\eqref{eq:ground_state_eq},
\begin{align}\label{eq:transf_2terms}
\intg \partial_y Q_{\gamma\lambda}\re\big(\bar Z_y \big)\diff y 
+ \gamma\lambda Q_{\gamma\lambda}(0)\re\big(\bar Z(0) \big)
&=(H_{\gamma\lambda}Q_{\gamma\lambda},Z) \notag \\
&=(-Q_{\gamma\lambda}+Q_{\gamma\lambda}^5,Z) \notag \\
&=-(Q_{\gamma\lambda},Z) + \intg Q_{\gamma\lambda}^5\re\big(\bar Z\big) \diff y.
\end{align}
Using~\eqref{eq:nehari+pohozaev} and~\eqref{eq:transf_2terms}, the energy becomes
\begin{align}\label{eq:expansion_E_tilde}
\wt E\big(\lambda,e^{-ib\frac{y^2}{4}}P\big)
=&\,\frac{\gamma\lambda}{4}Q_{\gamma\lambda}(0)^2 
+ \frac{b^2}{8}\intg y^2 Q_{\gamma\lambda}^2\diff y -\lambda (Q_{\gamma\lambda},Z) \notag \\
&-\frac16\intg\Big(|Q_{\gamma \lambda}+\lambda Z|^6-Q_{\gamma\lambda}^6-6Q_{\gamma\lambda}^5\re\big(\lambda \bar Z\big)\Big)\diff y \notag \\
&-\frac{b}{2}\lambda\im \intg y\big(Z_yQ_{\gamma\lambda}+\partial_yQ_{\gamma\lambda}\bar{Z}\big)\diff y
-\frac{b}{2}\lambda^2 \im\intg y Z_y\bar{Z}\diff y \notag \\
&+\frac{\lambda^2}{2}\intg|Z_y|^2\diff y + \frac{b^2}{4}\lambda\intg y^2 Q_{\gamma\lambda}\re\big(\bar{Z}\big)\diff y
+\frac{b^2}{8}\lambda^2\intg y^2|Z|^2\diff y  +\frac{\gamma\lambda^3}{2}|Z(0)|^2 .
\end{align}
The first line of the above right-hand side is of order $\lambda\sim b^2$, while
the next lines are of higher orders. In the rest of this proof we write $\Pol$ for any function of $b$ and $\lambda$ of the form 
\[
\sum_{\substack{(j,k) \in \Theta_\kappa\\ j \text{ even}, \ j/2+k\ge2}}
\eta_{j,k} b^j\lambda^{k}, \quad (\eta_{j,k}) \subset \R.
\]
In particular \eqref{energy_nls_dynsys} reads
\begin{equation} \label{energy_nls_dynsys_2}
E(\tilde P(b,\lambda,\theta)) = C_Q \mathcal E_{\mathrm{mo}}(b,\lambda) + \frac 1 {\lambda^2} \big( \Pol + O((b^2 + \lambda)^\kappa) \big).
\end{equation}

Using~\eqref{eq:def_of_Q_beta} and the definition of $\alpha_\star$ (see \eqref{def:alpha-01}), Taylor expansions about $\gamma\lambda=0$ yield
\begin{equation} \label{eq:gamma-lambda-Q0}
\begin{aligned}
\frac {\gamma \lambda}4 Q_{\gamma\lambda}(0)^2
& = \frac {\gamma \lambda}4Q(0)^2+   \Pol + O((b^2 + \lambda)^\kappa)\\
& =  -\alpha_\star C_Q \lambda +   \Pol + O((b^2 + \lambda)^\kappa).
\end{aligned}
\end{equation}
and
\begin{equation} \label{eq:b2-CQ}
\begin{aligned}
\frac {b^2} 8 \intg y^2 Q_{\gamma\lambda}^2(y)\diff y
& = \frac {b^2} 8 \intg y^2 Q^2(y)\diff y +  \Pol + O((b^2 + \lambda)^\kappa) \\
& = b^2 C_Q+  \Pol + O((b^2 + \lambda)^\kappa).
\end{aligned}
\end{equation}
On the other hand, since $L_+ \rho = y^2 Q$ (see Lemma \ref{nondeg.lem}), 
we deduce from \eqref{eq:Psi-2m} that 
\[
P_{0,1,0}= \frac{\alpha_{0,1,0}}{4}\rho = \frac{\alpha_\star}{4}\rho.
\]
Hence, by Lemma~\ref{intid.lem} and Lemma~\ref{lem:inversion-Lbeta},
\[
(Q,P_{0,1,0}) =  \frac{\alpha_\star}{4}(Q,\rho) = \alpha_\star\frac{\int_\Gc y^2 Q(y)^2 \diff y}{8}.
\]
This gives
\begin{equation} \label{eq:innp-Q-P}
-\lambda(Q_{\gamma\lambda},P_{0,1,\gamma \lambda})   
=  - \lambda \alpha_\star C_Q +  \Pol + O((b^2 + \lambda)^\kappa).
\end{equation}
With \eqref{eq:gamma-lambda-Q0}-\eqref{eq:innp-Q-P} it follows that
\[
\frac{1}{\lambda^2}\Big( \frac{\gamma\lambda}{4}Q_{\gamma\lambda}(0)^2 
+ \frac{b^2}{8}\intg y^2 Q_{\gamma\lambda}^2\diff y 
-\lambda(Q_{\gamma\lambda},P_{0,1,\gamma \lambda})\Big) = C_Q \mathcal E_{\mathrm{op}} + \frac 1 {\lambda^2} \big(  \Pol + O((b^2 + \lambda)^\kappa) \big).
\]
Finally, upon close inspection of the other terms in \eqref{eq:expansion_E_tilde}, using
the definition of $Z$, we have
\begin{align*}
&-\frac16\intg\Big(|Q_{\gamma \lambda}+\lambda Z|^6-Q_{\gamma\lambda}^6-6Q_{\gamma\lambda}^5\re\big(\lambda \bar Z\big)\Big)\diff y \notag \\
&-\frac{b}{2}\lambda\im \intg y\big(Z_yQ_{\gamma\lambda}+\partial_yQ_{\gamma\lambda}\bar{Z}\big)\diff y
-\frac{b}{2}\lambda^2 \im\intg y Z_y\bar{Z}\diff y \notag \\
&+\frac{\lambda^2}{2}\intg|Z_y|^2\diff y + \frac{b^2}{4}\lambda\intg y^2 Q_{\gamma\lambda}\re\big(\bar{Z}\big)\diff y
+\frac{b^2}{8}\lambda^2\intg y^2|Z|^2\diff y  +\frac{\gamma\lambda^3}{2}|Z(0)|^2 \\
& =  \Pol + O((b^2 + \lambda)^\kappa).
\end{align*}
This gives \eqref{energy_nls_dynsys_2} and concludes the proof.
\end{proof}

\begin{corollary} \label{cor:P-Q}
Let $k \in \N$. There exists $C > 0$ such that for $\lambda > 0$ with $\gamma \lambda \in (-N,N)$ we have 
\[
\norm{Q_{\gamma\lambda}-Q}_{\Sigma^k} \leq C \lambda, \quad
\norm{P_{b,\lambda} - Q}_{\Sigma^k} \leq C \lambda.
\]
\end{corollary}

\begin{proof}
By construction of the profile we have 
\[
\norm{P_{b,\lambda} - Q_{\lambda \gamma}}_{\Sigma^k} \lesssim \lambda.
\]
On the other hand, we see from~\eqref{eq:der-Q-beta} that 
\[
\norm{Q_{\beta} - Q}_{\Sigma^k} \lesssim \beta,
\]
and the conclusion follows.
\end{proof}

\section{Modulation} \label{sec:modulation}

In this section we prove Proposition~\ref{prop:modulation}. 
We set
\[
\Omega = \R%
\times \R \times (0,{N}/{|\gamma|}).
\]
Then for $\pi = (\theta,b,\lambda) \in  \Omega$  and $w =(w_j)_{j=1,\dots,N} \in L^2(\Gc)$  we define $\Theta_\pi w \in \mathcal L (L^2(\Gc))$ by 
\[
(\Theta_\pi w)_j(x) 
= \frac 1 {\sqrt \lambda} e^{i\theta} e^{-\frac {ibx^2}{4\lambda^2}} w_j\left( \frac x \lambda \right),
\quad j=1,\dots,N.
\]
This defines a unitary operator $\Theta_\pi$ on $L^2(\Gc)$.

For $\varepsilon > 0$ we also set 
\[
\mathcal Q_\varepsilon = \bigcup_{\theta \in 
\R
, \lambda \in (0,\eps)} B_{L^2(\Gc)}(\Theta_{\theta,0,\lambda} Q, \varepsilon). 
\]
It is endowed with the topology inherited from $L^2(\Gc)$. 
Finally, we recall that $P_{b,\lambda}$ is defined in Proposition~\ref{prop:profile}
and $\rho$ by Lemma~\ref{nondeg.lem}. Then Proposition~\ref{prop:modulation} is a consequence of the following result.

\begin{proposition} \label{prop:modulation-1} 
There exist $\varepsilon > 0$ and a function $\pi = (\theta,b,\lambda) \in C^1(\mathcal Q_\varepsilon, \Omega)$ such that for any $u \in \mathcal Q_\varepsilon$ we have in $L^2(\Gc)$
\[
\Theta_{\pi(u)}^{-1} u - P_{b(u),\lambda(u),\gamma \lambda(u)} \in \big\{ y^2 P_{b(u),\lambda(u),\gamma \lambda(u)} ,i \Lambda P_{b(u),\lambda(u),\gamma \lambda(u)} , i\rho \big\}^\bot.
\]
\end{proposition}

\begin{proof}
For $\pi = (\theta,b,\lambda) \in \R%
\times \R \times (0,2)$, $w \in L^2(\Gc)$ and $\sigma \in \big(-\frac {N}{2|\gamma|},\frac {N}{2|\gamma|} \big)$, we set 
\[
h(\pi ; w,\sigma) = \Theta_{\pi}^{-1} w - P_{b,\sigma \lambda, \gamma \sigma \lambda}
\quad \text{and} \quad
F(\pi;w,\sigma) = 
\begin{pmatrix}
(h(\pi,w,\sigma),y^2 P_{b,\sigma \lambda, \gamma \sigma \lambda})_{L^2(\Gc)} \\
(h(\pi,w,\sigma),i \Lambda P_{b,\sigma \lambda, \gamma \sigma \lambda})_{L^2(\Gc)} \\
(h(\pi,w,\sigma),i\rho)_{L^2(\Gc)}
\end{pmatrix}.
\]
This defines functions of class $C^1$ from $\R
\times \R \times (0,2) \times L^2(\Gc) \times \big(-\frac {N}{2|\gamma|},\frac {N}{2|\gamma|} \big)$ to $L^2(\Gc)$ and $\R^3$, respectively. Moreover, we have $h(0,0,1;Q,0) = 0$ and $F(0,0,1;Q,0) = 0$ (the interest of the extra parameter $\sigma$ is that we can start the analysis around $\lambda = 1$ and $Q = P_{0,0,0}$). We have 
$(\partial_\theta P_{b,\sigma \lambda, \gamma \sigma \lambda},
\partial_b P_{b,\sigma \lambda, \gamma \sigma \lambda}, 
\partial_\lambda P_{b,\sigma \lambda, \gamma \sigma \lambda}) |_{b=0, \lambda = 1,\sigma = 0} 
= (0,0,0)$, so
\begin{equation*}
\nabla_{\theta,b,\lambda} h(0,0,1;Q,0) = \begin{pmatrix} -iQ \\ \frac {iy^2}{4} Q \\ \Lambda Q \end{pmatrix}.
\end{equation*}
By Lemma~\ref{intid.lem} we have $(Q,\Lambda Q) = 0$, so
\[
\mathsf{Jac}_{\theta,b,\lambda} F \big(0,0,1; Q,0 \big)
= 
\begin{pmatrix}
0 & 0 &  (\Lambda Q,y^2 Q)  \\
0 & \frac 14 (y^2 Q, \Lambda Q) & 0 \\
- (Q,\rho) & \frac 14 (y^2 Q , \rho) & 0
\end{pmatrix}.
\]
We also have $(y^2 Q,\Lambda Q) \neq 0$ and $(Q,\rho) \neq 0$, so this partial Jacobian is invertible. By the Implicit Function Theorem, there exist a neighborhood 
$\Uc \subset \R%
\times \R \times (0,2)$ of $(0,0,1)$, a neighborhood $\Vc$ of $(Q,0)$ in $L^2(\Gc) \times \big(-\frac {N}{2|\gamma|},\frac {N}{2|\gamma|} \big)$ and a function $\Pi_0 = (\theta_0,b_0,\lambda_0)  : \Vc \to \Uc$ of class $C^1$ such that, for all $\pi \in \Uc$ and $(w,\sigma) \in \Vc$, there holds 
\[
F(\pi;w,\sigma) = 0 \quad \Longleftrightarrow \quad \pi = \Pi_0(w,\sigma).
\]
We now fix $\varepsilon  \in \big(0,\frac N {2|\gamma|}\big)$ so small that $B(Q,\varepsilon) \times (-\varepsilon,\varepsilon) \subset \Vc$.

Let $u \in \mathcal Q_\varepsilon$. 
For $(\theta_1,\lambda_1) \in \R%
\times (0,\varepsilon)$ 
such that $w_1 =  \Theta_{\theta_1,0,\lambda_1}^{-1} u \in B(Q,\varepsilon)$, we set 
\begin{equation} \label{eq:theta-b-lambda}
\theta(u) = \theta_0 (w_1, \lambda_1) + \theta_1, \quad b(u) =  b_0 (w_1 , \lambda_1) , \quad \lambda (u) = \lambda_0 (w_1  , \lambda_1 ) \lambda_1.
\end{equation}
We have to check that this definition does not depend on the choice of $(\theta_1,\lambda_1)$. Let $(\theta_2,\lambda_2) \in \R
\times (0,\varepsilon)$ such that we also have $w_2 =  \Theta_{\theta_2,0,\lambda_2}^{-1} u \in B(Q,\varepsilon)$.  Since $w_1 = \Theta_{\theta_1 - \theta_2,0,\lambda_1 / \lambda_2}^{-1} w_2$, we have 
\begin{multline*}
F \big(\theta_0(w_1,\lambda_1) + \theta_1 - \theta_2, b_0(w_1,\lambda_1), \lambda(w_1,\lambda_1) \lambda_1 / \lambda_2 ; w_2 , \lambda_2 \big)\\
= F \big(\theta_0(w_1,\lambda_1) , b_0(w_1,\lambda_1), \lambda(w_1,\lambda_1)  ; w_1 , \lambda_1 \big) = 0,
\end{multline*}
which implies that 
\[
\theta_0 (w_2,\lambda_2) + \theta_2 = \theta_0(w_1,\lambda_1) + \theta_1, \quad b_0 (w_2,\lambda_2) = b_0(w_1,\lambda_1), \quad \lambda_0 (w_2,\lambda_2) \lambda_2 = \lambda_0(w_1,\lambda_1) \lambda_1.
\]
Thus, $\theta(u)$, $b(u)$ and $\lambda(u)$ are well defined by~\eqref{eq:theta-b-lambda}. 
This yields a function $\pi = (\theta,b,\lambda) \in C^1(\mathcal Q_\varepsilon, \Omega)$ such that, 
for all $u \in \mathcal Q_\eps$,
\begin{align*}
\Theta_{\theta(u),b(u),\lambda(u)}^{-1} u - P_{b(u),\lambda(u),\gamma \lambda(u)} & = \Theta_{\theta_0(w_1,\lambda_1),0,\lambda_0(w_1,\lambda_1)}^{-1} w_1 - P_{b_0(w_1,\lambda_1), \lambda_1 \lambda_0(w_1,\lambda_1), \gamma \lambda_1 \lambda_0(w_1,\lambda_1)}\\
& \in \big\{ y^2 P_{b(u),\lambda(u),\gamma \lambda(u)}, i\Lambda P_{b(u),\lambda(u),\gamma \lambda(u)}, i\rho \big\}^\bot,
\end{align*}
which completes the proof. 
\end{proof}

The functions $\theta(t), b(t), \lambda(t)$ obtained in Proposition~\ref{prop:modulation} are 
$C^1$ since they come from the composition $\pi\circ u$, and $u\in C^1(I,L^2(\mathcal G))$.
More generally, one could show that they remain $C^1$ even for a weak solution
$u\in C(I,L^2(\mathcal G))\cap C^1(I,H^1(\mathcal G)^\star)$ by using
the following result.

\begin{proposition} \label{prop:modulation-2}
Let $I$ be an interval of $\R$. Let $u \in C^0(I,L^2(\Gc)) \cap C^1(I,H^1(\mathcal G)^\star)$. Assume that there exists a sequence $(u_k)_{k \in \N}$ in $C^1(I,L^2(\Gc))$ which goes to $u$ in $C^0(I,L^2(\Gc)) \cap C^1(I,H^1(\mathcal G)^\star)$. Then the map $\pi \circ u$ (with $\pi$ given by Proposition~\ref{prop:modulation-1}) is of class $C^1$ on $I$.
\end{proposition}

\section{Uniform estimates}
\label{unifest_proof.sec}

In this section, we prove Proposition~\ref{prop:uniform_in_t}.

\subsection{Change of variables and bootstrap argument}

We recall that the profile $P$ was defined in Proposition~\ref{prop:profile}. 
We consider $t_1 < 0$ and the maximal solution $u_1$ of~\eqref{eq:nls} such that $u_1(t_1)$ 
is given by~\eqref{eq:definition_of_u_1}. We denote by $I_{t_1}$ the maximal interval 
on which $u_1$ is defined and satisfies~\eqref{eq:dist_u_modulation}. 
Since $u_1(t_1)\in D_{\gamma}\cap \Sigma^2(\Gc)$, it follows that
$u_1 \in C(I_{t_1},D_{\gamma}\cap \Sigma^2(\Gc)) \cap C^1(I_{t_1},L^2(\Gc))$.
Let $\tilde \theta \in C^1(I_{t_1},\R)$, $\tilde b \in C^1(I_{t_1},\R)$, 
$\tilde \lambda \in C^1(I_{t_1},\R_+^*)$ and 
$\tilde h \in C(I,D_{\gamma\lambda}\cap \Sigma^2(\Gc)) \cap C^1(I,L^2(\Gc))$ 
be the corresponding modulation parameters and remainder constructed in Proposition~\ref{prop:modulation}. 

We now define precisely the rescaled time variable $s$ that appears in the formal
change of variables~\eqref{chvarw}-\eqref{chvar}. 
We first define the final time $s_1$ via the solution $\lambda_{\mathrm{mo}}$ of the model dynamical system~\eqref{dynsys_unp} (see~\eqref{def:lambda-mo})
by setting
\begin{equation} \label{eq:t1-s1}
t_1 = -\int_{s_1}^{+\infty} \lambda_{\mathrm{mo}}(s)^2 \diff s = -\frac {4}{3 \alpha_\star^2 s_1^3},
\end{equation} 
where we recall that $\alpha_\star$ was defined in~\eqref{def:alpha-01}. This yields
\begin{equation}\label{eq:def_of_s_1}
 s_1=\Big(\frac{4}{3\alpha_\star^2} \Big)^{1/3}|t_1|^{-1/3}.
\end{equation}
Then for $t \in I_{t_1}$ we define
\begin{equation}\label{ch_var_precise}
\mathsf s(t)=s_1-\int_t^{t_1}\frac{1}{\tilde\lambda(\tau)^2}\diff\tau.
\end{equation}
We set $J_{s_1} = \mathsf s (I_{t_1})$. Since $\mathsf s : I_{t_1} \to J_{s_1}$ is strictly increasing, $t$ may in turn be expressed as a function of $s$ via $\mathsf t = \mathsf s^{-1} : J_{s_1} \to I_{t_1}$. This will allow us to obtain Proposition~\ref{prop:uniform_in_t} as a consequence of the uniform estimates in variable $s$ which are stated in Proposition~\ref{prop:uniform_in_s} below.

We then express the modulation parameters $\tilde b$, 
$\tilde \lambda$ and $\tilde \theta$  and the remainder $\tilde h$ as functions of the variable $s$ by setting on $J_{s_1}$
\begin{equation}\label{mod_param_st}
b(s) = \tilde b(\mathsf t(s)), \quad \lambda(s)=\tilde\lambda(\mathsf t(s)),
\quad \theta(s)=\tilde\theta(\mathsf t(s)),\quad h(s,\cdot)=\tilde h(\mathsf t(s),\cdot).
\end{equation}
In the rest of this section, we will often use the notation
\[
h_1=\re h,\quad h_2=\im h.
\]

In the new variables $(s,y)$, the function $v : (s,y) \in J_{s_1} \times \Gc \mapsto u\big(\mathsf t(s), \lambda(s) y\big)$ satisfies
\begin{equation} \label{eq:nls-v}
iv_s - H_{\gamma \lambda} v - \frac {\lambda_s}{\lambda} y v_y + \lambda^2 |v|^4 v = 0,
\end{equation}
and the final condition
\begin{equation}\label{eq:u1s1}
v_1(s_1,y) = \frac 1 {\sqrt{\lambda_1}} P_{b_1,\lambda_1}(y) e^{-i\frac {b_1y^2}4}.
\end{equation}

From now on we denote by 
\[
P : (s,y) \mapsto P_{b(s),\lambda(s),\gamma \lambda(s)}(y)
\]
the profile given by Proposition~\ref{prop:profile} for the modulation parameters~\eqref{mod_param_st}. 
Then, in view of \eqref{chvarw},
$$w = P+h = \sqrt{\lambda} e^{-i\theta} e^{ib(s)y^2/4}v$$ 
satisfies~\eqref{nls_w}. 

With the definition~\eqref{def:Psi-K} of $\Psi_\kappa$,~\eqref{nls_w} can then be rewritten as an equation for $h$: 
\begin{equation}\label{equ_for_h}
ih_s-H_{\gamma\lambda}h-h + f(P+h)-f(P)+\tModop(s)h
=-\Psi_\kappa-\Modop(s)P+b\Big(b+\frac{\lambda_s}{\lambda}\Big)\frac{y^2}{2}P,
\end{equation}
where for $\phi\in \Sigma^2(\mathcal G)$ we have set
\begin{equation*}
\tModop(s) \phi
:=(1-\theta_s)\phi+(b_s+b^2)\frac{y^2}{4}\phi
-i\Big(b+\frac{\lambda_s}{\lambda}\Big)\Lambda \phi
-b\Big(b+\frac{\lambda_s}{\lambda}\Big)\frac{y^2}{2} \phi
\end{equation*}
and
\begin{equation*}
\Modop(s) \phi
:=(1-\theta_s)\phi+(b_s+b^2-\alpha)\frac{y^2}{4} \phi 
-i\Big(b+\frac{\lambda_s}{\lambda}\Big)\Lambda \phi.
\end{equation*}
Notice that
\begin{equation*}
\tModop(s)\phi =
\Modop(s) \phi  + \alpha\frac{y^2}{4} \phi
-b\Big(b+\frac{\lambda_s}{\lambda}\Big)\frac{y^2}{2} \phi.
\end{equation*}

For any fixed $\lambda>0$, we define on $\Sigma^1(\mathcal G)$ the norm $\|\cdot\|_{\lambda}$ by
\begin{equation}\label{norm_lambda}
\| \phi \|_{\lambda}^2:=\| \phi \|_{H^1}^2+\lambda\|y \phi \|_{L^2}^2, \quad \phi \in\Sigma^1(\mathcal G).
\end{equation}
Note that $\|\cdot\|_{\lambda}$ is equivalent to the usual norm 
$\|\cdot\|_{1}\equiv \|\cdot\|_{\Sigma^1}$.

\begin{proposition}[Uniform estimates in the $s$ variable] \label{prop:uniform_in_s}
Let $s_0$ be given by Proposition \ref{prop:dynamic}. 
Choosing $s_0$ larger if necessary, there exist 
$C,c>0$ such that, for any $s_1 > s_0$ and $\kappa\in\N^*$, we have $[s_0,s_1] \subset J_{s_1}$ and, for all $s \in [s_0,s_1]$,
\begin{equation}\label{improved_estimates_h}
\|h(s)\|_{\lambda(s)}\leq C s^{-(\kappa-1)},
\end{equation}
\begin{equation}\label{improved_estimates_bl}
\Big|\frac{\lambda(s)^{1/2}}{\lambda_\mathrm{mo}(s)^{1/2}}-1\Big|\leq \frac{c}{s^2},
\quad
\Big|\frac{b(s)}{b_\mathrm{mo}(s)}-1\Big|\leq \frac{c}{s^2}.
\end{equation}
\end{proposition}

\begin{remark}\label{rem:improved_estimates}
    The estimates~\eqref{improved_estimates_bl} can be improved to
    \begin{equation}\label{better_improved_estimates_bl}
\Big|\frac{\lambda(s)^{1/2}}{\lambda_\mathrm{mo}(s)^{1/2}}-1\Big|\lesssim \frac{1}{s^4},
\quad
\Big|\frac{b(s)}{b_\mathrm{mo}(s)}-1\Big|\lesssim \frac{1}{s^4},
\quad s\in[s_0,s_1],
\end{equation}
by shifting the energy level in the definition of $\mathcal F$ in~\eqref{defofF} to
$\underline{\mathcal E}^\star=\mathcal E^\star-(\eps_{0,2}+2\alpha_\star\eps_{2,1})$ 
(see Remark~\ref{rem:improved_est_3}).
\end{remark}

\begin{proof}[Proof of Proposition~\ref{prop:uniform_in_t}, assuming Proposition~\ref{prop:uniform_in_s}]
Let $s_1>s_0$ and $t_1$ be defined by \eqref{eq:t1-s1}.
Since $[s_0,s_1]\subset J_{s_1}$, we have $\mathsf t([s_0,s_1])\subset \mathsf t(J_{s_1})=I_{t_1}$, for all $s_1>0$. 
However, since the map $\mathsf t$ depends on $s_1$, so does the left end-point of the interval $\mathsf t([s_0,s_1])$.
We now construct $t_0$, independent of $t_1$, such that $\mathsf t(s_0)\le t_0$, and so $[t_0,t_1]\subset I_{t_1}$ for any $t_1$.

By \eqref{eq:t1-s1}, we have
\begin{align*}
\mathsf t(s_0)
&=t_1-\int_{s_0}^{s_1}\lambda^2(\sigma)\diff\sigma \\
&=-\int_{s_1}^{+\infty}\lambda_\mathrm{mo}^2(\sigma)\diff\sigma -\int_{s_0}^{s_1}\lambda^2(\sigma)\diff\sigma \\
&=-\int_{s_0}^{+\infty}\lambda_\mathrm{mo}^2(\sigma)\diff\sigma + \int_{s_0}^{s_1}\lambda_\mathrm{mo}^2(\sigma)\diff\sigma
-\int_{s_0}^{s_1}\lambda^2(\sigma)\diff\sigma \\
&=-\Big(\frac{4}{3\alpha_\star^2}\Big)s_0^{-3} -\int_{s_0}^{s_1}\big[\lambda^2(\sigma)-\lambda_\mathrm{mo}^2(\sigma)\big]\diff\sigma. 
\end{align*}
Furthermore, for $s \in J_{s_1}$, \eqref{improved_estimates_bl} implies
\begin{equation*}
\big| \lambda(s)^2 - \lambda_\mathrm{mo}(s)^2 \big| \lesssim s^{-6}.
\end{equation*}
Hence,
\[
\Big| \int_{s_0}^{s_1}\big[\lambda^2(\sigma)-\lambda_\mathrm{mo}^2(\sigma)\big]\diff\sigma \Big|
\le \int_{s_0}^{+\infty}\sigma^{-6}\diff\sigma \lesssim s_0^{-5}.
\]
It follows that, upon choosing $s_0$ large enough, $\mathsf t(s_0)\le t_0:=-(\frac{2}{3\alpha_\star^2})s_0^{-3}$ for any $t_1$.

Next, for $t \in [t_0,t_1]$,
\[
t = t_1 - \int_{\mathsf s (t)}^{s_1} \lambda(\sigma)^2 \diff \sigma = -\int_{\mathsf s (t)}^{\infty}\lambda_\mathrm{mo}(\sigma)^2\diff\sigma + O \big({\mathsf s (t)}^{-5} \big)
=-\frac{4}{3\alpha_\star^2}{\mathsf s (t)}^{-3} + O \big({\mathsf s (t)}^{-5} \big).
\]
Hence,
\begin{equation*}
\mathsf s(t) = \Big(\frac{3\alpha_\star^2}{4}|t| \Big)^{-1/3} \big( 1 + O(\abs t^{\frac 23}) \big), \quad t\to0^-.
\end{equation*}
It follows that
\begin{equation*}
b_\mathrm{mo}(\mathsf s(t))
=\frac{2}{\mathsf s(t)} =  2\Big(\frac{3\alpha_\star^2}{4}|t| \Big)^{1/3} + O(\abs t), \quad t\to0^-
\end{equation*}
and
\begin{equation*}
\lambda_\mathrm{mo}(\mathsf s(t))=\frac{2}{\alpha_\star \mathsf s(t)^2}
= \frac{2}{\alpha_\star}\Big(\frac{3\alpha_\star^2}{4}|t| \Big)^{2/3} + O(\abs t^{\frac 43}), \quad t\to0^-.
\end{equation*}
Estimates~\eqref{t_est_bl} now follow directly from~\eqref{improved_estimates_bl}.
Recalling the definition of the norm $\|\cdot\|_\lambda$ in~\eqref{norm_lambda},
the estimates~\eqref{t_est_h} then follow from~\eqref{t_est_bl} and~\eqref{improved_estimates_h}.
Finally, \eqref{t_est_nrj} follows from estimate \eqref{s_est_nrj}, proved below.
\end{proof}

The proof of Proposition~\ref{prop:uniform_in_s} relies on the following bootstrap result.

\begin{proposition} \label{prop:sstar}
Let $s_0$ be given by Proposition~\ref{prop:dynamic}. For $s_1>s_0$,
define $s_\star=s_\star(s_1)$ as the infimum of $\sigma\in J_{s_1} \cap [1,s_1]$ such that, for all $s\in[\sigma,s_1]$, 
\begin{equation}\label{bootstrap_hyp}
\|h(s)\|_{\lambda} < s^{-(\kappa-2)}, \quad
\Big|\frac{\lambda(s)^{1/2}}{\lambda_\mathrm{mo}(s)^{1/2}}-1\Big| < \frac{1}{s^{1/2}},
\quad
\Big|\frac{b(s)}{b_\mathrm{mo}(s)}-1\Big| < \frac{1}{s^{1/2}}.
\end{equation} 
There exist $\kappa_0 \ge 2$ and $C,c>0$, all independent of $s_1$,
such that \eqref{improved_estimates_h} and \eqref{improved_estimates_bl} 
hold for $\kappa \geq \kappa_0$ and $s \in [s_\star,s_1]$. 
\end{proposition}

\begin{remark}
The constant $\kappa_0$ can be chosen as the smallest integer larger than $\max\{9,2k/k_0+1\}$
(see Proposition~\ref{coercivity.prop}, Proposition~\ref{dhds_below.prop} and \eqref{aaa}).
\end{remark}

Note that, for $s \in [s_\star,s_1]$, \eqref{def:lambda-mo} and~\eqref{bootstrap_hyp} 
yield
\begin{equation} \label{eq:dec-b-lambda}
b(s)\sim \frac2s,\quad \lambda(s) \sim \frac2{\alpha_\star}\frac1{s^2}.
\end{equation}

\begin{proof}[Proof of Proposition~\ref{prop:uniform_in_s}, assuming Proposition~\ref{prop:sstar}] 
It is enough to prove Proposition~\ref{prop:uniform_in_s} for some large $\kappa$. 
We thus assume that $\kappa \geq \kappa_0$, where $\kappa_0$ is given by Proposition~\ref{prop:sstar}, 
and hence~\eqref{improved_estimates_h}-\eqref{improved_estimates_bl} hold on $[s_\star,s_1]$ for any $s_1$. 
Choosing $s_0$ larger if necessary, we can assume that $s_0 \geq  \max\{2C,2c^{2/3}\}$. Suppose by contradiction that there exists $s_1>s_0$ such that either 
$J_{s_1}\not\subset[s_0,s_1]$ or $J_{s_1}\subset[s_0,s_1]$ but one of the inequalities in \eqref{improved_estimates_h}-\eqref{improved_estimates_bl}
fails on $[s_0,s_1]$. By Proposition~\ref{prop:sstar}, this implies $s_0<s_\star$. By \eqref{bootstrap_hyp} and Corollary \ref{cor:P-Q},
we have $\inf J_{s_1} < s_\star$ if $s_0$ is large enough. Furthermore,
since $h(s_1)=0$ by construction, it follows from Proposition~\ref{prop:dynamic} that~\eqref{bootstrap_hyp} is satisfied at $s=s_1$, provided $s_0$ is large enough.
Hence, by continuity, $s_{\star} < s_1$ and at least one of the inequalities of~\eqref{bootstrap_hyp} becomes an equality at $s=s_\star$. 
Since $s_\star>\max\{2C,2c^{2/3}\}$, this yields a 
contradiction with \eqref{improved_estimates_h}-\eqref{improved_estimates_bl}.
Thus, $s_\star \leq s_0$. This means that $s_0 \in J_{s_1}$ and \eqref{improved_estimates_h}-\eqref{improved_estimates_bl} hold on $[s_0,s_1]$.
\end{proof}

The rest of this section is devoted to the proof of Proposition~\ref{prop:sstar}.

\subsection{Modulation estimates}\label{sec:modulation_estimates}

For the rest of the section we work under the assumptions of Proposition~\ref{prop:sstar}. All the estimates which appear in the discussion are independent of $s_1$.

We first justify quantitatively that, for large times, 
the modulation parameters are approximate solutions
of the model dynamical system~\eqref{dynsys_unp}.

Let 
\begin{equation*}
\Mod(s):=
\begin{pmatrix}
    b+\lambda_s/\lambda \\
    b_s+b^2-\alpha \\
    1-\theta_s
\end{pmatrix},
\quad s\in[s_0,s_1].
\end{equation*}
Then by~\eqref{eq:dec-b-lambda} and~\eqref{def:alpha} we have 
\begin{equation} \label{eq:b-s-lambda-s}
|b_s| \lesssim \frac 1 {s^2} + |\Mod(s)|, \quad |\lambda_s| \lesssim \frac 1 {s^3} + \frac {|\Mod(s)|}{s^2}
\end{equation}

\begin{lemma}\label{mod_est.lem}
If $\kappa\ge 5$, there exists $C_\kappa > 0$ (independent of $s_1$) such that, 
for all $s\in[s_\star,s_1]$,
\begin{align}
|\Mod(s)|	& \leq C_\kappa  s^{-\kappa} \label{mod_est}	,\\ 	
|(h(s),Q)_{L^2}|		&\leq C_\kappa s^{-\kappa}.		\label{almost_orth}
\end{align}
\end{lemma}

\begin{proof}
Let us first define
\begin{equation*}
s_{\star\star}:=\inf\big\{\sigma\in[s_\star,s_1] : |(h(s),P(s))_{L^2}| < s^{-\kappa},
\ \forall s\in[\sigma,s_1]\big\}.
\end{equation*}
Since $h(s_1)=0$, we have $s_{\star\star}\in[s_\star,s_1)$. 
The main step is to prove that~\eqref{mod_est} holds on $[s_{\star\star},s_1]$. By a bootstrap argument, 
it will then follow that $s_{\star\star}=s_\star$, and we will finally deduce~\eqref{almost_orth}. 

We start by differentiating the equality $(h,i\Lambda P)=0$ with respect to $s$. Since $P$, $h$ 
and $\Lambda P$ belong to $C^1(J_{s_1},L^2(\Gc))$ by Proposition~\ref{prop:profile}, this simply gives, on $[s_{\star\star},s_1]$,
\begin{equation}\label{diff_oc1}
\scalar{h_s}{i\Lambda P}+ \scalar{h}{i\Lambda P_s}=0.
\end{equation}
Firstly, for  $\kappa \geq 4$, \eqref{def:alpha}, \eqref{def:Psi-K}, \eqref{remainder_est},
\eqref{eq:expansion_of_f} and \eqref{bootstrap_hyp} yield
\begin{equation} \label{eq:h-Lambda-Ps}
|(h,i\Lambda P_s)| \lesssim \|h\|_{L^2}\big(|\Mod(s)|+s^{-2})
\lesssim s^{-(\kappa-2)}|\Mod(s)|+s^{-\kappa} 
\lesssim s^{-2}|\Mod(s)|+s^{-\kappa}.
\end{equation}
We next estimate $\scalar{h_s}{i\Lambda P}=-\scalar{ih_s}{\Lambda P}$. 
By Corollary~\ref{cor:P-Q} and \eqref{bootstrap_hyp}, we have
\begin{equation*}
f(P+h)-f(P)=\dif f(P)h+O(|h|^2)=\dif f(P)h+O(s^{-(\kappa-2)}|h|)
=\dif f(Q_{\gamma\lambda})h+O(s^{-2}|h|),
\end{equation*}
where $\dif f$ denotes the differential (in the real sense) of $f:\C\to\C, \ f(z)=|z|^4z$, 
that is, $\dif f(z)h=|z|^4h+4|z|^2z\re(z\bar h)$.
Thus, the equation for $h$ reads 
\begin{align}\label{equ_for_h1}
ih_s=
& \ H_{\gamma\lambda}h+h-\dif f(Q_{\gamma\lambda})h +O(s^{-2}|h|) \notag \\
&-\Big[\tModop(s)h+\Modop(s)P
-b\Big(b+\frac{\lambda_s}{\lambda}\Big)\frac{y^2}{2}P+\Psi_\kappa\Big].
\end{align}
Furthermore, noting that
\begin{equation}\label{eq:lin_op_h}
H_{\gamma\lambda} h + h - \dif f(Q_{\gamma\lambda}) h = \Lop_{+,\gamma\lambda} h_1 
+ i \Lop_{-,\gamma\lambda} h_2
\end{equation}
we have, for all $s\in[s_{\star\star},s_1]$, 
\begin{align*}
\scalar{H_{\gamma\lambda}h+h-\dif f(Q_{\gamma\lambda})h}{\Lambda P} 
=(\Lop_{+,\gamma \lambda} h_1, \Lambda Q)
+(H_{\gamma\lambda}h+h-\dif f(Q_{\gamma\lambda})h, \Lambda (P-Q)).
\end{align*}
By Corollary~\ref{cor:P-Q}, it follows that 
\begin{align*}
|(H_{\gamma\lambda}h+h-\dif f(Q_{\gamma\lambda})h, \Lambda (P-Q))| \lesssim \|h\|_{H^1}\|\Lambda (P-Q)\|_{H^1} 
\lesssim s^{-2}\|h\|_{H^1}.
\end{align*}
On the other hand, by Lemma~\ref{nondeg.lem} and~\eqref{bootstrap_hyp}, 
\begin{align*}
(\Lop_{+,\gamma \lambda} h_1, \Lambda Q)
& = \dual{\Lf_{+,\gamma \lambda} h_1}{\Lambda Q} \\
& = \dual{\Lf_{+} h_1}{\Lambda Q}+\gamma\lambda h_1(0)(\Lambda Q)(0)\\
& = (h_1,\Lop_+\Lambda Q)+O(s^{-2} \norm{h}_{H^1}) \\
& = -2(h,Q)+O(s^{-\kappa}).
\end{align*}
Therefore, by Corollary~\ref{cor:P-Q} and the definition of $s_{\star\star}$ we have, on $[s_{\star\star},s_1]$,
\begin{align*}
\scalar{H_{\gamma\lambda}h+h-\dif f(Q_{\gamma\lambda})h +O(s^{-2}|h|)}{\Lambda P}
&=-2(h,Q)+O(s^{-\kappa}) \\
&=-2(h,P)+O(s^{-\kappa})=O(s^{-\kappa}).
\end{align*}
Now we estimate the terms in the second line of~\eqref{equ_for_h1}. 
Firstly, putting the factors $y^2$ on the right,
\begin{align*}
\scalar{\tModop(s)h}{\Lambda P} 
&=\scalar{\Modop(s)h  + \alpha\frac{y^2}{4}h
-b\Big(b+\frac{\lambda_s}{\lambda}\Big)\frac{y^2}{2}h}{\Lambda P}  \\
&=O(|\Mod(s)|\|h\|_{L^2})+O(|\alpha(s)|\|h\|_{L^2})+O(b|\Mod(s)|\|h\|_{L^2})
 \\
&=O(|\Mod(s)|\|h\|_{L^2})+O(s^{-\kappa}).
\end{align*}
Next, by Lemma~\ref{intid.lem},
\begin{align*}
&\scalar{\Modop(s)P}{\Lambda P}
- \scalar{b\Big(b+\frac{\lambda_s}{\lambda}\Big)\frac{y^2}{2}P}{\Lambda P}\\
&=\scalar{\Modop(s)Q}{\Lambda Q}
- \scalar{b\Big(b+\frac{\lambda_s}{\lambda}\Big)\frac{y^2}{2}Q}{\Lambda Q}
+O(s^{-2}|\Mod(s)|) \\
&=-\frac14(b_s+b^2-\alpha)\|yQ\|_{L^2}^2
+\frac12b\Big(b+\frac{\lambda_s}{\lambda}\Big)\|yQ\|_{L^2}^2 
+O(s^{-2}|\Mod(s)|) .%
\end{align*}
Finally,
\begin{equation*}
\scalar{\Psi_\kappa}{\Lambda P}=O(s^{-2}|\Mod(s)|)+O(s^{-2\kappa}).
\end{equation*}
All in all, we find that
\begin{equation*}
\scalar{h_s}{i\Lambda P}
=\frac14\|yQ\|_{L^2}^2\Big[2b\Big(b+\frac{\lambda_s}{\lambda}\Big)
-(b_s+b^2-\alpha)\Big)\Big]
+O(s^{-2}|\Mod(s)|)+O(s^{-\kappa}).
\end{equation*}
Hence, the restriction of~\eqref{diff_oc1} to $[s_{\star\star},s_1]$ gives
\begin{equation}\label{est_mod_inter}
(b_s+b^2-\alpha)-2b\Big(b+\frac{\lambda_s}{\lambda}\Big)=
O(s^{-2}|\Mod(s)|)+O(s^{-\kappa}).
\end{equation}

We next differentiate $(h,y^2 P)=0$ with respect to $s$ and we get, on $[s_{\star\star},s_1]$,
\begin{equation}\label{diff_oc2}
\scalar{h_s}{y^2 P}+\scalar{h}{y^2 P_s}=0.
\end{equation}
Similarly to \eqref{eq:h-Lambda-Ps}, for $\kappa\ge 4$ we deduce from \eqref{bootstrap_hyp} that
\begin{equation*}
|(h,y^2 P_s)|\lesssim s^{-2}|\Mod(s)|+s^{-\kappa}.
\end{equation*}
For the term $\scalar{h_s}{y^2 P}=\scalar{ih_s}{iy^2 P}$ we shall again use 
\eqref{equ_for_h1} on $[s_{\star\star},s_1]$. By Lemma~\ref{nondeg.lem},~\eqref{bootstrap_hyp}, Corollary~\ref{cor:P-Q} 
and Proposition~\ref{prop:modulation} we have, on $[s_{\star\star},s_1]$,
\begin{align*}
\scalar{H_{\gamma\lambda}h+h-\dif f(Q_{\gamma\lambda})h +O(s^{-2}|h|)}{iy^2P} 
&= \scalar{H_{\gamma\lambda}h+h-\dif f(Q)h}{iy^2 Q} +O(s^{-\kappa}) \\
&=\scalar{ h_2}{\Lf_- y^2Q}+O(s^{-\kappa}) \\
&=-4(h_2,\Lambda Q)+O(s^{-\kappa}) \\
&=-4(h,i\Lambda P)+O(s^{-\kappa}) \\
&=O(s^{-\kappa}).
\end{align*}
Using Lemma~\ref{intid.lem} and~\eqref{bootstrap_hyp}, similar calculations to the above yield
\begin{multline*}
\scalar{\tModop(s)h+\Modop(s)P
-b\Big(b+\frac{\lambda_s}{\lambda}\Big)\frac{y^2}{2}P+\Psi_\kappa}{iy^2P}\\
=\|yQ\|_{L^2}^2\Big(b+\frac{\lambda_s}{\lambda}\Big)+O(s^{-2}|\Mod(s)|)+O(s^{-\kappa}).
\end{multline*}
Hence, gathering all terms of~\eqref{diff_oc2}, we find that
\begin{equation}\label{est_mod_i}
\Big|b+\frac{\lambda_s}{\lambda}\Big|\lesssim s^{-2}|\Mod(s)|+s^{-\kappa},
\quad  s\in[s_{\star\star},s_1].
\end{equation}
It then follows from~\eqref{est_mod_inter} that
\begin{equation}\label{est_mod_ii}
|b_s+b^2-\alpha|\lesssim s^{-2}|\Mod(s)|+s^{-\kappa}, \quad  s\in[s_{\star\star},s_1].
\end{equation}

Finally, differentiating $(h, i\rho)=0$ with respect to $s$ yields, on $[s_{\star\star},s_1]$,
\begin{equation}\label{eq:diff_oc2}
\scalar{ih_s}{\rho}=0.
\end{equation}
As above we have, on $[s_{\star\star},s_1]$,
\begin{multline*}
\scalar{H_{\gamma\lambda}h +h-\dif f(Q_{\gamma\lambda})h +O(s^{-2}|h|)}{\rho} \\
= \scalar{h_1}{L_+ \rho}+O(s^{-\kappa}) 
=(h_1,y^2Q)+O(s^{-\kappa}) 
=(h,y^2P)+O(s^{-\kappa})
=O(s^{-\kappa}).
\end{multline*}
On the other hand, Lemma~\ref{intid.lem} and~\eqref{bootstrap_hyp} yield (notice that $\rho$ also decays exponentially)
\begin{eqnarray*}
\lefteqn{\scalar{\tModop(s)h+\Modop(s)P-b\Big(b+\frac{\lambda_s}{\lambda}\Big)\frac{y^2}{2}P+\Psi_K}{\rho}} \\
&& =\frac12\|yQ\|_{L^2}^2(1-\theta_s) + \frac14(y^2Q,\rho)(b_s+b^2-\alpha)
-\frac12(y^2Q,\rho)b\Big(b+\frac{\lambda_s}{\lambda}\Big)\\
&& +O(s^{-2}|\Mod(s)|)+O(s^{-\kappa}).
\end{eqnarray*}
It then follows from~\eqref{eq:diff_oc2} that
\begin{equation*}
|1-\theta_s|=O(|b_s+b^2-\alpha|)+O\Big(b\Big|b+\frac{\lambda_s}{\lambda}\Big|\Big)
+O(s^{-2}|\Mod(s)|)+O(s^{-\kappa}).
\end{equation*}
Thus, by~\eqref{est_mod_i} and~\eqref{est_mod_ii},
\begin{equation}\label{est_mod_iii}
|1-\theta_s|\lesssim s^{-2}|\Mod(s)|+s^{-\kappa}, \quad  s\in[s_{\star\star},s_1].
\end{equation}
Gathering~\eqref{est_mod_i},~\eqref{est_mod_ii} and~\eqref{est_mod_iii}, we conclude that
\begin{equation*}
|\Mod(s)|\lesssim s^{-2}|\Mod(s)|+s^{-\kappa},
\end{equation*}
which implies that~\eqref{mod_est} holds on $[s_{\star\star},s_1]$.

Since $h(s_1)=0$ and by conservation of the mass, we now have, for $s \in [s_{\star\star},s_1]$,
\begin{align*}
\|P(s_1)\|_{L^2}^2
&=\|v(s_1)\|_{L^2}^2=\|v(s)\|_{L^2}^2=\|P(s)+h(s)\|_{L^2}^2 \\
&=\|P(s)\|_{L^2}^2+2(P(s),h(s))+\|h(s)\|_{L^2}^2,
\end{align*}
so that
\begin{equation*}
(P(s),h(s))=-\frac12\|h(s)\|_{L^2}^2
+\frac12\big(\|P(s_1)\|_{L^2}^2-\|P(s)\|_{L^2}^2\big).
\end{equation*}
Furthermore, using
\eqref{def:Psi-K},~\eqref{remainder_est} and~\eqref{mod_est} (proved on $[s_{\star\star},s_1]$), 
we get
\begin{align*}
\big|\partial_s \|P\|_{L^2}^2\big|
=2|(P,P_s)|
= 2|(P,-i\Psi_\kappa)|
\lesssim \norm{\Psi_\kappa} \lesssim s^{-(\kappa+2)}.
\end{align*}
Hence, integrating from $s$ to $s_1$,
\begin{equation*}
\left|\|P(s_1)\|_{L^2}^2-\|P(s)\|_{L^2}^2\right| \lesssim s^{-(\kappa+1)}.
\end{equation*}
For $\kappa \geq 5$, it thus follows by~\eqref{bootstrap_hyp} that 
\begin{equation*}
|(P,h)| \lesssim \|h\|_{L^2}^2+s^{-(\kappa+1)}\lesssim s^{-(\kappa+1)}, \quad s\in [s_{\star\star},s_1].
\end{equation*}
Therefore, $s_{\star\star}=s_{\star}$ and in particular~\eqref{mod_est} holds on $[s_\star,s_1]$. 

Finally,~\eqref{almost_orth} follows from Corollary~\ref{cor:P-Q},~\eqref{bootstrap_hyp} and the estimate 
$|(P,h)| \lesssim s^{-\kappa}$, which now holds on $[s_\star,s_1]$.
\end{proof}

\subsection{Monotone energy-virial functional}

We define $F:\mathbb C\to\mathbb C$ by $F(z) = \frac16|z|^6$. 
The differential of $F$ (in the real sense)  is given by $\dif F(z)h = \abs z^4 \re(z\bar h)$. 
For $s\in[s_\star,s_1]$ and $\phi \in H^1(\Gc)$, we set 
\begin{equation}\label{functional}
\begin{aligned}
G(s,\phi)=\frac12\|\phi\|_{\lambda}^2
+\frac{\gamma\lambda}{2}|\phi(  0)|^2
-\int_{\mathcal G} \big[F(  P+ \phi)-F(  P)-\dif F(  P) \phi \big]\diff   y.
\end{aligned}
\end{equation}
If $\phi \in D_{\gamma \lambda}$, we also have 
\begin{equation*}
\begin{aligned}
G(s,\phi)=\frac12 \scalar{H_{\gamma \lambda} \phi + \phi}{\phi} + \frac \lambda 2 \norm{y \phi}_{L^2}^2
-\int_{\mathcal G} \big[F(  P+ \phi)-F(  P)-\dif F(  P) \phi \big]\diff   y.
\end{aligned}
\end{equation*}
Then our main tool to bootstrap the estimate on $h$ will be the functional 
\begin{equation}\label{Sfunctional}
S(s,h(s)):=\frac{G(s,h(s))}{\lambda^m(s)},
\end{equation}
where $m$ is a large positive integer which will be determined later.

\begin{proposition}\label{coercivity.prop}
There exists $k_0 >0$ and $C > 0$ such that, for $s\in[s_\star,s_1]$, there holds
\begin{equation*}
\frac{1}{\lambda^m}\big( k_0 \|  h\|_{\lambda}^2 - C s^{-2(\kappa-1)} \big) \leq S(s,  h) \leq \frac12\|  h\|_{\lambda}^2 + C s^{-2}\norm{  h}_{H^1}^2.
\end{equation*}
\end{proposition}

\begin{proof}
By~\eqref{eq:dec-b-lambda} we have 
\[
\frac{|\gamma| \lambda}{2} |h(0)|^2  \lesssim s^{-2} \norm{h}_{H^1}^2.
\]
On the other hand, by Corollary~\ref{cor:P-Q}, %
\begin{align*}
F(P+h)-F(P)-\dif F(P)h
&=\frac16|P+h|^6-\frac16|P|^6-\re\left(|P|^4P\bar h\right) \notag \\
&=\frac52Q^4h_1^2+\frac12Q^4h_2^2+O(\lambda |h|)+O(|h|^3).
\end{align*}
Hence
\[
\left|\int_{\mathcal G} \big[F(  P+  h)-F(  P)-\dif F(  P)  h\big]\diff   y \right| \lesssim s^{-2} \norm{h}_{H^1}^2
\]
and the upper bound follows. 

For the lower bound, we use Lemma~\ref{coercivity.lem}, Proposition~\ref{prop:modulation} and Corollary~\ref{cor:P-Q} to write 
\begin{align*}
\frac12\|h\|_{H^1}^2-\int_{\mathcal G} 
&\big[F(P+h)-F(P)
-\dif F(P)h\big]\diff y \\
&=\frac12\Big[ \|h\|_{H^1}^2
-\int_{\mathcal G} (5Q^4h_1^2+Q^4h_2^2)\diff y \Big] + O(s^{-2}\|h\|^2_{L^2}) + O(\|h\|_{H^1}^3) \\
&=\frac12\Big[\la \Lf_+ h_1,h_1\ra+\la \Lf_- h_2,h_2\ra\Big] + O(s^{-2}\|h\|^2_{L^2}) + O(\|h\|_{H^1}^3) \\
&\ge \frac{\mu}{2}\|h\|_{H^1}^2-\frac{1}{2\mu}\left((h,Q)_{L^2}^2 + O(s^{-4}\|h\|_{L^2}^2)\right) 
+ O(s^{-2}\|h\|^2_{L^2}) + O(\|h\|_{H^1}^3).
\end{align*}
By~\eqref{bootstrap_hyp} and~\eqref{almost_orth}, we conclude that
\begin{multline*}
\frac12\|h\|_{H^1}^2
-\int_{\mathcal G} \big[F(P+h)-F(P)-\dif F(P)h\big]\diff y \\
\ge \frac{\mu}{2}\|h\|_{H^1}^2 + O(s^{-2\kappa}) +O(s^{-2(\kappa-1)}) + O(s^{-3(\kappa-2)}),
\end{multline*}
and hence
\begin{equation*}
G(s,h) \ge \frac \mu 2 \|h\|_{\lambda}^2 + O(s^{-2(\kappa-1)}),
\end{equation*}
which completes the proof.
\end{proof}

We shall next estimate the total derivative of $G(s,h(s))$ with respect to $s$.
Note that
$G$ is Fr\'echet differentiable with respect to its second variable 
at any $h \in D_{\gamma \lambda}$, with derivative 
\begin{align} \label{expr_DhH-1}
D_hG = H_{\gamma\lambda}h+h+\lambda y^2h- \big( f(P+h)-f(P) \big).
\end{align}
This can also be written as 
\begin{align}\label{expr_DhH-2}
D_hG = H_{\gamma\lambda}h+h+\lambda y^2h-\dif f(P)h-R(h),
\end{align}
where
\[
R(h)=f(P+h)-f(P)-\dif f(P)h=O(|h|^2).
\]
Since $h(s) \in D_{\gamma \lambda}$ for all $s \in J_{s_1}$ and $h \in C^1(J_{s_1},L^2(\Gc))$, 
we can write
\begin{equation}\label{total_deriv}
\frac{\dif}{\dif s}G(s,h(s))=D_sG(s,h(s))+( D_hG(s,h(s)),h_s).
\end{equation}

\begin{proposition}\label{dhds_below.prop}
There exists a constant $k>0$ such that, for all $s\in[s_\star,s_1]$,
\begin{equation}\label{dhds_below}
\frac{\dif}{\dif s} G(s,h(s)) \ge b\left(-k\|h\|_{\lambda}^2 + O(s^{-2(\kappa-1)})\right).
\end{equation}
\end{proposition}

The proof of Proposition~\ref{dhds_below.prop} breaks down into several lemmas.

\begin{remark}
Upon close inspection of the estimates involved, one notices that \eqref{dhds_below} holds for any $k$ verifying
$$
k>1+\max\{1,4/\alpha_\star\}+\max\{2,\alpha_\star/2\}.
$$
\end{remark}

\begin{lemma}\label{dsH.lem}
There exists $k_1>0$ such that, for all $s\in[s_\star,s_1]$,
\begin{equation*}
D_sG(s,h(s))\ge -b\,k_1\|h\|_{\lambda}^2. \quad %
\end{equation*}
\end{lemma}

\begin{proof}
We have
\begin{equation*}
D_s G=\frac{\lambda_s}{2} \|yh\|_{L^2}^2+\frac{\gamma\lambda_s}{2}|h(0)|^2
-D_s\int_{\mathcal G} \big[F(P+h)-F(P)
-\dif F(P)h\big]\diff y.
\end{equation*}
By~\eqref{eq:b-s-lambda-s} and~\eqref{mod_est} we have $\lambda_s = O(s^{-3})$, so
\begin{equation*}
\frac{|\lambda_s|}{2} \|yh\|_{L^2}^2+\frac{\gamma|\lambda_s|}{2}|h(0)|^2
\lesssim s^{-3}\|yh\|_{L^2}^2+s^{-3}\|h\|_{H^1}^2.
\end{equation*}
The expression
\[
F(P+h)-F(P)
-\dif F(P)h
\]
is composed of a collection of polynomial terms in $P$, $\bar P$, $h$ and $\bar h$, all of which of order $6$, and at least of order $2$ in $(h,\bar h)$.
As a consequence, we may apply $D_s$ to get
\begin{equation*}
D_s\int_{\mathcal G} \big[F(P+h)-F(P)-\dif F(P)h\big]\diff y
=O(\lambda_s\|h\|_{H^1}^2)= O( s^{-3}\|h\|_{H^1}^2).
\end{equation*}
Thus,
\begin{equation*}
|D_s G|\lesssim s^{-3}\|yh\|_{L^2}^2+s^{-3}\|h\|_{H^1}^2
\lesssim s^{-1}\left(\|h\|_{H^1}^2+s^{-2}\|yh\|_{L^2}^2\right).
\end{equation*}
The result now follows from~\eqref{eq:dec-b-lambda}.
\end{proof}

We next estimate the second term $\scalar{D_hG}{h_s} = \scalar{iD_h G}{ih_s}$ in~\eqref{total_deriv}. We rewrite~\eqref{equ_for_h} as
\begin{equation*}
ih_s
=D_hG-\lambda y^2h-\tModop(s)h
-\Modop(s)P+\Phi_\kappa, 
\end{equation*}
where, by~\eqref{remainder_est} and~\eqref{mod_est}, 
\begin{equation}\label{phiK_est}
\Phi_\kappa =
b\Big(b+\frac{\lambda_s}{\lambda}\Big)\frac{y^2}{2}P-\Psi_\kappa(s)
=O_{\Sigma^1(\mathcal G)}(s^{-(\kappa+1)}).
\end{equation}
Since $(iD_hG,D_hG)=0$, we have
\begin{equation}\label{D_hHh_s}
( D_hG,h_s)=-\lambda ( iD_hG,y^2h)-(iD_hG,\tModop(s)h)
-( iD_hG,\Modop(s)P)+(iD_hG,\Phi_\kappa).
\end{equation}

The main contributions in the right-hand side of~\eqref{D_hHh_s} are estimated as follows (in~\eqref{highorder2} we estimate a contribution coming from the term $(iD_hG,\tModop(s)h)$).

\begin{lemma}\label{DhHy2h}
For $s\in[s_\star,s_1]$, we have 
\begin{equation}\label{highorder1}
|\lambda( iD_hG,y^2h)| \lesssim s^{-1}\|h\|_{\lambda}^2
\end{equation}
and
\begin{equation}\label{highorder2}
\Big|\frac14\Big(b_s-b^2-2b\frac{\lambda_s}{\lambda}\Big)( iD_hG,y^2h)\Big| 
\lesssim s^{-1}\|h\|_{\lambda}^2.
\end{equation}
\end{lemma}

\begin{proof}
We begin with~\eqref{highorder1}. 
Discarding inner products whose real part is zero, we have 
\begin{align*}
( iD_hG,y^2h)
&=-( H_{\gamma\lambda}h+h+\lambda y^2h-\dif f(P)h-R(h),iy^2h) \\
&=-( H_{\gamma\lambda}h-\dif f(P)h-R(h),iy^2h)\\
&=-(\partial_yh,i\partial_y(y^2h))+( \dif f(P)h+R(h),iy^2h).
\end{align*}
First, 
\begin{equation*}
\big|(\partial_yh,i\partial_y(y^2h))\big|= 
\big|2(\partial_yh,iyh)\big|
\le 2s^{1/2}\|h_y\|_{L^2}s^{-1/2}\|yh\|_{L^2}\le s\|h_y\|_{L^2}^2+s^{-1}\|yh\|_{L^2}^2.
\end{equation*}
On the other hand,
\begin{equation*}
|( \dif f(P)h,y^2h)| \lesssim \|h\|_{L^2}^2
\end{equation*}
and 
\begin{equation*}
|( R(h),y^2h)|\lesssim \int_{\mathcal G}(|P|^3y^2|h|^3+y^2|h|^6)\diff y\lesssim \norm{h}_{H^1}^3+\norm{h}_{H^1}^4\norm{yh}_{L^2}^2.
\end{equation*}
 It follows that
\begin{align*}
|\lambda( iD_hG,y^2h)| 
&\lesssim
\lambda\left(s\|h_y\|_{L^2}^2+s^{-1}\|yh\|_{L^2}^2+\|h\|_{L^2}^2+\|h\|_{H^1}^3+\norm{h}_{H^1}^4\norm{yh}_{L^2}^2\right) \\
&\lesssim
s^{-1}\left(\|h_y\|_{L^2}^2+\lambda\|yh\|_{L^2}^2+\|h\|_{L^2}^2\right)+s^{-2}\|h\|_{H^1}^3 \\
&\lesssim
s^{-1}\left(\|h\|_{H^1}^2+\lambda\|yh\|_{L^2}^2\right),
\end{align*}
which concludes the proof of~\eqref{highorder1}. The estimate~\eqref{highorder2} follows by the same
arguments using
$$
b_s-b^2-2b\frac{\lambda_s}{\lambda} = O(\lambda). 
$$
\end{proof}

We now estimate the other contributions in~\eqref{D_hHh_s}.

\begin{lemma}\label{DhHModh}
For $s\in[s_\star,s_1]$,
\begin{equation*}
|( iD_hG,\tModop(s)h)| \lesssim 
s^{-1}\|h\|_{\lambda}^2.
\end{equation*}
\end{lemma}

\begin{proof}
We have 
\begin{equation*}
( iD_hG,\tModop(s)h) 
=(1-\theta_s)( iD_hG,h)
+\frac14\Big(b_s-b^2-2b\frac{\lambda_s}{\lambda}\Big)( iD_hG,y^2h)
-\Big(b+\frac{\lambda_s}{\lambda}\Big)( iD_hG,i\Lambda h).
\end{equation*}
Note that the second term has already been dealt with 
in Lemma~\ref{DhHy2h}. Recall that $h\in \Sigma^2$, hence $\Lambda h\in L^2$.
Now,
\begin{align*}
( iD_hG,h) = - ( D_hG,ih)
&= -( H_{\gamma\lambda}h+h+\lambda y^2h - \dif f(P)h-R(h),ih) \\
&=( \dif f(P)h+R(h),ih)
\end{align*}
so, by~\eqref{mod_est},
\begin{equation}\label{est1}
|(1-\theta_s)( iD_hG,h)| \lesssim s^{-\kappa}\|h\|_{H^1}^2.
\end{equation}
Next,
\begin{align}\label{eq:DGLh}
( iD_hG,i\Lambda h) 
&= ( D_hG,\Lambda h) 
= ( H_{\gamma\lambda}h+h+\lambda y^2h -[f(P+h)-f(P)],\Lambda h).  
\end{align}
Integrating by parts, we find that
\begin{align*}
    (H_{\gamma\lambda}h,\Lambda h)=
\Big(H_{\gamma\lambda}h,\frac h2+yh_y\Big)&=
\frac12(H_{\gamma\lambda}h,h)
-
\re\int_\mathcal{G}h_{yy}y\overline{h}_y\diff y \\
&
=\frac12\norm{h_y}_{L^2}^2+\frac{\gamma\lambda}{2}|h(0)|^2
+ \frac12\|h_y\|_{L^2}^2.
\end{align*}
Further integrations by parts give
$$
\re\int_\mathcal{G}h\overline{\Lambda h}\diff y=0
,\quad
\re\int_\mathcal{G} y^2 h \overline{\Lambda h}\diff y = -\int_\mathcal{G}y^2|h|^2\diff y.
$$
Hence, 
\begin{equation} \label{eq:DG-Lambda-h}
\big| ( H_{\gamma\lambda}h+h+\lambda y^2h ,ih) \big| \lesssim \norm h_\lambda^2.
\end{equation}
Next, integrating by parts,
\begin{eqnarray*}
\lefteqn{( f(P+h)-f(P),\Lambda h)}\\
&&=\frac12\re\int_{\mathcal G} \big(f(P+h)-f(P)\big)\bar h\diff y + \re\int_{\mathcal G} \big(f(P+h)-f(P)\big)y\overline{h_y}\diff y\\
&&=-\frac12\re\int_{\mathcal G} \big(f(P+h)-f(P)\big)\bar h\diff y
-\re\int_{\mathcal G} y \partial_y \big(f(P+h)-f(P)\big) \overline{h}\diff y.
\end{eqnarray*}
For the second term, we have
\begin{align*}
\int_{\mathcal G} y \partial_y \big(f(P+h)-f(P)\big) \overline{h}\diff y
&=\int_{\mathcal G} y\big(\dif f(P+h)(P_y+h_y)-\dif f(P)P_y\big)\overline{h}\diff y \\
&=\int_{\mathcal G} y\big(\dif f(P+h)-\dif f(P)\big)P_y\overline{h}\diff y + \int_{\mathcal G} y\dif f(P+h)h_y\bar h\diff y.
\end{align*}
We estimate 
\begin{equation*}
\Big| \int_{\mathcal G} \big(f(P+h)-f(P)\big)\bar h \diff y\Big| 
\le \|f(P+h)-f(P)\|_{L^2} \|h\|_{L^2}
\lesssim \|h\|_{H^1}^2
\end{equation*}
and
\begin{eqnarray*}
\lefteqn{\Big|\int_{\mathcal G} y\big(f(P+h)-f(P)\big)_y\overline{h}\diff y \Big|}\\
&&\le \|yP_y\|_{L^\infty}\|\dif f(P+h)-\dif f(P)\|_{L^2} \|h\|_{L^2} + \|\dif f(P+h)\|_{L^\infty}\|h_y\|_{L^2} \|yh\|_{L^2} \\
&&\lesssim \|h\|_{H^1}^2+\|yh\|_{L^2}^2.
\end{eqnarray*}
We conclude that
\begin{equation*}
|( f(P+h)-f(P),\Lambda h)| \lesssim s^2\|h\|_{\lambda}^2.
\end{equation*}
With~\eqref{eq:DGLh},~\eqref{eq:DG-Lambda-h} and~\eqref{mod_est} we get
\begin{equation}\label{est3}
\Big|\Big(b+\frac{\lambda_s}{\lambda}\Big)( iD_hG,i\Lambda h)\Big|
\lesssim s^{-(\kappa-2)}\|h\|_{\lambda}^2,
\end{equation}
and the result follows from~\eqref{highorder2},~\eqref{est1} and~\eqref{est3}, 
provided $\kappa \geq 3$.
\end{proof}

\begin{lemma}\label{DhHModQ}
For $s\in[s_\star,s_1]$,
\begin{equation*}
|( iD_hG,\Modop(s)P)| 
\lesssim s^{-2\kappa}.
\end{equation*}
\end{lemma}

\begin{proof}
By Corollary~\ref{cor:P-Q} and~\eqref{mod_est}, we have 
\begin{equation}\label{eq:QP_on_the_right}
|( iD_hG,\Modop(s)(P-Q))| \lesssim \norm{h}_{H^1} \norm{\Modop(s)(P-Q)}_{H^1} \lesssim s^{-2\kappa}. 
\end{equation}
Furthermore,
\begin{equation*}
( iD_hG,\Modop(s)Q)
=(1-\theta_s)( iD_hG,Q)
+\frac14\Big(b_s-b^2-2b\frac{\lambda_s}{\lambda}-\alpha\Big)( iD_hG,y^2Q)
-\Big(b+\frac{\lambda_s}{\lambda}\Big)( iD_hG,i\Lambda Q).
\end{equation*}
Let us recall here the notation $h = h_1 + ih_2$. 
By Corollary~\ref{cor:P-Q} and~\eqref{bootstrap_hyp}, we have 
\begin{align*}
( iD_hG,Q) 
& = - ( D_hG,iQ) = -( H_{\gamma\lambda}h+h-\dif f(P)h+\lambda y^2h-R(h),iQ) \\
&= -( H_{\gamma\lambda}h+h-\dif f(Q)h+\lambda y^2h-R(h),iQ) 
+O(s^{-2}\norm{h}_{L^2})
\\
&=-( h_2,\Lop_-Q)-\gamma\lambda \im h(0)Q(0) - \lambda \im\int_{\mathcal G} y^2h Q\diff y
 +O(s^{-\kappa}).%
\end{align*} 
Since $\Lop_-Q=0$, it follows by~\eqref{mod_est},~\eqref{eq:dec-b-lambda} and~\eqref{bootstrap_hyp} that
\begin{align}\label{est1'}
|(1-\theta_s)( iD_hG,Q)| \lesssim 
s^{-\kappa}\big(s^{-2}\|h\|_{H^1} + s^{-\kappa} \big)
\lesssim s^{-2\kappa}.
\end{align}
Next, by Lemma~\ref{Lpm.lem},
\begin{align*}
( iD_hG,y^2Q) = - ( D_hG,iy^2Q)
&= -( H_{\gamma\lambda}h+h-\dif f(P)h+\lambda y^2h
-R(h),iy^2Q) \\
&=-( h_2,\Lop_-y^2Q) - \lambda \im\int_{\mathcal G} y^2h y^2Q\diff y
+O(s^{-\kappa})\\
&=4( h_2,\Lambda Q)- O(s^{-2} \norm h_{L^2})
+O(s^{-\kappa}) .
\end{align*}
Since $( h_2,\Lambda Q)=O(s^{-2}\|h\|_{L^2})$ by Proposition~\ref{prop:modulation} and Corollary~\ref{cor:P-Q}, 
it follows that
\begin{align}\label{est2'}
\Big|\frac14\Big(b_s-b^2-2b\frac{\lambda_s}{\lambda}-\alpha\Big)( iD_hG,y^2Q)\Big|
&%
\lesssim s^{-2\kappa}.
\end{align}
Finally, by~\eqref{bootstrap_hyp} and~\eqref{almost_orth},
\begin{align*}
( iD_hG,i\Lambda Q) 
&= ( D_hG,\Lambda Q) \\
&= ( H_{\gamma\lambda}h+h-\dif f(P)h+\lambda y^2h -R(h),\Lambda Q)  \\
&=( h_1, L_+\Lambda Q)
+\gamma\lambda \re h(0)\Lambda Q(0) 
+\lambda \re\int_{\mathcal G} y^2h \Lambda Q\diff y
+ O(s^{-\kappa})\\
&=-2(h,Q)_{L^2}+ O(s^{-\kappa}).
\end{align*}
Using~\eqref{almost_orth}, this implies
\begin{align}\label{est3'}
\Big|\Big(b+\frac{\lambda_s}{\lambda}\Big)( iD_hG,i\Lambda Q)\Big|
\lesssim  s^{-2\kappa}.
\end{align}
The result now follows by combining estimates \eqref{eq:QP_on_the_right},
\eqref{est1'},~\eqref{est2'} and~\eqref{est3'}.
\end{proof}

\begin{lemma}\label{DhHffi}
For $s\in[s_\star,s_1]$,
\begin{equation*}
|( iD_hG,\Phi_\kappa)| \lesssim s^{-(2\kappa-1)}.
\end{equation*}
\end{lemma}

\begin{proof}
We have
\begin{equation*}
( iD_hG,\Phi_\kappa)=-( D_hG,i\Phi_\kappa)
=-( H_{\gamma\lambda}h+h+\lambda y^2h
-\dif f(P)h-R(h),i\Phi_\kappa).
\end{equation*}
Hence, by~\eqref{bootstrap_hyp} and~\eqref{phiK_est},
\begin{align*}
|( iD_hG,\Phi_\kappa)|
&\le \|h\|_{H^1}\|\Phi_\kappa\|_{H^1}+\lambda\int_{\mathcal G} y^2|h||\Phi_\kappa|\diff y
+|(\dif f(P)h,\Phi_\kappa)|+|( R(h),\Phi_\kappa)| \\
&\lesssim \|h\|_{H^1}\|\Phi_\kappa\|_{H^1}+
s^{-2} \|yh\|_{L^2}\|y\Phi_\kappa\|_{L^2}\\
& \lesssim s^{-(2\kappa-1)}.
\end{align*}
\end{proof}

\begin{proof}[Proof of Proposition~\ref{dhds_below.prop}]
By~\eqref{D_hHh_s}, Lemmas~\ref{DhHy2h} to~\ref{DhHffi} and the fact that $s^{-1}=O(b)$, 
there exists $k_2>0$ such that, for all $s\in[s_\star,s_1]$,
\begin{equation*}
( D_hG(s,h(s)),h_s) \ge -b\,k_2\|h\|_{\lambda}^2+O(s^{-(2\kappa-1)}).
\end{equation*}
Together with Lemma~\ref{dsH.lem}, this gives Proposition~\ref{dhds_below.prop}.
\end{proof}

\begin{proposition}\label{monotonicity.prop}
Let $k_0$ and $k$ be as in Propositions~\ref{coercivity.prop} and~\ref{dhds_below.prop}, respectively. 
Choose $m\in\N$ such that $m> 2k/k_0$. Then, for all $s\in[s_\star,s_1]$,
\begin{equation*}
\frac{\dif S}{\dif s} \gtrsim \frac{b}{\lambda^m}\left(\|h\|_{\lambda}^2 
+ O(s^{-2(\kappa-1)})\right).
\end{equation*}
\end{proposition}

\begin{proof}
We have 
\begin{equation*}
\frac{\dif S}{\dif s}
=\frac{1}{\lambda^m}\left(-m\frac{\lambda_s}{\lambda}G+\frac{\dif G}{\dif s}\right).
\end{equation*}
Furthermore, for $s$ large enough, $-\lambda_s/\lambda\ge b/2$.
Hence, in view of Propositions~\ref{coercivity.prop} and~\ref{dhds_below.prop},
\begin{align*}
\frac{\dif S}{\dif s}
&\ge\frac{b}{\lambda^m}\left[\frac{m}{2}\left(k_0\|h\|_{\lambda}^2+ O(s^{-2(\kappa-1)})\right)
+\left(-k\|h\|_{\lambda}^2 + O(s^{-2(\kappa-1)})\right)\right] \\
&\gtrsim \frac{b}{\lambda^m}\left(\|h\|_{\lambda}^2 + O(s^{-2(\kappa-1)})\right),
\end{align*}
as claimed.
\end{proof}

\subsection{Uniform estimates in rescaled time}\label{sec:end_proof_sstar}

This section is entirely devoted to the proof of the uniform estimates in $s$. 
\begin{proof}[Proof of Proposition~\ref{prop:sstar}]
We will prove that estimates~\eqref{bootstrap_hyp} can be improved to
\eqref{improved_estimates_h} and~\eqref{improved_estimates_bl} 
on $[s_\star,s_1]$. Then, choosing $s_0$ large enough,
it follows by continuity that, in fact, $s_\star=s_0$, so that 
\eqref{improved_estimates_h} and~\eqref{improved_estimates_bl} hold on $[s_0,s_1]$.

We first prove~\eqref{improved_estimates_h}. By Proposition~\ref{coercivity.prop}
there exists a constant $a>1$ such that, for $s \in[s_\star,s_1]$
\begin{equation}\label{boundsforS}
\frac{1}{a}\frac{1}{\lambda^m}
\big( \|h\|_{\lambda}^2-a^2s^{-2(\kappa-1)}\big)
\le S(s,h)
\le \frac{a}{\lambda^m}\|h\|_{\lambda}^2.
\end{equation}
Choosing $a$ large enough, we also have by Proposition~\ref{monotonicity.prop}
\begin{equation}\label{posder}
\frac{\dif S}{\dif s} \ge 
\frac{1}{a}\frac{b}{\lambda^m}
\big( \|h\|_{\lambda}^2-a^2s^{-2(\kappa-1)}\big).
\end{equation}

Let 
$$
s_\dagger:=\inf\{s\in[s_\star,s_1] : 
\|h(\sigma)\|_{\lambda(\sigma)}\le 2a^2\sigma^{-(\kappa-1)} \ 
\forall\sigma\in[s_\dagger,s_1]\}.
$$
Since $h(s_1)=0$, it follows by continuity of $h$ in $\Sigma^1$ that $s_\dagger\in[s_\star,s_1)$. We prove that
$s_\dagger=s_\star$. 

Suppose by contradiction that
$s_\dagger>s_\star$. Then, in particular,
$\|h(s_\dagger)\|_{\lambda(s_\dagger)}
=2a^2s_\dagger^{-(\kappa-1)}$. Defining
$$
s_\ddagger:=\sup\{s\in[s_\dagger,s_1] : 
\|h(\sigma)\|_{\lambda(\sigma)}\ge a\sigma^{-(\kappa-1)} \ 
\forall\sigma\in[s_\dagger,s]\},
$$
we have $s_\star<s_\dagger<s_\ddagger<s_1$ and
$\|h(s_\ddagger)\|_{\lambda(s_\ddagger)}
=as_\ddagger^{-(\kappa-1)}$. Furthermore, by~\eqref{posder},
$S$ is non-decreasing on $[s_\dagger,s_\ddagger]$.
Hence, using~\eqref{boundsforS} and our
bootstrap assumption on $\lambda$ in~\eqref{bootstrap_hyp},
we find that
\begin{align*}
\|h(s_\dagger)\|_{\lambda(s_\dagger)}^2-a^2s_\dagger^{-2(\kappa-1)}
&\le a\lambda^m(s_\dagger)S(s_\dagger,h(s_\dagger))
\le a\lambda^m(s_\dagger)S(s_\ddagger,h(s_\ddagger)) \\
&\le a^2\frac{\lambda^m(s_\dagger)}{\lambda^m(s_\ddagger)}
    \|h(s_\ddagger)\|_{\lambda(s_\ddagger)}^2
=a^4\frac{\lambda^m(s_\dagger)}{\lambda^m(s_\ddagger)}
    s_\ddagger^{-2(\kappa-1)} \\
&\le 2a^4
\Big(\frac{s_\ddagger}{s_\dagger}\Big)^{2m}s_\ddagger^{-2(\kappa-1)} \le 
2a^4
\Big(\frac{s_\ddagger}{s_\dagger}\Big)^{2m}\Big(\frac{s_\ddagger}{s_\dagger}\Big)^{-2(\kappa-1)}s_\dagger^{-2(\kappa-1)} \le
2a^4s_\dagger^{-2(\kappa-1)},
\end{align*}
where the last inequality holds since we may choose $\kappa$ so large that $2m-2(\kappa-1)<0$.
It follows that 
$$
\|h(s_\dagger)\|_{\lambda(s_\dagger)}^2
\le a^2s_\dagger^{-2(\kappa-1)}+2a^4s_\dagger^{-2(\kappa-1)}
\le 3a^4s_\dagger^{-2(\kappa-1)},
$$
a contradiction.

We now prove~\eqref{improved_estimates_bl}.
Let $E^\star\in\R$, $\mathcal E^\star=C_Q^{-1}E^\star$ and
$b_1, \lambda_1$ be given by Proposition~\ref{prop:dynamic}. 
It follows from~\eqref{energy_nls_dynsys} that 
\begin{equation}\label{final_energy}
|E(\tilde P(b_1,\lambda_1,\theta_1))- E^\star|
\lesssim \frac{(b_1^2+\lambda_1)}{\lambda_1^2}^\kappa
\lesssim s_1^{4-2\kappa}.
\end{equation}
Now, the energy estimate~\eqref{deriv_energy} and the modulation estimate~\eqref{mod_est}
yield 
\begin{align}\label{integr_energy}
|E(\tilde P(b(s),\lambda(s),\theta(s)))-E(\tilde P(b_1,\lambda_1,\theta_1))|
&=\Big|\int_s^{s_1}\frac{\dif}{\dif\sigma}E(\tilde P(b(\sigma),\lambda(\sigma),\theta(\sigma)))\diff\sigma\Big| \notag \\
&\lesssim \int_s^{s_1}\sigma^{4-\kappa}\diff\sigma\lesssim s^{5-\kappa}, \quad s\in[s_\star,s_1].
\end{align}
It then follows by~\eqref{final_energy} and~\eqref{integr_energy} that
\begin{equation}\label{s_est_nrj}
|E(\tilde P(b(s),\lambda(s),\theta(s)))-E^\star|\lesssim s^{5-\kappa}, \quad s\in[s_\star,s_1].
\end{equation}
Next, using~\eqref{energy_nls_dynsys} at time $s$, we deduce that 
\begin{align}\label{variation_energy}
|\mathcal E(b(s),\lambda(s))-\mathcal E^\star|
&\lesssim
\Big| \mathcal E(b,\lambda)- \frac{E(\tilde P(b,\lambda,\theta))}{C_Q}\Big|
+\Big|\frac{E(\tilde P(b,\lambda,\theta))}{C_Q}-\mathcal E^\star \Big| \notag \\
&\lesssim
s^{4-2\kappa}+s^{5-\kappa}\lesssim s^{5-\kappa}, \quad s\in[s_\star,s_1],
\end{align}
and the formula~\eqref{asympt_nrj} defining $\mathcal E$ yields
$$
\lambda^2\mathcal E(b,\lambda) = b^2-2\alpha_\star\lambda+O(\lambda^2).
$$
For $\kappa\ge 9$, \eqref{variation_energy} implies
\begin{equation}\label{aaa}
|b^2-2\alpha_\star\lambda-\mathcal E^\star\lambda^2|\lesssim \lambda^2(s) + s^{5-\kappa} \lesssim s^{-4}
\end{equation}
and it follows that
$$
\big|b-\sqrt{2\alpha_\star\lambda+\mathcal E^\star\lambda^2}\big|
\big|b+\sqrt{2\alpha_\star\lambda+\mathcal E^\star\lambda^2}\big|
\lesssim s^{-4}.
$$
Hence, by~\eqref{bootstrap_hyp},
\begin{equation}\label{bbb}
\big|b-\sqrt{2\alpha_\star\lambda+\mathcal E^\star\lambda^2}\big|
\lesssim s^{-3}, \quad s\in[s_\star,s_1].
\end{equation}
From~\eqref{mod_est}, we have $b=-\lambda_s/\lambda+O(s^{-\kappa})$, hence
$$
\Big|\frac{\lambda_s}{\lambda}+\sqrt{2\alpha_\star\lambda+\mathcal E^\star\lambda^2}\Big|
\lesssim s^{-3}
$$
and we deduce from the definition of $\mathcal F$ in~\eqref{defofF} that
$$
\Big| \frac{\dif}{\dif s}\mathcal F(\lambda(s))-1 \Big|
=\Big| \frac{\lambda_s}{\lambda\sqrt{2\alpha_\star\lambda+\mathcal E^\star\lambda^2}}+1 \Big|
=\frac{1}{\sqrt{2\alpha_\star\lambda+\mathcal E^\star\lambda^2}}
\Big| \frac{\lambda_s}{\lambda}+ \sqrt{2\alpha_\star\lambda+\mathcal E^\star\lambda^2}\Big|
\lesssim s^{-2}.
$$
Integrating from $s$ to $s_1$ and using $\mathcal F(\lambda(s_1))=s_1$ yields
$$
\Big| \mathcal F(\lambda(s))-s \Big|
\le \Big| \int_s^{s_1}\Big(\frac{\dif}{\dif \sigma}\mathcal F(\lambda(\sigma))-1\Big)\diff\sigma\Big|
\lesssim s^{-1}
$$
$$
\implies \mathcal F(\lambda(s))=s+O(s^{-1}), \quad s\in[s_\star,s_1].
$$
On the other hand, it follows from~\eqref{Fl_est} that
$$
\Big|\frac{\lambda(s)^{1/2}}{\lambda_\mathrm{mo}(s)^{1/2}}-1\Big|\lesssim \frac{1}{s^2},
\quad s\in[s_\star,s_1].
$$
Finally, returning to~\eqref{bbb} and using again~\eqref{bootstrap_hyp},
\begin{align*}
b-b_\mathrm{mo}
&=\sqrt{2\alpha_\star\lambda+\mathcal E^\star\lambda^2}-\sqrt{2\alpha_\star\lambda_\mathrm{mo}}+O(s^{-3})
=\frac{2\alpha_\star\lambda+\mathcal E^\star\lambda^2-2\alpha_\star\lambda_\mathrm{mo}}
{\sqrt{2\alpha_\star\lambda+\mathcal E^\star\lambda^2}+\sqrt{2\alpha_\star\lambda_\mathrm{mo}}}+O(s^{-3}) \\
&=O(b_\mathrm{mo})[\mathcal E^\star\lambda^2
+2\alpha_\star(\lambda^{1/2}-\lambda_\mathrm{mo}^{1/2})(\lambda^{1/2}+\lambda_\mathrm{mo}^{1/2})]
+O(s^{-3})
\end{align*}
$$
\implies \Big|\frac{b(s)}{b_\mathrm{mo}(s)}-1\Big|\lesssim \frac{1}{s^2},
\quad s\in[s_\star,s_1].
$$
This concludes the proof. 
\end{proof}

\begin{remark}\label{rem:improved_est_3}
We observe here that estimates~\eqref{improved_estimates_bl} can be improved by a closer inspection
of the energy expansion $\mathcal E$. 
Indeed,~\eqref{asympt_nrj} yields
$$
\lambda^2\mathcal E(b,\lambda) = b^2-2\alpha_\star\lambda+e_0\lambda^2+O(s^{-6}),
\quad e_0:=\eps_{0,1}+2\alpha_\star\eps_{2,0}.
$$
Hence, using~\eqref{variation_energy} and choosing $\kappa\ge 11$, estimate~\eqref{aaa} improves to
\begin{equation*}
|b^2-2\alpha_\star\lambda-(\mathcal E^\star-e_0)\lambda^2|\lesssim s^{-6}+ s^{5-\kappa}\lesssim s^{-6}.
\end{equation*}
Then, replacing $\mathcal E^\star$ by $\mathcal E^\star-e_0$ in Proposition~\ref{prop:dynamic}
and using this improved estimate in the rest of the proof yields~\eqref{better_improved_estimates_bl}.
\end{remark}

\appendix

\section{Dynamical systems}\label{sec:dynamical_systems}

In this appendix, we prove Proposition \ref{prop:dynamic}.
Given $s_1 \geq 1$, we consider the system
\begin{equation}
  \label{eq:sys-in-s}
  \begin{cases}
         b+\disp\frac{ \lambda_s}{ \lambda}=0,\\
     b_s+ b^2-\alpha_\star \lambda=0,
  \end{cases}
\end{equation}
with initial data
\begin{equation}
    \label{eq:init-cond-s_0}
 b(s_1)=b_1,\quad  \lambda(s_1)=\lambda_1. 
\end{equation}
This is a Hamiltonian system, with conserved energy
\[
\mathcal E_\mathrm{mo} (b,\lambda)=\left(\frac{b}{\lambda}\right)^2-\frac{2\alpha_\star}{\lambda}.
\]
An exact solution with energy $\mathcal E_\mathrm{mo} = 0$ is given by
\begin{equation}\label{model_solution_2}
b_\mathrm{mo}(s)=\frac2s, \quad \lambda_\mathrm{mo}(s)=\frac{2}{\alpha_\star s^2}.
\end{equation}

\begin{lemma} \label{lem:universal}
Let $b_1,\lambda_1>0$. The solution of the 
Cauchy problem~\eqref{eq:sys-in-s}-\eqref{eq:init-cond-s_0} is given by
  \begin{align*}
   b(s)
   &=\left(\alpha_\star (s-s_1)+\frac{b_1}{\lambda_1}\right) 
   \left(\frac12 \alpha_\star (s-s_1)^2+\frac{b_1}{\lambda_1}(s-s_1)+\frac{1}{\lambda_1}\right)^{-1}=b_\mathrm{mo}(s)+O(s^{-2}),\\
   \lambda (s)&= \left(\frac12 \alpha_\star (s-s_1)^2+\frac{b_1}{\lambda_1}(s-s_1)+\frac{1}{\lambda_1}\right)^{-1} = \lambda_\mathrm{mo}(s)+O(s^{-4}).
  \end{align*}
\end{lemma}

\begin{proof}
Defining the auxiliary unkown $\mu$ by
\[
 \mu=\frac{1}{ \lambda},
\]
direct calculations using~\eqref{eq:sys-in-s} yield
\[
 \mu_s=-\frac{ \lambda_s/ \lambda}{ \lambda}=\frac{ b}{ \lambda},\quad  
 \mu_{ss}=\left(\frac{ b}{ \lambda}\right)_s
 =\frac{ b_s- b \lambda_s/ \lambda}{ \lambda}=\frac{ b_s+ b^2}{ \lambda}=\alpha_\star.
\]
Integrating in $s$, we get
\[
 \mu_s(s)=\alpha_\star (s-s_1)+\frac{b_1}{\lambda_1},\quad  
 \mu(s)=\frac12 \alpha_\star (s-s_1)^2+\frac{b_1}{\lambda_1}(s-s_1)+\frac{1}{\lambda_1},
\]
and the result follows.
\end{proof}

We can now prove Proposition \ref{prop:dynamic}. Let us first explain how the function $\mathcal F$ comes into play. Observe that $\mathcal F$ is strictly decreasing on $(0,\lambda_0]$, vanishes at $\lambda_0$ and, since 
$$
\frac{1}{\mu^{3/2}\sqrt{\mathcal E^\star\mu+2\alpha_\star}} 
\sim \frac{1}{\sqrt{2\alpha_\star}\mu^{3/2}}, \quad \mu\to0^+,
$$
$\mathcal F(\lambda)\to +\infty$ as $\lambda\to 0^+$. Thus, $\mathcal F$ is a bijection from $(0,\lambda_0]$ to $[0,+\infty)$. Now let $s_0 \geq 0$ and assume that $(b,\lambda) : [s_0,+\infty) \to \R \times \R_+^*$ is a solution of energy $\mathcal E^\star$ of \eqref{eq:sys-in-s} such that $\lambda(s_0) = \lambda_0$ and $\mathcal E^\star \lambda^2+2\alpha_\star\lambda > 0$ for all $s \geq s_0$. Then we have 
\begin{equation}\label{integr_dyn_sys_1}
b=\sqrt{\mathcal E^\star \lambda^2+2\alpha_\star\lambda}, \quad b+\disp\frac{ \lambda_s}{ \lambda}=0,
\end{equation}
which gives, for all $s \geq s_0$,
\begin{equation}\label{integr_dyn_sys_2}
\mathcal F(\lambda(s)) = - \int_{s_0}^{s}\frac{\lambda_\sigma(\sigma)}
{\lambda(\sigma)^{3/2}\sqrt{\mathcal E^\star\lambda(\sigma)+2\alpha_\star}}\diff\sigma=s-s_0.
\end{equation}
Conversely, if $\lambda$ is a solution of \eqref{integr_dyn_sys_2}, then we get a solution of energy $\mathcal E^\star$ by returning to the first equation in~\eqref{integr_dyn_sys_1}.

We use $\mathcal F$ to construct the final data $(b_1,\lambda_1)$ used in \eqref{eq:definition_of_u_1} (notice that the definitions of $\lambda_1$ and then $b_1$ will depend on the choice of $\lambda_0$ in the definition of $\mathcal F$, but this does not alter the rates of decay in the estimates).
 
\begin{proof}[Proof of Proposition~\ref{prop:dynamic}]
For any $s_1>0$, there exists a unique $\lambda_1>0$ such that 
$\mathcal F(\lambda_1)=s_1$. For the model system,~\eqref{model_solution_2} yields
\begin{equation*}
s_1=\sqrt{\frac{2}{\alpha_\star \lambda_\mathrm{mo}(s_1)}}.
\end{equation*}
On the other hand, for any $\lambda \in (0,\lambda_0]$, we have
\begin{align}\label{Fl_est}
\left|\mathcal F(\lambda)- \sqrt{\frac{2}{\alpha_\star \lambda}} \right| 
&= \left| \int_\lambda^{\lambda_0}\frac{\diff\mu}{\mu^{3/2}\sqrt{\mathcal E^\star\mu+2\alpha_\star}}
	-\sqrt{\frac{2}{\alpha_\star}}\lambda^{-1/2} \right|\notag \\
&\le  \left| \int_\lambda^{\lambda_0}\frac{\diff\mu}{\mu^{3/2}\sqrt{\mathcal E^\star\mu+2\alpha_\star}}
	-\sqrt{\frac{2}{\alpha_\star}}\Big(\lambda^{-1/2}-\lambda_0^{-1/2}\Big) \right| 
	 + \sqrt{\frac{2}{\alpha_\star}}\lambda_0^{-1/2}\notag\\
&\le  \left| \int_\lambda^{\lambda_0}\frac{1}{\mu^{3/2}}
	\Big(\frac{1}{\sqrt{\mathcal E^\star\mu+2\alpha_\star}}
	-\frac{1}{\sqrt{2\alpha_\star}}\Big)\diff\mu \right|  + \sqrt{\frac{2}{\alpha_\star}}\lambda_0^{-1/2}\notag \\
&\lesssim \int_\lambda^{\lambda_0}\frac{\diff\mu}{\mu^{1/2}}+1=O_{\lambda\to 0}(1).
\end{align}
Hence, with $\lambda=\lambda_1$, we get
$$
\left|\sqrt{\frac{2}{\alpha_\star \lambda_\mathrm{mo}(s_1)}} - \sqrt{\frac{2}{\alpha_\star \lambda_1}} \right|=\left| s_1- \sqrt{\frac{2}{\alpha_\star \lambda_1}} \right|=O(1).
$$
After some algebra, this yields
\begin{equation*}\label{l1est}
\Big|\frac{\lambda_1^{1/2}}{\lambda_\mathrm{mo}(s_1)^{1/2}}-1\Big|
\lesssim \lambda_\mathrm{mo}(s_1)^{1/2} \lesssim \frac{1}{s_1}.
\end{equation*}
This also gives 
\begin{equation*}
\big|\lambda_1^{1/2}-\lambda_\mathrm{mo}(s_1)^{1/2}\big|
\lesssim \lambda_\mathrm{mo}(s_1) \lesssim \frac{1}{s_1^2}
\end{equation*}
or 
\begin{equation*}
\big|\lambda_1-\lambda_\mathrm{mo}(s_1)\big| \lesssim \frac{1}{s_1^3}.
\end{equation*}

To find $b_1$, we set
\[
g(b)=\lambda_1^2\mathcal E(b,\lambda_1) = b^2 - 2\lambda_1 \alpha_\star + \sum_{\substack{(j,k) \in \Theta_\kappa\\ j \text{ even}, \ j/2+k\ge2}}
\eps_{j,k} b^j\lambda_1^{k}
\]
(see \eqref{asympt_nrj}).
We then seek a solution of $g(b) = \lambda_1^2\mathcal E^\star$ close to $b_\mathrm{mo}(s_1)=2/s_1$. We have 
\[
g(b_\mathrm{mo}(s_1)) = 2\alpha_\star ((\lambda_\mathrm{mo}(s_1)-\lambda_1) + O(s_1^{-4}) = O(s_1^{-3}).
\]
Furthermore, for $b \in \big[b_\mathrm{mo}(s_1)- \frac 1 {s_1^2},b_\mathrm{mo}(s_1) + \frac 1 {s_1^2} \big]$ we have, if $s_1$ is large enough, 
\begin{align*}
g'(b) = 2b + O(s_1^{-3}) \geq \frac 1 {s_1}.
\end{align*}
Since $\lambda_1^2\mathcal E^\star=O(s_1^{-4})$, if $s_1$ is large enough then
there exists a unique $b_1>0$ such that
\begin{equation*}
g(b_1)=\lambda_1^2\mathcal E^\star \quad\text{and}\quad
|b_1-b_\mathrm{mo}(s_1)| \leq \frac 1 {s_1^2}.
\end{equation*}
It follows that
\begin{equation*}
\Big|\frac{b_1}{b_\mathrm{mo}(s_1)}-1\Big|\leq \frac{1}{2s_1},
\end{equation*}
which finishes the proof.
\end{proof}

\bibliographystyle{abbrv}
\bibliography{master}

\ifx \undefined \booktitle \def \booktitle#1{{{\em #1}}} \fi\ifx \cftil
  \undefined \def \cftil#1{\~#1} \fi\ifx \undefined \cprime \def \cprime {$'$}
  \fi\ifx \undefined \flqq \def \flqq {\ifmmode \ll \else \leavevmode \raise
  0.2ex \hbox{$\scriptscriptstyle \ll $}\fi}\fi\ifx \undefined \frqq \def \frqq
  {\ifmmode \gg \else \leavevmode \raise 0.2ex \hbox{$\scriptscriptstyle \gg
  $}\fi}\fi\ifx \undefined \k \let \k = \c \fi\ifx \undefined \mathbb \def
  \mathbb #1{{\bf #1}}\fi\ifx \undefined \mathbf \def \mathbf #1{{\bf
  #1}}\fi\ifx \undefined \mathrm \def \mathrm #1{{\rm #1}}\fi\ifx \undefined
  \pkg \def \pkg #1{{{\tt #1}}} \fi\ifx \undefined \scr \let \scr = \cal
  \fi\def\cprime{$'$} \def\cprime{$'$}
\begin{thebibliography}{10}

\bibitem{AdCaFiNo11}
R.~Adami, C.~Cacciapuoti, D.~Finco, and D.~Noja.
\newblock Fast solitons on star graphs.
\newblock {\em Rev. Math. Phys.}, 23(4):409--451, 2011.

\bibitem{AdCaFiNo12}
R.~Adami, C.~Cacciapuoti, D.~Finco, and D.~Noja.
\newblock On the structure of critical energy levels for the cubic focusing
  {NLS} on star graphs.
\newblock {\em J. Phys. A}, 45(19):192001, 7, 2012.

\bibitem{adami2012stationary}
R.~Adami, C.~Cacciapuoti, D.~Finco, and D.~Noja.
\newblock {Stationary states of NLS on star graphs}.
\newblock {\em EPL (Europhysics Letters)}, 100(1):10003, 2012.

\bibitem{AdCaFiNo14a}
R.~Adami, C.~Cacciapuoti, D.~Finco, and D.~Noja.
\newblock Constrained energy minimization and orbital stability for the {NLS}
  equation on a star graph.
\newblock {\em Ann. Inst. H. Poincar\'e Anal. Non Lin\'eaire},
  31(6):1289--1310, 2014.

\bibitem{AdCaFiNo14}
R.~Adami, C.~Cacciapuoti, D.~Finco, and D.~Noja.
\newblock Variational properties and orbital stability of standing waves for
  {NLS} equation on a star graph.
\newblock {\em J. Differential Equations}, 257(10):3738--3777, 2014.

\bibitem{AdCaFiNo16}
R.~Adami, C.~Cacciapuoti, D.~Finco, and D.~Noja.
\newblock Stable standing waves for a {NLS} on star graphs as local minimizers
  of the constrained energy.
\newblock {\em J. Differential Equations}, 260(10):7397--7415, 2016.

\bibitem{AdSeTi16}
R.~Adami, E.~Serra, and P.~Tilli.
\newblock Threshold phenomena and existence results for {NLS} ground states on
  metric graphs.
\newblock {\em J. Funct. Anal.}, 271(1):201--223, 2016.

\bibitem{AdSeTi17a}
R.~Adami, E.~Serra, and P.~Tilli.
\newblock Negative energy ground states for the {$L^2$}-critical {NLSE} on
  metric graphs.
\newblock {\em Comm. Math. Phys.}, 352(1):387--406, 2017.

\bibitem{AdSeTi17b}
R.~Adami, E.~Serra, and P.~Tilli.
\newblock Nonlinear dynamics on branched structures and networks.
\newblock {\em Riv. Math. Univ. Parma (N.S.)}, 8(1):109--159, 2017.

\bibitem{AmBcMe21}
K.~Ammari, A.~Bchatnia, and N.~Mehenaoui.
\newblock Exponential stability for the nonlinear {S}chr\"{o}dinger equation on
  a star-shaped network.
\newblock {\em Z. Angew. Math. Phys.}, 72(1):Paper No. 35, 19, 2021.

\bibitem{AnCa19}
J.~Angulo~Pava and M.~Cavalcante.
\newblock {\em Nonlinear Dispersive Equations on Star Graphs}, volume~32 of
  {\em Brazilian Mathematics Colloquium}.
\newblock Instituto Nacional de Matemática Pura e Aplicada, Rio de Janeiro,
  Brasil, 2019.

\bibitem{AnGo18a}
J.~Angulo~Pava and N.~Goloshchapova.
\newblock Extension theory approach in the stability of the standing waves for
  the {NLS} equation with point interactions on a star graph.
\newblock {\em Adv. Differential Equations}, 23(11-12):793--846, 2018.

\bibitem{AnGo18b}
J.~Angulo~Pava and N.~Goloshchapova.
\newblock On the orbital instability of excited states for the {NLS} equation
  with the {$\delta$}-interaction on a star graph.
\newblock {\em Discrete Contin. Dyn. Syst.}, 38(10):5039--5066, 2018.

\bibitem{AoInMi21}
K.~Aoki, T.~Inui, and H.~Mizutani.
\newblock Failure of scattering to standing waves for a {S}chr\"{o}dinger
  equation with long-range nonlinearity on star graph.
\newblock {\em J. Evol. Equ.}, 21(1):297--312, 2021.

\bibitem{Ar17}
A.~H. Ardila.
\newblock Logarithmic {NLS} equation on star graphs: existence and stability of
  standing waves.
\newblock {\em Differential Integral Equations}, 30(9-10):735--762, 2017.

\bibitem{BaCaDu11}
V.~Banica, R.~Carles, and T.~Duyckaerts.
\newblock Minimal blow-up solutions to the mass-critical inhomogeneous {NLS}
  equation.
\newblock {\em Comm. Partial Differential Equations}, 36(3):487--531, 2011.

\bibitem{BaVi16}
V.~Banica and N.~Visciglia.
\newblock Scattering for {NLS} with a delta potential.
\newblock {\em J. Differential Equations}, 260(5):4410--4439, 2016.

\bibitem{Grafidi}
C.~Besse, R.~Duboscq, and S.~Le~Coz.
\newblock Grafidi.
\newblock {\em PLMlab repository},
  \url{https://plmlab.math.cnrs.fr/cbesse/grafidi}, 2021.

\bibitem{BeDuLe22A}
C.~Besse, R.~Duboscq, and S.~Le~Coz.
\newblock Gradient flow approach to the calculation of stationary states on
  nonlinear quantum graphs.
\newblock {\em Ann. Henri Lebesgue}, 5:387--428, 2022.

\bibitem{BeDuLe22B}
C.~Besse, R.~Duboscq, and S.~Le~Coz.
\newblock Numerical simulations on nonlinear quantum graphs with the {GraFiDi}
  library.
\newblock {\em SMAI J. Comput. Math.}, 8:1--47, 2022.

\bibitem{BoWa97}
J.~Bourgain and W.~Wang.
\newblock Construction of blowup solutions for the nonlinear {S}chr\"odinger
  equation with critical nonlinearity.
\newblock {\em Ann. Scuola Norm. Sup. Pisa Cl. Sci. (4)}, 25(1-2):197--215
  (1998), 1997.
\newblock Dedicated to Ennio De Giorgi.

\bibitem{ChGuNaTs07}
S.-M. Chang, S.~Gustafson, K.~Nakanishi, and T.-P. Tsai.
\newblock Spectra of linearized operators for {NLS} solitary waves.
\newblock {\em SIAM J. Math. Anal.}, 39(4):1070--1111, 2007/08.

\bibitem{FuJe08}
R.~Fukuizumi and L.~Jeanjean.
\newblock Stability of standing waves for a nonlinear {S}chr\"odinger equation
  with a repulsive {D}irac delta potential.
\newblock {\em Discrete Contin. Dyn. Syst.}, 21(1):121--136, 2008.

\bibitem{FuOhOz08}
R.~Fukuizumi, M.~Ohta, and T.~Ozawa.
\newblock Nonlinear {S}chr\"odinger equation with a point defect.
\newblock {\em Ann. Inst. H. Poincar\'e Anal. Non Lin\'eaire}, 25(5):837--845,
  2008.

\bibitem{GeMaWe16}
F.~Genoud, B.~A. Malomed, and R.~M. Weish{\"a}upl.
\newblock Stable {NLS} solitons in a cubic-quintic medium with a delta-function
  potential.
\newblock {\em Nonlinear Anal.}, 133:28--50, 2016.

\bibitem{Gol19}
N.~Goloshchapova.
\newblock On the standing waves of the {NLS}-log equation with a point
  interaction on a star graph.
\newblock {\em J. Math. Anal. Appl.}, 473(1):53--70, 2019.

\bibitem{Go22}
N.~Goloshchapova.
\newblock Dynamical and variational properties of the {NLS}-{$\delta'_s$}
  equation on the star graph.
\newblock {\em J. Differential Equations}, 310:1--44, 2022.

\bibitem{GoOh20}
N.~Goloshchapova and M.~Ohta.
\newblock {Blow-up and strong instability of standing waves for the
  NLS-$\delta$ equation on a star graph}.
\newblock {\em Nonlinear Analysis}, 196:111753, 2020.

\bibitem{GuLeTs17}
S.~Gustafson, S.~Le~Coz, and T.-P. Tsai.
\newblock Stability of periodic waves of 1{D} cubic nonlinear {S}chr\"odinger
  equations.
\newblock {\em Appl. Math. Res. Express. AMRX}, 2:431--487, 2017.

\bibitem{IaLeRo17}
I.~Ianni, S.~Le~Coz, and J.~Royer.
\newblock On the {C}auchy problem and the black solitons of a singularly
  perturbed {G}ross-{P}itaevskii equation.
\newblock {\em SIAM J. Math. Anal.}, 49(2):1060--1099, 2017.

\bibitem{Ka19}
A.~Kairzhan.
\newblock Orbital instability of standing waves for {NLS} equation on star
  graphs.
\newblock {\em Proc. Amer. Math. Soc.}, 147(7):2911--2924, 2019.

\bibitem{KaPeGo19}
A.~Kairzhan, D.~E. Pelinovsky, and R.~H. Goodman.
\newblock Drift of spectrally stable shifted states on star graphs.
\newblock {\em SIAM J. Appl. Dyn. Syst.}, 18(4):1723--1755, 2019.

\bibitem{kato}
T.~Kato.
\newblock {\em Perturbation Theory for Linear Operators}.
\newblock Classics in Mathematics. Springer, second edition, 1980.

\bibitem{KrLeRa13}
J.~Krieger, E.~Lenzmann, and P.~Rapha{\"e}l.
\newblock Nondispersive solutions to the {$L^2$}-critical half-wave equation.
\newblock {\em Arch. Ration. Mech. Anal.}, 209(1):61--129, 2013.

\bibitem{KrSc09}
J.~Krieger and W.~Schlag.
\newblock Non-generic blow-up solutions for the critical focusing {NLS} in
  1-{D}.
\newblock {\em J. Eur. Math. Soc. (JEMS)}, 11(1):1--125, 2009.

\bibitem{LeFuFiKsSi08}
S.~Le~Coz, R.~Fukuizumi, G.~Fibich, B.~Ksherim, and Y.~Sivan.
\newblock Instability of bound states of a nonlinear {S}chr\"odinger equation
  with a {D}irac potential.
\newblock {\em Phys. D}, 237(8):1103--1128, 2008.

\bibitem{LeMaRa16}
S.~Le~Coz, Y.~Martel, and P.~Rapha\"el.
\newblock {Minimal mass blow up solutions for a double power nonlinear
  Schr\"{o}dinger equation}.
\newblock {\em Rev. Mat. Iberoam.}, 32(3):795--833, 2016.

\bibitem{LiLiSh18}
Y.~Li, F.~Li, and J.~Shi.
\newblock Ground states of nonlinear {S}chr\"{o}dinger equation on star metric
  graphs.
\newblock {\em J. Math. Anal. Appl.}, 459(2):661--685, 2018.

\bibitem{MaPi17}
Y.~Martel and D.~Pilod.
\newblock Construction of a minimal mass blow up solution of the modified
  {B}enjamin-{O}no equation.
\newblock {\em Math. Ann.}, 369(1-2):153--245, 2017.

\bibitem{Ma20a}
N.~Matsui.
\newblock Remarks on minimal mass blow up solutions for a double power
  nonlinear {S}chr\"odinger equation.
\newblock {\em arXiv preprint arXiv:2012.14562}, 2020.

\bibitem{Ma21c}
N.~Matsui.
\newblock Minimal mass blow-up solutions for double power nonlinear
  {S}chr\"odinger equations with an inverse potential.
\newblock {\em arXiv preprint arXiv:2109.08840}, 2021.

\bibitem{Ma21e}
N.~Matsui.
\newblock Minimal mass blow-up solutions for nonlinear {S}chr\"odinger
  equations with a {H}artree nonlinearity.
\newblock {\em arXiv preprint arXiv:2111.08443}, 2021.

\bibitem{Ma21d}
N.~Matsui.
\newblock Minimal mass blow-up solutions for nonlinear {S}chr\"odinger
  equations with a singular potential.
\newblock {\em arXiv preprint arXiv:2110.12980}, 2021.

\bibitem{Ma21a}
N.~Matsui.
\newblock Minimal-mass blow-up solutions for nonlinear {S}chr\"{o}dinger
  equations with an inverse potential.
\newblock {\em Nonlinear Anal.}, 213:Paper No. 112497, 32, 2021.

\bibitem{Ma21b}
N.~Matsui.
\newblock Minimal-mass blow-up solutions for nonlinear {S}chr\"odinger
  equations with growth potentials.
\newblock {\em arXiv preprint arXiv:2108.06205}, 2021.

\bibitem{Ma23}
N.~Matsui.
\newblock Minimal-mass blow-up solutions for inhomogeneous nonlinear
  {Schr{\"o}dinger} equations with growing potentials.
\newblock {\em Ark. Mat.}, 61(2):413--436, 2023.

\bibitem{Ma20b}
N.~Matsui.
\newblock Minimal mass blow-up solutions for nonlinear {Schr{\"o}dinger}
  equations with a potential.
\newblock {\em T{\^o}hoku Math. J. (2)}, 75(2):215--232, 2023.

\bibitem{Me93}
F.~Merle.
\newblock Determination of blow-up solutions with minimal mass for nonlinear
  {S}chr\"odinger equations with critical power.
\newblock {\em Duke Math. J.}, 69(2):427--454, 1993.

\bibitem{Me96}
F.~Merle.
\newblock Nonexistence of minimal blow-up solutions of equations {$iu_t=-\Delta
  u-k(x)|u|^{4/N}u$} in {${\bf R}^N$}.
\newblock {\em Ann. Inst. H. Poincar\'{e} Phys. Th\'{e}or.}, 64(1):33--85,
  1996.

\bibitem{No14}
D.~Noja.
\newblock Nonlinear {S}chr\"odinger equation on graphs: recent results and open
  problems.
\newblock {\em Philos. Trans. R. Soc. Lond. Ser. A Math. Phys. Eng. Sci.},
  372(2007):20130002, 20, 2014.

\bibitem{RaSz11}
P.~Rapha{\"e}l and J.~Szeftel.
\newblock Existence and uniqueness of minimal blow-up solutions to an
  inhomogeneous mass critical {NLS}.
\newblock {\em J. Amer. Math. Soc.}, 24(2):471--546, 2011.

\bibitem{SoMaSaSaNa10}
Z.~A. Sobirov, D.~U. Matrasulov, K.~K. Sabirov, S.-i. Sawada, and K.~Nakamura.
\newblock Integrable nonlinear schr\"odinger equation on simple networks:
  Connection formula at vertices.
\newblock {\em Phys. Rev. E}, 81:066602, Jun 2010.

\bibitem{TaXu21}
X.~Tang and G.~Xu.
\newblock Minimal mass blow-up solutions for the {$L^2$}-critical {NLS} with
  the delta potential for even data in one dimension.
\newblock {\em SIAM J. Math. Anal.}, 56(2):1727--1769, 2024.

\bibitem{We85}
M.~I. Weinstein.
\newblock Modulational stability of ground states of nonlinear {S}chr\"odinger
  equations.
\newblock {\em SIAM J. Math. Anal.}, 16:472--491, 1985.

\end{thebibliography}

\end{document}